\documentclass{amsart}
\usepackage{graphicx}
\usepackage{amsmath}
\usepackage{amscd}
\usepackage{amsfonts}
\usepackage{amssymb}

\numberwithin{equation}{section}

\newtheorem{theorem}{Theorem}[section]
\newtheorem{lemma}[theorem]{Lemma}
\newtheorem{proposition}[theorem]{Proposition}

\newtheorem{definition}[theorem]{Definition}
\newtheorem{corollary}[theorem]{Corollary}

\newtheorem{example}[theorem]{Example}
\newtheorem{remark}[theorem]{Remark}

\newcommand{\be}{\begin{equation}}
\newcommand{\ee}{\end{equation}}
\newcommand{\bes}{\begin{equation*}}
\newcommand{\ees}{\end{equation*}}

\newcommand{\cE}{\mathcal{E}}
\newcommand{\cH}{\mathcal{H}}
\newcommand{\cK}{\mathcal{K}}
\newcommand{\cL}{\mathcal{L}}
\newcommand{\cM}{\mathcal{M}}
\newcommand{\cN}{\mathcal{N}}
\newcommand{\cF}{\mathcal{F}}

\newcommand{\cA}{\mathcal{A}}
\newcommand{\cR}{\mathcal{R}}
\newcommand{\cO}{\mathcal{O}}
\newcommand{\cS}{\mathcal{S}}
\newcommand{\cT}{\mathcal{T}}
\newcommand{\cI}{\mathcal{I}}

\newcommand{\cP}{\mathcal{P}}

\newcommand{\tT}{\widetilde{T}}

\newcommand{\Rpt}{\mathbb{R}_+^2}
\newcommand{\Rpk}{\mathbb{R}_+^k}
\newcommand{\lel}{\left\langle}
\newcommand{\rir}{\right\rangle}

\newcommand{\mb}[1]{\mathbb{#1}}

\begin{document}

\title{Subproduct systems}

\author{Orr Moshe Shalit}
\address{Department of Mathematics, Technion - Israel
Institute of Technology, 32000, Haifa, Israel.}
\email{orrms@tx.technion.ac.il}
\author{Baruch Solel}
\address{Department of Mathematics, Technion - Israel
Institute of Technology, 32000, Haifa, Israel.}
\email{mabaruch@tx.technion.ac.il}
\thanks{B.S. was supported in part by the US-Israel Binational Foundation and by the Fund
for the Promotion of Research at the Technion.}

\begin{abstract}
{The notion of a \emph{subproduct system}, a generalization of that of a product system, is introduced. We show that there is an essentially $1$ to $1$ correspondence between $cp$-semigroups and pairs $(X,T)$ where $X$
is a subproduct system and $T$ is an injective subproduct system representation. A similar statement holds for
subproduct systems and units of subproduct systems. This correspondence is used as a framework for developing a dilation theory for $cp$-semigroups. Results we obtain: (i) a $*$-automorphic
dilation to semigroups of $*$-endomorphisms over quite general semigroups; (ii) necessary and sufficient conditions for a semigroup of CP maps to have a $*$-endomorphic dilation; (iii) an analogue of Parrot's example of three
contractions with no isometric dilation, that is, an example of three commuting, contractive normal CP maps on
$B(H)$ that admit no $*$-endomorphic dilation (thereby solving an open problem raised by Bhat in 1998).
Special attention is given to subproduct
systems over the semigroup $\mb{N}$, which are used as a framework for studying tuples of operators satisfying
homogeneous polynomial relations, and the operator algebras they generate. As applications we obtain a
noncommutative (projective) Nullstellansatz, a model for tuples of operators subject to homogeneous
polynomial relations, a complete description of all representations of Matsumoto's subshift C$^*$-algebra
when the subshift is of finite type, and a classification of certain operator algebras -- including an interesting
non-selfadjoint generalization of the noncommutative tori.
}
\end{abstract}

\maketitle

\tableofcontents

\section*{Introduction}

\subsection{Motivation: dilation theory of CP$_0$-semigroups}

We begin by describing the problems that motivated this work.

Let $H$ be a separable Hilbert space, and let $\cM \subseteq B(H)$ be a von Neumann algebra. A \emph{CP map} on $\cM$ is a contractive, normal and completely positive map. A  \emph{CP$_0$-semigroup} on $\cM$ is a family $\Theta = \{\Theta_t\}_{t\geq0}$ of unital CP maps on $\cM$ satisfying the semigroup property
$$\Theta_{s+t}(a) = \Theta_s (\Theta_t(a)) \,\, ,\,\, s,t\geq 0, a\in \cM ,$$
$$\Theta_{0}(a) = a \,\, , \,\,  a\in B(H) ,$$
and the continuity condition
$$\lim_{t\rightarrow t_0} \langle \Theta_t(a)h,g\rangle = \langle \Theta_{t_0}(a)h,g\rangle \,\, , \,\, a\in \cM, h,g \in H .$$
A CP$_0$-semigroup is called an \emph{E$_0$-semigroup} if each of
its elements is a $*$-endomorphism.

Let $\Theta$ be a CP$_0$-semigroup acting on $\cM$, and let $\alpha$ be an
E$_0$-semigroup acting on $\cR$, where $\cR$ is a von Neumann subalgebra of $B(K)$ and $K\supseteq H$. Denote the orthogonal projection of $K$ onto $H$ by $p$. We say that $\alpha$ is an \emph{E$_0$-dilation} of
$\Theta$ if for all $t \geq 0$ and $b \in \cR$
\begin{equation}\label{eq:dilation}
\Theta_t(p b p) = p \alpha_t (b) p .
\end{equation}
In the mid 1990's Bhat proved the following result, known today as ``Bhat's Theorem" (see \cite{Bhat96} for the case $\cM = B(H)$, and also \cite{SeLegue, BS00, MS02, Arv03} for different proofs and for the general case):
\begin{theorem}\label{thm:Bhat}
\emph{{\bf (Bhat).}} Every CP$_0$-semigroup has a unique minimal
E$_0$-dilation.
\end{theorem}

A natural question is then this: \emph{given {\bf two} commuting CP$_0$-semigroups, can one simultaneously dilate them to a pair of commuting E$_0$-semigroups?} In \cite{Shalit08} the following partial positive answer was obtained\footnote{The same result was obtained in \cite{Shalit07b} for nonunital semigroups acting on $\cM = B(H)$.}:

\begin{theorem}\label{thm:E0Dil} {\bf\cite[Theorem 6.6]{Shalit08}}
Let $\{\phi_t\}_{t\geq0}$ and $\{\theta_t\}_{t\geq0}$ be two strongly
commuting CP$_0$-semigroups on a von Neumann algebra $\cM\subseteq
B(H)$, where $H$ is a separable Hilbert space. Then there is a separable Hilbert space $K$ containing $H$ and an orthogonal projection $p:K \rightarrow H$, a von Neumann algebra $\cR \subseteq B(K)$ such that $\cM = p \cR p$, and two commuting
E$_0$-semigroups $\alpha$ and $\beta$ on $\cR$ such that
$$\phi_s \circ \theta_t (p b p) = p \alpha_s \circ \beta_t (b) p$$
for all $s,t \geq 0$ and all $b \in \cR$.
\end{theorem}

In other words: every two-parameter CP$_0$-semigroup that satisfies an additional condition of \emph{strong commutativity} has a two-parameter E$_0$-dilation. The condition of strong commutativity was introduced in \cite{S06}. A precise definition will not be given here. {\bf The main tools in the proof of Theorem \ref{thm:E0Dil}, and also in some of the proofs of Theorem \ref{thm:Bhat}, were product systems of W$^*$-correspondences and their representations}. In fact, the only place in the proof of Theorem \ref{thm:E0Dil} where the assumption of strong commutativity is used, is in the construction of a certain product system. More about that later.

In \cite{Bhat98}, Bhat showed that given a pair of commuting CP maps $\Theta$ and $\Phi$ on $B(H)$, there is a Hilbert space $K\supseteq H$ and a pair of commuting normal $*$-endomorphisms $\alpha$ and $\beta$ acting on $B(K)$ such that
\bes
\Theta^m \circ \Phi^n (p b p) = p \alpha^m \circ \beta^n (b) p \,\, , \,\, b \in B(K)
\ees
for all $m,n \in \mathbb{N}$ (here $p$ denotes the projection of $K$ onto $H$). Later on Solel, using a different method (using in fact product systems and their representations), proved this result for commuting CP maps on arbitrary von Neumann algebras \cite{S06}. Neither one of the above results requires strong commutativity.

In light of the above discussion, and inspired by classical dilation theory \cite{Slocinski,SzNF70}, it is natural to conjecture that \emph{every} two commuting (not necessarily \emph{strongly} commuting) CP$_0$-semigroups have an E$_0$-dilation, and in fact that the same is true for any $k$ commuting CP$_0$-semigroups, for any positive integer $k$. However, the framework given by product systems seems to be too weak to prove this. Trying to bypass this stoppage, we arrived at the notion of a \emph{subproduct system}.

\subsection{Background: from product systems to subproduct systems}

Product systems of Hilbert spaces over $\mb{R}_+$ were introduced by Arveson some 20 years ago in his study of E$_0$-semigroups \cite{Arv89}. In a few imprecise words, a product system of Hilbert spaces over $\mb{R}_+$ is a bundle $\{X(t)\}_{t\in\mb{R}_+}$ of Hilbert spaces such that
\bes
X(s+t) = X(s) \otimes X(t) \,\, , \,\, s,t \in \mb{R}_+.
\ees
We emphasize immediately that Arveson's definition of product systems required also that the bundle carry a certain Borel measurable structure, but we do not deal with these matters here. To every E$_0$-semigroup Arveson associated a product system, and it turns out that the product system associated to an E$_0$-semigroup is a complete cocycle conjugacy invariant of the E$_0$-semigroup.

Later, product systems of Hilbert $C^*$-correspondences over $\mb{R}_+$ appeared (see the survey \cite{SkeideSurvey} by Skeide). In \cite{BS00}, Bhat and Skeide associate with every semigroup of completely positive maps on a C$^*$-algebra $A$ a product system of Hilbert $A$-correspondences. This product system was then used in showing that every semigroup of completely positive maps can be ``dilated" to a semigroup of $*$-endomorphisms. Muhly and Solel introduced a different construction \cite{MS02}: to every CP$_0$-semigroup on a von Neumann algebra $\cM$ they associated a product system of Hilbert W$^*$-correspondences over $\cM'$, the commutant of $\cM$. Again, this product system is then used in constructing an E$_0$-dilation for the original CP$_0$-semigroup.

Product systems of C$^*$-correspondences over semigroups other than $\mb{R}_+$ were first studied by Fowler \cite{Fowler2002}, and they have been studied since then by many authors. In \cite{S06}, product systems over $\mb{N}^2$ (and their representations) were studied, and the results were used to prove that every pair of commuting CP maps has a $*$-endomorphic dilation. Product systems over $\Rpt$ were also central to the proof of Theorem \ref{thm:E0Dil}, where every pair of strongly commuting CP$_0$-semigroups is associated with a product system over $\Rpt$. However, the construction of the product system is one of the hardest parts in that proof. Furthermore, that construction fails when one drops the assumption of strong commutativity, and it also fails when one tries to repeat it for $k$ strongly commuting semigroups.

On the other hand, there is another object that may be naturally associated with a semigroup of CP maps over \emph{any} semigroup: this object is the \emph{subproduct system}, which, when the CP maps act on $B(H)$, is the bundle of Arveson's ``metric operator spaces" (introduced in \cite{Arv97}). Roughly, a subproduct system of correspondences over a semigroup $\cS$ is a bundle $\{X(s)\}_{s\in\cS}$ of correspondences such that
\bes
X(s+t) \subseteq X(s) \otimes X(t) \,\, , \,\, s,t \in \cS.
\ees
See Definition \ref{def:subproduct_system} below. Of course, a difficult problem cannot be made easy just by introducing a new notion, and the problem of dilating $k$-parameter CP$_0$-semigroups remains unsolved. However, subproduct systems did already provide us with an efficient general framework for tackling various problems in operator algebras, and in particular it has led us to a progress toward the solution of the discrete analogue of the above unsolved problem.

This paper consists of two parts. In the first part we introduce subproduct systems over general semigroups, show
the connection between subproduct systems and $cp$-semigroups, and use this connection
to obtain three main results in dilation theory of $cp$-semigroups. The first result is that every $e_0$-semigroup over a (certain kind of) semigroup $\cS$ can be dilated to a
semigroup of $*$-automorphisms on some type I factor. The second is some necessary conditions and sufficient conditions for a $cp$-semigroup to have a (minimal) $*$-endomorphic dilation. The third is
an analogue of Parrot's example of three contractions with no isometric dilation, that is, an example of three commuting, contractive normal CP maps on $B(H)$ that admit no $*$-endomorphic dilation. The CP maps in the stated example can be taken to have \emph{arbitrarily small norm}, providing the first example of a theorem in the classical theory of isometric dilations that cannot be generalized to the theory of $e$-dilations of $cp$-semigroups.

Having convinced the reader that subproduct systems are an interesting and important object, we turn in the second part
of the paper to take a closer look at the simplest examples of subproduct systems, that is, subproduct systems of
Hilbert spaces over $\mb{N}$. We study certain tuples of operators and operator algebras that can be naturally
associated with every subproduct system, and explore the relationship between these objects and the subproduct systems
that give rise to them.

\subsection{Some preliminaries}

$\cM$ and $\cN$ will denote von Neumann subalgebras of $B(H)$, where $H$ is some Hilbert space.

In Sections \ref{sec:subproduct} through \ref{sec:dil}, $\cS$ will denote a sub-semigroup of $\Rpk$.
In fact, in large parts of the paper $\cS$ can be taken to be any semigroup with unit, or at least any \emph{Ore semigroup} (see \cite{Laca} for a definition), but we prefer to avoid this distraction.

\begin{definition}
A \emph{$cp$-semigroup} is a semigroup of CP maps, that is, a family $\Theta = \{\Theta_s\}_{s \in \cS}$ of completely positive, contractive and normal maps
on $\cM$ such that
\bes
\Theta_{s+t}(a) = \Theta_s (\Theta_t(a)) \,\, ,\,\, s,t\in\cS, a\in \cM
\ees
and
\bes
\Theta_{0}(a) = a \,\, ,\,\,  a\in \cM .
\ees
A \emph{$cp_0$-semigroup} is a semigroup of
unital CP maps. An \emph{$e$-semigroup} is a semigroup of $*$-endomorphisms. An \emph{$e_0$-semigroup} is a semigroup of
unital $*$-endomorphisms.
\end{definition}
For concreteness, one should think of the case $\cS = \mb{N}^k$, where a $cp$-semigroup is a $k$-tuple of commuting CP maps, or the case $\cS = \Rpk$, where a $cp$-semigroup is a $k$-parameter semigroup of CP maps, or $k$ mutually commuting one-parameter $cp$-semigroups.

\begin{definition}
Let $\Theta = \{\Theta_s\}_{s \in \cS}$ be a $cp$-semigroup acting on a von Neumann algebra $\cM \subseteq B(H)$. An \emph{$e$-dilation of $\Theta$} is a triple $(\alpha,K,\cR)$ consisting of a Hilbert space $K \supseteq H$ (with orthogonal projection $P_H : K \rightarrow H$), a von Neumann algebra $\cR \subseteq B(K)$ that contains $\cM$ as a corner $\cM = P_H \cR P_H$, and an $e$-semigroup $\alpha = \{\alpha_s\}_{s\in\cS}$ on $\cR$ such that for all $T \in \cR$, $s \in \cS$,
\bes
\Theta_s(P_H T P_H) = P_H \alpha_s(T) P_H .
\ees
\end{definition}

\begin{definition}
Let $\cA$ be a $C^*$-algebra. A \emph{Hilbert
$C^*$-correspondences over $\cA$} is a (right) Hilbert
$\cA$-module $E$ which carries an adjointable, left action of
$\cA$.
\end{definition}
\begin{definition}
Let $\cM$ be a $W^*$-algebra. A \emph{Hilbert
$W^*$-correspondence over $\cM$} is a self-adjoint Hilbert
$C^*$-correspondence $E$ over $\cM$, such that the map $\cM
\rightarrow \cL(E)$ which gives rise to the left action is normal.
\end{definition}

\begin{definition}
Let $E$ be a $C^*$-correspondence over $\cA$, and let $H$ be a
Hilbert space. A pair $(\sigma, T)$ is called a \emph{completely
contractive covariant representation} of $E$ on $H$ (or, for
brevity, a \emph{c.c. representation}) if
\begin{enumerate}
    \item $T: E \rightarrow B(H)$ is a completely contractive linear map;
    \item $\sigma : A \rightarrow B(H)$ is a nondegenerate $*$-homomorphism; and
    \item $T(xa) = T(x) \sigma(a)$ and $T(a\cdot x) = \sigma(a) T(x)$ for all $x \in E$ and  all $a\in\cA$.
\end{enumerate}
If $\cA$ is a $W^*$-algebra and $E$ is $W^*$-correspondence then
we also require that $\sigma$ be normal.
\end{definition}
Given a $C^*$-correspondence $E$ and a c.c. representation
$(\sigma,T)$ of $E$ on $H$, one can form the Hilbert space $E
\otimes_\sigma H$, which is defined as the Hausdorff completion of
the algebraic tensor product with respect to the inner product
$$\langle x \otimes h, y \otimes g \rangle = \langle h, \sigma (\langle x,y\rangle) g \rangle .$$
One then defines $\widetilde{T} : E \otimes_\sigma H \rightarrow H$ by
$$\widetilde{T} (x \otimes h) = T(x)h .$$
\begin{definition}
A c.c. representation $(\sigma, T)$ is called \emph{isometric} if
for all $x, y \in E$,
\begin{equation*}
T(x)^*T(y) = \sigma(\langle x, y \rangle) .
\end{equation*}
(This is the case if and only if $\widetilde{T}$ is an isometry). It
is called \emph{fully coisometric} if $\widetilde{T}$ is a coisometry.
\end{definition}

Given two Hilbert $C^*$-correspondences $E$ and $F$ over $\cA$,
the \emph{balanced} (or \emph{inner}) tensor product $E
\otimes F$ is a Hilbert $C^*$-correspondence over $\cA$
defined to be the Hausdorff completion of the algebraic tensor
product with respect to the inner product
$$\langle x \otimes y, w \otimes z \rangle = \langle y , \langle x,w\rangle \cdot z \rangle \,\, , \,\,  x,w\in E, y,z\in F .$$
The left and right actions are defined as $a \cdot (x \otimes y) =
(a\cdot x) \otimes y$ and $(x \otimes y)a = x \otimes (ya)$,
respectively, for all $a\in A, x\in E, y\in F$. When working
in the context of $W^*$-correspondences, that is, if $E$ and $F$
are $W$*-correspondences and $\cA$ is a $W^*$-algebra, then $E
\otimes F$ is understood do be the \emph{self-dual
extension} of the above construction.

\subsection{Detailed overview of the paper}

Subproduct systems, their representations, and their units, are defined in the next section.
The following two sections, \ref{sec:subncp} and \ref{sec:subunitsncp},
can be viewed as a reorganization  and sharpening of some known results, including several new observations.

Section \ref{sec:subncp} establishes the correspondence between $cp$-semigroups and subproduct systems. It is shown that given a subproduct system $X$ of $\cN$ - correspondences and a subproduct system representation $R$ of $X$ on $H$, we may construct a $cp$-semigroup $\Theta$ acting on $\cN'$. We denote this assignment as $\Theta = \Sigma(X,R)$. Conversely, it is shown that given a  $cp$-semigroup $\Theta$ acting on $\cM$, there is a subproduct system $E$ (called the \emph{Arveson-Stinespring subproduct system} of $\Theta$) of $\cM'$-correspondences and an \emph{injective} representation $T$ of $E$ on $H$ such that $\Theta = \Sigma(E,T)$. Denoting this assignment as $(E,T) = \Xi(\Theta)$, we have that $\Sigma \circ \Xi$ is the identity. In Theorem \ref{thm:essentially_inverse} we show that $\Xi \circ \Sigma$ is also, after restricting to pairs $(X,R)$ with $R$ an injective representation (and up to some ``isomorphism"), the identity. This allows us to deduce (Corollary \ref{cor:onlyproductisorep}) that a subproduct system that is not a product system has no isometric representations.
We introduce the \emph{Fock spaces} associated to a subproduct system and the
canonical \emph{shift representations}. These constructs allow us to show that every subproduct system
is the Arveson-Stinespring subproduct system of some $cp$-semigroup.

In Section \ref{sec:subunitsncp} we briefly sketch the picture that is dual to that of Section \ref{sec:subncp}.
It is shown that given a subproduct system and a unit of that subproduct system one may construct a $cp$-semigroup, and that every $cp$-semigroup arises this way.

In Section \ref{sec:aut_dil}, we construct for every subproduct system $X$ and every fully coisometric subproduct system
representation $T$ of $X$ on a Hilbert space, a semigroup $\hat{T}$ of contractions on a Hilbert space that captures ``all the information" about $X$ and $T$. This construction is a modification of the construction introduced in \cite{Shalit07a} for the case where $X$ is a \emph{product} system. It turns out that when $X$ is merely a \emph{sub}product system, it is hard to apply $\hat{T}$ to obtain new results about the representation $T$. However, when $X$ is a true \emph{product} system $\hat{T}$ is very handy, and we use it to prove that every $e_0$-semigroup has a $*$-automorphic dilation (in a certain sense).

Section \ref{sec:dil} begins with some general remarks regarding dilations and pieces of subproduct system representations,
and then the connection between the dilation theories of $cp$-semigroups and of representations of subproduct systems is made.
We define the notion of a \emph{subproduct subsystem} and then we define \emph{dilations} and \emph{pieces} of subproduct
system representations. These notions generalize the notions of \emph{commuting piece} or \emph{$q$-commuting piece} of
\cite{BBD03} and \cite{Dey07}, and also generalizes the definition of \emph{dilation} of a product system
representation of \cite{MS02}. Proposition \ref{prop:dil_rep_dil_CP}, Theorem \ref{thm:edil_repdil} and Corollary \ref{cor:edil_repdil} show that the 1-1 correspondences $\Sigma$ and $\Xi$ between $cp$-semigroups and subproduct systems with representations take isometric dilations of representations to $e$-dilations and vice-versa. This is used to obtain an example of three commuting, unital and contractive CP maps on $B(H)$ for which there exists no $e$-dilation acting on a $B(K)$, and no \emph{minimal} dilation acting on any von Neumann algebra (Theorem \ref{thm:parrot}).

In Section \ref{sec:dil} we also present a reduction of both the problem of constructing an $e_0$-dilation to a $cp_0$-semigroup, and the problem of constructing an $e$-dilation to a $k$-tuple of commuting CP maps with \emph{small enough norm},
to the problem of embedding a subproduct system in a larger \emph{product system}. We show that not every subproduct system can be embedded in a product system (Proposition \ref{prop:sprdctcntrexample}), and we use this to construct an example of three commuting CP maps $\theta_1, \theta_2, \theta_3$ such that for \emph{any} $\lambda >0$ the three-tuple $\lambda \theta_1, \lambda \theta_2, \lambda \theta_3$ has no $e$-dilation (Theorem \ref{thm:strongparrot}). This unexpected phenomenon has no counterpart in the classical theory of isometric dilations, and provides the first example of a theorem in classical dilation theory that cannot be generalized to the theory of $e$-dilations of $cp$-semigroups.

The developments described in the first part of the paper indicate that subproduct systems are
worthwhile objects of study, but to make progress we must look at plenty concrete examples.
In the second part of the paper we begin studying subproduct systems of Hilbert spaces over the semigroup $\mb{N}$. In Section \ref{sec:subproductN} we show that every subproduct system (of W$^*$-correspondences) over $\mb{N}$ is isomorphic to a \emph{standard} subproduct system, that is, it is a subproduct subsystem of the full product system $\{E^{\otimes n}\}_{n\in\mb{N}}$ for some W$^*$-correspondence $E$. Using the results of the previous section, this gives a new proof to the discrete analogue of Bhat's Theorem: \emph{every $cp_0$-semigroup over $\mb{N}$ has an $e_0$-dilation}. Given a subproduct system we define the \emph{standard $X$-shift}, and we show that if $X$ is a subproduct subsystem of $Y$, then the standard $X$-shift is the maximal $X$-piece of the standard $Y$-shift, generalizing and unifying results from \cite{BBD03,Dey07,Popescu06}.

In Section \ref{sec:projective} we explain why subproduct systems are convenient for studying noncommutative
projective algebraic geometry. We show that every homogeneous ideal $I$ in the algebra
$\mb{C}\langle x_1, \ldots, x_d\rangle$ of noncommutative polynomials corresponds to a unique subproduct system
$X_I$, and vice-versa. The representations of $X_I$ on a Hilbert space $H$ are precisely determined by the $d$-tuples in the zero set of $I$,
\bes
Z(I) = \{\underline{T} = (T_1, \ldots, T_d) \in B(H)^d : \forall p \in I. p(\underline{T}) = 0\}.
\ees
A noncommutative version of the Nullstellansatz is obtained, stating that
\bes
\{p \in \mb{C}\langle x_1, \ldots, x_d\rangle : \forall \underline{T} \in Z(I). p(\underline{T}) =0\} = I.
\ees

Section \ref{sec:universal} starts with a review of a powerful tool, Gelu Popescu's ``Poisson Transform" \cite{Popescu99}.
Using this tool we derive some basic results (obtained previously by Popescu in \cite{Popescu06}) which allow us to
identify the operator algebra $\cA_X$ generated by the $X$-shift as the universal unital operator algebra generated by a row
contraction subject to homogeneous polynomial identities. We then prove that every completely bounded representation of
a subproduct system $X$ is a piece of a scaled inflation of the $X$-shift, and derive a related ``von Neumann inequality".

In Section \ref{sec:operator_algebra} we discuss the relationship between a subproduct system $X$ and $\cA_X$, the (non-selfadjoint
operator algebra generated by the $X$-shift). The main result in this section is Theorem \ref{thm:algebra_iso}, which
says that $X \cong Y$ if and only if $\cA_X$ is completely isometrically isomorphic to $\cA_Y$ by an isomorphism that
preserves the vacuum state. This result is used in Section \ref{sec:qcommuting}, where we study the universal
norm closed unital operator algebra generated by a row contraction $(T_1, \ldots, T_d)$ satisfying the relations
\bes
T_i T_j = q_{ij}T_j T_i \,\, , \,\, 1 \leq i < j \leq d,
\ees
where $q = (q_{i,j})_{i,j=1}^d \in M_n(\mb{C})$ is a matrix such that $q_{j,i} = q_{i,j}^{-1}$. These non-selfadjoint analogues of the noncommutative tori, are shown to be classified by their subproduct systems when $q_{i,j}\neq 1$ for all $i,j$. In particular, when $d=2$, we obtain the universal algebra for the relation
\bes
T_1 T_2 = q T_2 T_1,
\ees
which we call $\cA_q$. It is shown that $\cA_q$ is isomorphically isomorphic to $\cA_r$ if and only if $q = r$ or $q = r^{-1}$.

In Section \ref{sec:A} we describe all standard maximal subproduct systems $X$ with $\dim X(1) = 2$ and $\dim X(2) = 3$, and classify their algebras up to isometric isomorphisms.

In the closing section of this paper, Section \ref{sec:subshift}, we find that subproduct systems are also closely related to \emph{subshifts} and to the \emph{subshift C$^*$-algebras} introduced by K. Matsumoto \cite{Ma}. We show how every subshift gives rise to a subproduct system, and characterize the subproduct systems that come from subshifts. We use this connection together with the results of Section \ref{sec:universal} to describe all representations of subshift C$^*$-algebras that come from a subshift of \emph{finite type} (Theorem \ref{thm:rep_subshift}).

\subsection{Acknowledgment}
The authors owe their thanks to Eliahu Levy for pointing out a mistake in a previous version of the paper.

\newpage
\setcounter{secnumdepth}{2}
\part{Subproduct systems and $cp$-semigroups}

\section{Subproduct systems of Hilbert $W^*$-correspondences}\label{sec:subproduct}

\begin{definition}\label{def:subproduct_system}
Let $\cN$ be a von Neumann algebra. A \emph{subproduct system of Hilbert $W^*$-correspondences} over $\cN$ is a family $X = \{X(s)\}_{s\in\cS}$ of Hilbert $W^*$-correspondences over $\cN$ such that
\begin{enumerate}
\item $X(0) = \cN$,
\item For every $s,t \in \cS$ there is a coisometric mapping of $\cN$-correspondences
$$U_{s,t}: X(s) \otimes X(t) \rightarrow X(s+t),$$
\item The maps $U_{s,0}$ and $U_{0,s}$ are given by the left and right actions of $\cN$ on $X(s)$,
\item The maps $U_{s,t}$ satisfy the following associativity condition:
\begin{equation}\label{eq:assoc_prod}
U_{s+t,r} \left(U_{s,t} \otimes I_{X(r)} \right) = U_{s,t+r} \left(I_{X(s)} \otimes U_{t,r} \right).
\end{equation}
\end{enumerate}
\end{definition}
The difference between a subproduct system and a product system is that in a subproduct system the maps $U_{s,t}$ are only required to be coisometric, while in a product system these maps are required to be unitaries. Thus, given the image $U_{s,t}(x \otimes y)$ of $x \otimes y$ in $X(s + t)$, one cannot recover $x$ and $y$. Thus, subproduct systems may be thought of as \emph{irreversible} product systems. The terminology is, admittedly, a bit awkward. It may be more sensible -- however, impossible at present -- to use the term \emph{product system} for the objects described above and to use the term \emph{full product system} for product system.

\begin{example}\label{expl:full}
\emph{
The simplest example of a subproduct system $F = F_E = \{F(n)\}_{n \in \mb{N}}$ is given by
\bes
F(n) = E^{\otimes n},
\ees
where $E$ is some W$^*$-correspondence. $F$ is actually a product system. We shall call this subproduct system \emph{the full product system (over $E$)}.}
\end{example}

\begin{example}\label{expl:symm}\emph{
Let $E$ be a fixed Hilbert space. We define a subproduct system (of Hilbert spaces) $SSP = SSP_E$ over $\mb{N}$ using the familiar symmetric tensor products (one can obtain a subproduct system from the anti-symmetric tensor products as well). Define
$$E^{\otimes n} = E \otimes \cdots \otimes E ,$$
($n$ times). Let $p_n$ be the projection of $E^{\otimes n}$ onto the symmetric subspace of $E^{\otimes n}$, given by
$$p_n k_1 \otimes \cdots \otimes k_n = \frac{1}{n!}\sum_{\sigma \in S_n} k_{\sigma^{-1}(1)} \otimes \cdots \otimes k_{\sigma^{-1}(n)}.$$
We define
$$SSP(n) = E^{\circledS n} : = p_n  E^{\otimes n},$$
the symmetric tensor product of $E$ with itself $n$ times ($SSP(0) = \mb{C}$).
We define a map $U_{m,n}: SSP(m) \otimes SSP(n) \rightarrow SSP(m+n)$ by
$$U_{m,n} (x \otimes y) = p_{m+n} (x \otimes y).$$
The $U$'s are coisometric maps because every projection, when considered as a map from its domain onto its range, is coisometric. A straightforward calculation shows that (\ref{eq:assoc_prod}) holds (see \cite[Corollary 17.2]{Parthasarathy}). In these notes we shall refer to $SSP$ (or $SSP_E$ to be precise) as the \emph{symmetric subproduct system (over $E$)}. }
\end{example}
\begin{definition}
Let $X$ and $Y$ be two subproduct systems over the same semigroup $\cS$ (with families of coisometries $\{U_{s,t}^X\}_{s,t \in \cS}$ and $\{U_{s,t}^Y\}_{s,t \in \cS}$).  A family $V=\{V_s\}_{s\in \cS}$ of maps $V_{s}: X(s) \rightarrow Y(s)$ is called \emph{a morphism} of subproduct systems if $V_0$ is a unital $*$-isomorphism, if for all $s \in \cS \setminus \{0\}$ the map $V_s$ is a coisometric correspondence map, and if for all $s,t \in \cS$ the following identity holds:
\be\label{eq:iso}
V_{s+t}\circ U_{s,t}^X = U_{s,t}^Y \circ (V_s \otimes V_t) .
\ee
$V$ is said to be an \emph{isomorphism} if $V_s$ is a unitary for all $s \in \cS \setminus \{0\}$.
$X$ is said to be isomorphic to $Y$ if there exists an isomorphism $V: X \rightarrow Y$.
\end{definition}
The above notion of \emph{morphism} is not optimized in any way. It is simply precisely what we need in order to develop dilation theory for $cp$-semigroups.

\begin{definition}\label{def:rep}
Let $\cN$ be a von Neumann algebra, let $H$ be a Hilbert space, and let $X$ be a
subproduct system of Hilbert $\cN$-correspondences over the semigroup
$\cS$. Assume that $T : X \rightarrow B(H)$, and write $T_s$ for
the restriction of $T$ to $X(s)$, $s \in \cS$, and $\sigma$ for
$T_0$. $T$ (or $(\sigma, T)$) is said to be a \emph{completely
contractive covariant representation} of $X$ if
\begin{enumerate}
    \item For each $s \in \cS$, $(\sigma, T_s)$ is a c.c. representation of $X(s)$; and
    \item\label{it:sg} $T_{s+t}(U_{s,t}(x\otimes y)) = T_s(x)T_t(y)$ for all $s,t \in \cS$ and all $x \in X(s), y \in X(t)$.
\end{enumerate}
$T$ is said to be an isometric (fully coisometric) representation if it is an isometric (fully coisometric) representation on every fiber $X(s)$.
\end{definition}
Since we shall not be concerned with any other kind of representation, we shall call a completely contractive covariant representation of a subproduct system simply a \emph{representation}.

\begin{remark}\label{rem:reptilde}
\emph{Item 2 in the above definition of product system can be rewritten as follows:
\bes
\tT_{s+t} (U_{s,t} \otimes I_H) = \tT_s (I_{X(s)} \otimes \tT_t).
\ees
Here $\tT_s: X(s) \otimes_\sigma H \rightarrow H$ is the map given by}
$$\tT_s (x \otimes h) = T_s(x)h .$$
\end{remark}

\begin{example}\label{expl:FockRep}\emph{
We now define a representation $T$ of the symmetric subproduct system $SSP$ from Example \ref{expl:symm} on the symmetric Fock space. Denote by $\mathfrak{F}_+$ the symmetric Fock space
$$\mathfrak{F}_+ = \bigoplus_{n \in \mb{N}} E^{\circledS n} .$$
For every $n \in \mb{N}$, the map $T_n : SSP(n) = E^{\circledS n} \rightarrow B(\mathfrak{F}_+)$  is defined on the $m$-particle space $E^{\circledS m}$ by putting
$$T_n(x) y = p_{n+m}(x \otimes y) $$
for all $x \in X(n), y \in E^{\circledS m}$. Then $T$ extends to a representation of the subproduct system $SSP$ on $\mathfrak{F}_+$ (to see that item \ref{it:sg} of Definition \ref{def:rep} is satisfied one may use again \cite[Corollary 17.2]{Parthasarathy}).}
\end{example}
\begin{definition}
Let $X = \{X(s)\}_{s \in \cS}$ be a subproduct
system of $\cN$-correspondences over $\cS$. A family $\xi = \{\xi_s \}_{s \in \cS}$
is called a \emph{unit} for $X$ if
\be\label{eq:unit}
\xi_s \otimes \xi_t = U_{s,t}^* \xi_{s+t}.
\ee
A unit $\xi = \{\xi_s \}_{s \in \cS}$ is called \emph{unital} if $\lel \xi_s , \xi_s \rir = 1_{\cN}$ for all $s \in \cS$,
it is called \emph{contractive} if $\lel \xi_s , \xi_s \rir \leq 1_{\cN}$ for all $s \in \cS$,
and it is called \emph{generating} if $X(s)$ is spanned by elements of the form
\be\label{eq:spanning}
U_{s_1 + \cdots + s_{n-1},s_n}(\cdots U_{s_1+s_2,s_3}( U_{s_1,s_2}(a_1 \xi_{s_1} \otimes a_2 \xi_{s_2}) \otimes a_3 \xi_{s_3}) \otimes \cdots \otimes a_n \xi_{s_n} a_{n+1} ),
\ee
where $s = s_1 + s_2 + \cdots + s_n$.
\end{definition}

From (\ref{eq:unit}) follows the perhaps more natural looking
\bes
U_{s,t}(\xi_s \otimes \xi_t) = \xi_{s+t}.
\ees
\begin{example}\label{expl:symmUnit}\emph{
A unital unit for the symmetric subproduct system $SSP$ from Example \ref{expl:symm} is given by
defining $\xi_0 = 1$ and
\bes
\xi_{n} = \underbrace{v \otimes v \otimes \cdots \otimes v}_{n \text{ times}}
\ees
for $n \geq 1$. This unit is generating only if $E$ is one dimensional.
}
\end{example}

\section{Subproduct system representations and $cp$-semigroups}\label{sec:subncp}

In this section, following Muhly and Solel's constructions from \cite{MS02}, we show that subproduct systems and their representations provide a framework for dealing with $cp$-semigroups, and allow us to obtain a generalization
of the classical result of Wigner that any strongly continuous one-parameter group of automorphisms of
$B(H) $ is given by $X\mapsto U_t X U_t^*$ for
some one-parameter unitary group $\{U_t\}_{t\in \mathbb{R}}$.

\subsection{All $cp$-semigroups come from subproduct system representations}

\begin{proposition}\label{prop:semigroup}
Let $\cN$ be a von Neumann algebra and let $X$ be a subproduct system of $\cN$-correspondences over $\cS$, and let $R$ be completely contractive covariant representation of $X$ on a Hilbert space $H$, such that $R_0$ is unital. Then the family of maps
\be\label{eq:reprep}
\Theta_s : a \mapsto \widetilde{R}_s (I_{X(s)} \otimes a) \widetilde{R}_s^* \,\, , \,\, a \in R_0 (\cN)',
\ee
is a semigroup of CP maps on $R_0 (\cN)'$. Moreover, if $R$ is an isometric (a fully coisometric) representation, then $\Theta_s$ is a $*$-endomorphism (a unital map) for all $s\in\cS$.
\end{proposition}
\begin{proof}
By Proposition 2.21 in \cite{MS02}, $\{\Theta_s\}_{s\in\cS}$ is a family of contractive, normal, completely positive maps on $R_0(\cN)'$. Moreover, these maps are unital if $R$ is a fully coisometric representation, and they are $*$-endomorphisms if $R$ is an isometric representation. It remains is to check that $\Theta = \{\Theta_s \}_{s\in\cS}$ satisfies the semigroup condition $\Theta_s \circ \Theta_t = \Theta_{s + t}$. Fix $a \in R_0 (\cN)'$. For all $s,t\in\cS$,
\begin{align*}
\Theta_s (\Theta_t (a))
&= \widetilde{R}_s \left(I_{X(s)} \otimes \left(\widetilde{R}_t (I_{X(t)} \otimes a) \widetilde{R}_t^*\right)\right) \widetilde{R}_s^* \\
&= \widetilde{R}_s (I_{X(s)} \otimes \widetilde{R}_t) (I_{X(s)}\otimes I_{X(t)} \otimes a)(I_{X(s)} \otimes \widetilde{R}_t^*) \widetilde{R}_s^* \\
&= \widetilde{R}_{s + t} (U_{s,t}\otimes I_G)(I_{X(s)}\otimes I_{X(t)}\otimes a)(U_{s,t}^{*}\otimes I_G)\widetilde{R}_{s + t}^* \\
&= \widetilde{R}_{s + t} (I_{X(s\cdot t)}\otimes a)\widetilde{R}_{s + t}^* \\
&= \Theta_{s + t}(a) .
\end{align*}
Using the fact that $R_0$ is unital, we have
\begin{align*}
\Theta_0(a) h
&= \widetilde{R_0} (I_{\cN} \otimes a) \widetilde{R_0}^* h \\
&= \widetilde{R_0} (I_{\cN} \otimes a) (1_{\cN} \otimes h) \\
&= R_0(1_{\cN})ah \\
&= ah ,
\end{align*}
thus $\Theta_0(a) = a$ for all $a\in \cN$.
\end{proof}

We will now show that \emph{every} $cp$-semigroup is given by a subproduct representation
as in (\ref{eq:reprep}) above. We recall some constructions from \cite{MS02} (building on the foundations  set in \cite{Arv97}).

Fix a CP map $\Theta$ on von Neumann algebra $\cM \subseteq B(H)$. We define $\cM \otimes_\Theta H$ to be the Hausdorff completion of the algebraic tensor product $\cM \otimes H$ with respect to the sesquilinear positive semidefinite form
\bes
\lel T_1 \otimes h_1 , T_2 \otimes h_2 \rir = \lel h_1, \Theta(T_1^* T_2) h_2 \rir .
\ees
We define a representation $\pi_\Theta$ of $\cM$ on $\cM \otimes_\Theta H$ by
\bes
\pi_\Theta (S) (T \otimes h) = ST \otimes h,
\ees
and we define a (contractive) linear map $W_\Theta : H \rightarrow \cM \otimes H$ by
\bes
W_\Theta (h) = I \otimes h.
\ees
If $\Theta$ is unital then $W_\Theta$ is an isometry, and if $\Theta$ is an endomorphism then $W_\Theta$ is
a coisometry. The adjoint of $W_\Theta$ is given by
\bes
W_\Theta^* (T \otimes h) = \Theta(T)h .
\ees

For a given CP semigroup $\Theta$ on $\cM$, Muhly and Solel defined in \cite{MS02} a $W^*$-correspondence $E_\Theta$ over $\cM'$ and a c.c.
representation $(\sigma,T_\Theta)$ of $E_\Theta$ on $H$ such that for all $a \in \cM$
\be\label{eq:reprep1}
\Theta(a) = \tT_\Theta \left(I_{E_\Theta} \otimes a \right)\tT_\Theta^* .
\ee
The $W^*$-correspondence $E_\Theta$ is defined as the intertwining space
\bes
E_\Theta = \cL_\cM (H, \cM \otimes_\Theta H),
\ees
where
\bes
\cL_\cM (H, \cM \otimes_\Theta H):= \{X \in B(H,\cM \otimes_\Theta H) \big| \forall T \in \cM. XT = \pi_\Theta(T)X \}.
\ees
The left and right actions of $\cM'$ are given by
\bes
S \cdot X = (I \otimes S)X \quad , \quad X \cdot S = XS
\ees
for all $X \in E_\Theta$ and $S \in \cM'$. The $\cM'$-valued inner product on $E_\Theta$ is defined by
$\lel X, Y \rir = X^* Y$. $E_\Theta$ is called \emph{the Arveson-Stinespring correspondence} (associated with $\Theta$).

The representation $(\sigma,T_\Theta)$ is defined by letting $\sigma = {\bf id}_{\cM'}$, the identity representation
of $\cM'$ on $H$, and by defining
\bes
T_\Theta(X) = W_\Theta^* X .
\ees
$({\bf id}_{\cM'}, T_\Theta)$ is called \emph{the identity representation} (associated with $\Theta$). We remark that the paper \cite{MS02}
focused on unital CP maps, but the results we cite are true for nonunital CP maps, with the proofs unchanged.

The case where $\cM = B(H)$ in the following theorem appears, in essence at least, in \cite{Arv97}.

\begin{theorem}\label{thm:reprep}
Let $\Theta = \{\Theta_s \}_{s\in\cS}$ be a $cp$-semigroup on a von Neumann algebra $\cM \subseteq B(H)$, and for all $s\in \cS$ let $E(s) := E_{\Theta_s}$ be the Arveson-Stinespring correspondence
associated with $\Theta_s$, and let $T_s := T_{\Theta_s}$ denote the identity representation for $\Theta_s$. Then
$E = \{E(s)\}_{s\in\cS}$ is a subproduct system of $\cM'$-correspondences, and $({\bf id}_{\cM'},T)$ is a representation of $E$ on $H$ that satisfies
\be\label{eq:RepRep}
\Theta_s(a) = \tT_s\left(I_{E(s)} \otimes a \right) \tT_s^*
\ee
for all $a \in \cM$ and all $s \in \cS$. $T_s$ is injective for all $s \in \cS$.
If $\Theta$ is an $e$-semigroup (cp$_0$-semigroup), then $({\bf id}_{\cM'},T)$ is isometric (fully coisometric).
\end{theorem}

\begin{proof}
We begin by defining the subproduct system maps $U_{s,t}:E(s) \otimes E(t) \rightarrow E(s+t)$. We use the constructions made in \cite[Proposition 2.12]{MS02} and the surrounding discussion. We define
\bes
U_{s,t} = V_{s,t}^*\Psi_{s,t} \,\,,
\ees
where the map
\bes
\Psi_{s,t}: \cL_\cM(H,\cM \otimes_{\Theta_s} H) \otimes \cL_\cM(H,\cM \otimes_{\Theta_t} H) \rightarrow \cL_\cM(H,\cM \otimes_{\Theta_t} \cM \otimes_{\Theta_s} H)
\ees
is given by $\Psi_{s,t}(X \otimes Y) = (I \otimes X)Y$, and
\bes
V_{s,t}: \cL_\cM(H,\cM \otimes_{\Theta_{s+t}}H) \rightarrow \cL_\cM(H,\cM \otimes_{\Theta_t} \cM \otimes_{\Theta_s} H)
\ees
is given by $V_{s,t}(X) = \Gamma_{s,t} \circ X$, where $\Gamma_{s,t}: \cM \otimes_{\Theta_{s+t}} H \rightarrow \cM \otimes_{\Theta_t} \cM \otimes_{\Theta_s} H$ is the isometry
\bes
\Gamma_{s,t} : S \otimes_{\Theta_{s+t}} h \mapsto S \otimes_{\Theta_t} I \otimes_{\Theta_s} h .
\ees
By  \cite[Proposition 2.12]{MS02}, $U_{s,t}$ is a coisometric bimodule map for all $s,t \in \cS$. To see that the $U$'s compose associatively as in (\ref{eq:assoc_prod}), take $s,t,u \in \cS$, $X \in E(s), Y \in E(t), Z \in E(u)$, and compute:
\begin{align*}
U_{s,t+u}(I_{E(s)} \otimes U_{t,u})(X \otimes Y \otimes Z)
&= U_{s,t+u}(X \otimes V_{t,u}^*(I \otimes Y)Z) \\
&= V_{s,t+u}^*\left((I \otimes X) V_{t,u}^*(I \otimes Y)Z \right) \\
&= \Gamma_{s,t+u}^*(I \otimes X) \Gamma_{t,u}^*(I \otimes Y)Z
\end{align*}
and
\begin{align*}
U_{s+t,u}(U_{s,t} \otimes I_{E(u)})(X \otimes Y \otimes Z)
&= U_{s+t,u}( V_{s,t}^*(I \otimes X)Y \otimes Z) \\
&= V_{s+t,u}^*\left((I \otimes V_{s,t}^*(I \otimes X)Y)Z\right) \\
&= \Gamma_{s+t,u}^*(I \otimes \Gamma_{s,t}^*)(I \otimes I \otimes X)(I \otimes Y)Z \,\, .
\end{align*}
So it suffices to show that
\bes
\Gamma_{s,t+u}^*(I \otimes X) \Gamma_{t,u}^* = \Gamma_{s+t,u}^*(I \otimes \Gamma_{s,t}^*)(I \otimes I \otimes X)
\ees
It is easier to show that their adjoints are equal. Let $a \otimes h$ be a typical element of $\cM \otimes_{\Theta_{s+t+u}} h$.
\begin{align*}
\Gamma_{t,u}(I \otimes X^*)\Gamma_{s,t+u} (a \otimes_{\Theta_{s+t+u}} h)
&= \Gamma_{t,u}(I \otimes X^*)(a \otimes_{\Theta_{t+u}} I \otimes_{\Theta_s} h) \\
&= \Gamma_{t,u}(a \otimes_{\Theta_{t+u}} X^*(I\otimes_{\Theta_s} h)) \\
&= a \otimes_{\Theta_{u}} I \otimes_{\Theta_{t}}X^*(I\otimes_{\Theta_s} h).
\end{align*}
On the other hand
\begin{align*}
(I \otimes I \otimes X^*)(I \otimes \Gamma_{s,t})\Gamma_{s+t,u} (a \otimes_{\Theta_{s+t+u}} h)
&= (I \otimes I \otimes X^*)(I \otimes \Gamma_{s,t}) (a \otimes_{\Theta_{u}} I \otimes_{\Theta_{s+t}} h)  \\
&= (I \otimes I \otimes X^*) (a \otimes_{\Theta_{u}} I \otimes_{\Theta_{t}} I \otimes_{\Theta_{s}} h)\\
&= a \otimes_{\Theta_{u}} I \otimes_{\Theta_{t}}X^*(I\otimes_{\Theta_s} h).
\end{align*}
This shows that the maps $\{U_{s,t}\}$ make $E$ into a subproduct system.

Let us now verify that $T$ is a representation of subproduct systems. That $({\bf id}_{\cM'},T_s)$ is a c.c. representation
of $E(s)$ is explained in \cite[page 878]{MS02}. Let $X \in E(s), Y \in E(t)$.
\bes
T_{s+t}(U_{s,t}(X \otimes Y)) = W_{\Theta_{s+t}}^* \Gamma_{s,t}^*(I \otimes X)Y,
\ees
while
\bes
T_{s}(X) T_t(Y) = W_{\Theta_{s}}^*XW_{\Theta_{t}}^*Y.
\ees
But for all $h \in H$,
\begin{align*}
W_{\Theta_{t}} X^* W_{\Theta_{s}} h &= W_{\Theta_{t}} X^* (I \otimes_{\Theta_s} h) \\
&= I \otimes_{\Theta_t} X^* (I \otimes_{\Theta_s} h) \\
&= (I \otimes X^*) (I \otimes_{\Theta_t} I \otimes_{\Theta_s} h) \\
&= (I \otimes X^*) \Gamma_{s,t}(I \otimes_{\Theta_{s+t}} h) \\
&= (I \otimes X^*) \Gamma_{s,t}W_{\Theta_{s+t}} h ,
\end{align*}
which implies
$W_{\Theta_{s}}^*XW_{\Theta_{t}}^*Y = W_{\Theta_{s+t}}^* \Gamma_{s,t}^*(I \otimes X)Y$, so
we have the desired equality
\bes
T_{s+t}(U_{s,t}(X \otimes Y)) = T_{s}(X) T_t(Y) .
\ees
Equation (\ref{eq:RepRep}) is a consequence of (\ref{eq:reprep1}).
The injectivity of the identity representation has already been noted by Solel in \cite{S06} (for all $h, g \in H$ and $a \in M$, $\langle W^*_\Theta X a^* h, g \rangle =  \langle Xa^* h, I \otimes g \rangle = \langle (I \otimes a^*)Xh, I \otimes g \rangle = \langle Xh, a \otimes g \rangle$ ).
The final assertion of the theorem is trivial (if $\Theta_s$ is a $*$-endomorphism, then $W_{\Theta_s}$ is a coisometry).
\end{proof}

\begin{definition}\label{def:identityrep}
Given a $cp$-semigroup $\Theta$ on a $W^*$ algebra $\cM$, the pair $(E,T)$ constructed in Theorem \ref{thm:reprep} is called \emph{the identity representation} of $\Theta$, and $E$ is called \emph{the Arveson-Stinespring subproduct system} associated with $\Theta$.
\end{definition}

\begin{remark}\label{rem:identityproduct}
\emph{If follows from \cite[Proposition 2.14]{MS02}, if $\Theta$ is an $e$-semigroup, then the identity representation
subproduct system is, in fact, a (full) product system.}
\end{remark}

\begin{remark}\emph{
In \cite{Markiewicz}, Daniel Markiewicz has studied the Arveson-Stinespring subproduct system of a CP$_0$-semigroup over $\mb{R}_+$ acting on $B(H)$, and has also shown that it carries a structure of a \emph{measurable Hilbert bundle}. }
\end{remark}

\subsection{Essentially all injective subproduct system representations come from $cp$-semigroups}\label{subsec:essentially}

The following generalizes and is motivated by \cite[Proposition 5.7]{S06}. We shall also repeat some arguments from  \cite[Theorem 2.1]{MS07}.

By Theorem \ref{thm:reprep}, with every $cp$-semigroup $\Theta = \{\Theta_s \}_{s \in \cS}$ on $\cM \subseteq B(H)$ we can associate a pair $(E, T)$ - the identity representation of $\Theta$ - consisting of a subproduct system $E$ (of correspondences over $\cM'$) and an injective subproduct system c.c. representation $T$. Let us write $(E,T) = \Xi(\Theta)$.
Conversely, given a pair $(X,R)$ consisting of a subproduct system $X$ of correspondences over $\cM'$ and a c.c. representation $R$ of $X$ such that $R_0 = id$, one may define by equation (\ref{eq:reprep}) a $cp$-semigroup $\Theta$ acting on $\cM$, which we denote as $\Theta = \Sigma(X,R)$. The meaning of equation (\ref{eq:RepRep}) is that $\Sigma \circ \Xi$ is the identity map on the set of $cp$-semigroups of $\cM$.
We will show below that $\Xi \circ \Sigma$, when restricted to pairs $(X,R)$ such that $R$ is injective, is also, essentially, the identity. When $(X,R)$ is not injective, we will show that $\Xi \circ \Sigma(X,R)$ ``sits inside" $(X,R)$.

\begin{theorem}\label{thm:essentially_inverse}
Let $\cN$ be a W$^*$-algebra, let $X = \{X(s) \}_{s \in \cS}$ be a subproduct system of $\cN$-correspondences, and let $R$ be a c.c. representation of $X$ on $H$, such that $\sigma := R_0$ is faithful and nondegenerate. Let $\cM \subseteq B(H)$ be the commutant $\sigma(\cN)'$ of $\sigma(\cN)$. Let $\Theta = \Sigma(X,R)$, and let $(E,T) = \Xi(\Theta)$. Then there is a morphism of subproduct systems $W: X \rightarrow E$ such that
\be\label{eq:isosubrep}
R_s = T_s \circ W_s \,\, ,\,\, s \in \cS.
\ee
$W_s^* W_s = I_{X(s)} - q_s$, where $q_s$ is the orthogonal projection of $X(s)$ onto $\textrm{Ker}R_s$.
In particular, $W$ is an isomorphism if and only if $R_s$ is injective for all $s \in \cS$.
\end{theorem}
\begin{remark}
\emph{The construction of the morphism $W$ below basically comes from \cite{S06,MS07}, and it remains only to show that it respects the subproduct system structure. The details are carried out for completeness.}
\end{remark}
\begin{proof}
We may construct a subproduct system $X'$ of $\cM'$-correspondences (recall that $\cM' = \sigma(\cN)$), and a representation $T'$ of $X'$ on $H$ such that $T'_0$ is the identity, in such a way that $(X,T)$ may be naturally identified with $(X',T')$. Indeed, put
\bes
X'(0) = \cM' \,\,,\,\, X'(s) = X(s) \,\, \textrm{for } s \neq 0,
\ees
where the inner product is defined by
\bes
\langle x, y \rangle_{X'} = \sigma ( \langle x, y \rangle_{X} ),
\ees
and the left and right actions are defined by
\bes
a \cdot x \cdot b := \sigma^{-1}(a) x \sigma^{-1}(b),
\ees
for $a,b \in \cM'$ and $x,y \in X'(s)$, $s \in \cS \setminus \{0\}$. Defining $T'_0 = id$ and $W_0 = \sigma$; and $T'_s = T_s$ for and $W_s = id$ for $s \in \cS \setminus \{0\}$, we have that $W$ is a morphism $X \rightarrow X'$ that sends $T$ to $T'$.

Assume, therefore, that $\cN = \cM'$ and that $\sigma$ is the identity representation.

We begin by defining for every $s \neq 0$
\bes
v_s: \cM \otimes_{\Theta_s} H \rightarrow X(s) \otimes H
\ees
by
\bes
v_s(a \otimes h) = (I_{X(s)} \otimes a)\widetilde{R}_s^*h.
\ees
We claim that for all $s \in \cS$ the map $v_s$ is a well-defined isometry. To see that, let $a,b \in \cM$ and $h,g \in H$, and compute:
\begin{align*}
\langle a \otimes_{\Theta_s} h,  b \otimes_{\Theta_s} g \rangle
&= \langle h , \Theta_s(a^* b)g \rangle \\
&= \langle h , \widetilde{R}_s(I_{X(s)} \otimes a^* b)\widetilde{R}_s^* g \rangle \\
&= \langle (I_{X(s)} \otimes a)\widetilde{R}_s^* h , (I_{X(s)} \otimes b)\widetilde{R}_s^* g \rangle.
\end{align*}
$[(I_{X(s)} \otimes \cM)\widetilde{R}_s^*H]$ is invariant under $I_{X(s)} \otimes \cM$, thus the projection onto the orthocomplement of this subspace is in $(I_{X(s)} \otimes \cM)' = \cL (X(s)) \otimes I_H$, so it has the form $q_s \otimes I_H$ for some projection $q_s \in \cL(X(s))$. In fact, $q_s$ is the orthogonal projection of $X(s)$ onto $\textrm{Ker} R_s$:
for all $g,h \in H$, $a \in \cM$,
\begin{align*}
\langle \xi \otimes h, (I \otimes a)\widetilde{R}_s^* g \rangle
&= \langle \widetilde{R}_s (\xi \otimes a^*h), g \rangle \\
&= \langle R_s (\xi) a^*h, g \rangle ,
\end{align*}
thus, $R_s(\xi) = 0$ if and only if $\xi \otimes h \in \left(\textrm{Im} v_s\right)^\perp$ for all $h \in H$, that is, $\xi \in q_s X(s)$.

By the definition of $v_s$ and by the covariance properties of $T$, we have for all $a \in \cM$ and $b \in \cM'$,
\bes
v_s(a \otimes I) = (I \otimes a)v_s \,\, , \,\, v_s (I \otimes b) = (b \otimes I)v_s .
\ees

Fix $s \in \cS$ and $x \in E(s)$. For all $\xi \in X(s), h \in H$, write
\bes
\psi(\xi)h = x^* v_s^*(\xi \otimes h) .
\ees
For $a \in \cM$ we have
\begin{align*}
\psi(\xi)a h &= x^*v_s^* (\xi \otimes ah) \\
&= x^*v_s^*(I \otimes a)(\xi \otimes h) \\
&= x^*(a \otimes I)v_s^*(\xi \otimes h) \\
&= a x^* v_s^*(\xi \otimes h) = a \psi(\xi)h.
\end{align*}
Thus the linear map $\xi \mapsto \psi(\xi)$ maps from $X(s)$ into $\cM'$ and it is apparent that this map is bounded. $\psi$
is also a right module map: for all $b \in \cM'$, $\psi(\xi b) h = x^* v^*(\xi b \otimes h) = x^* v^*(\xi  \otimes bh) = \psi(\xi) b h$. From the self duality of $X(s)$ it follows that there is a unique element in $X(s)$, which we denote by $V_s(x)$, such that for all $\xi \in X(s), h \in H$,
\be\label{eq:V_s}
\langle V_s(x), \xi \rangle h = x^* v_s^*(\xi \otimes h).
\ee
For $a,b \in \cM'$, $\xi \in X(s)$ and $h \in H$, we have
\begin{align*}
\langle V_s(a  x b), \xi \rangle h &= \langle V_s((I \otimes a) x b), \xi  \rangle h \\
&= b^* x^* (I \otimes a^*)v_s^* (\xi \otimes h) \\
&= b^* x^* v_s^*  (a^* \xi \otimes h) \\
&= b^* \langle V_s(x), a^* \xi \rangle h \\
&= \langle a V_s(x) b , \xi \rangle h .
\end{align*}
That is, $V_s(a  x b) = a V_s(x) b$. In a similar way one proves that $V_s : E(s) \rightarrow X(s)$ is linear.

Let us now show that $V_s$ preserves inner products. For all $\xi \in X(s)$, write $L_\xi$ for the operator $L_\xi : H \rightarrow X(s) \otimes H$ that maps $h$ to $\xi \otimes h$. One checks that $L_\xi^* (\eta \otimes h) = \langle \xi, \eta \rangle h$, so equation (\ref{eq:V_s}) becomes
\bes
L_{V_s(x)}^* L_\xi = x^* v_s^* L_\xi \,\, , \,\, \xi \in X(s),
\ees
or $L_{V_s(x)} = v_s x$, for all $x \in E(s)$. But since $v_s$ preserves inner products, we have for all $x,y \in E(s)$:
\bes
\langle x, y \rangle = x^* y = x^* v_s^* v_s y = L_{V_s(x)}^* L_{V_s(y)} = \langle V_s(x), V_s(y) \rangle.
\ees

We now prove that $V_s V_s^* = I_{X(s)} - q_s$.
$\xi \in \textrm{Im}V_s ^\perp$ if and only if $L_\xi^* v_s E(s) H = 0$. But by \cite[Lemma 2.10]{MS02},  $E(s)H = \cM \otimes_{\Theta_s} H$, thus $L_\xi^* v_s E(s) H = 0$ if and only if $\langle \xi , \eta \rangle = 0$ for all $ \eta \in (I_{X(s)} - q_s)X(s)$, which is the same as $\xi \in q_s X(s)$.

Fix $h,k \in H$. For $x \in E(s)$, we compute:
\begin{align*}
\langle T_s(x) h, k \rangle &= \langle W_{\Theta_s}^* x h, k \rangle \\
&= \langle x h, I \otimes_{\Theta_s} k \rangle \\
&= \langle v_s x h, v_s ( I \otimes_{\Theta_s} k) \rangle \\
&= \langle V_s(x) \otimes h, \widetilde{R}_s^* k \rangle \\
&= \langle R_s (V_s(x)) h, k \rangle ,
\end{align*}
thus $T_s = R_s \circ V_s$ for all $s\in \cS$. Define $W_s = V_s^*$. Then $T_s = R_s \circ W_s^*$. Multiplying both sides by $W_s$ we obtain $T_s \circ W_s = R_s \circ W_s^* W_s$. But $W_s^* W_s = I - q_s$ is the orthogonal projection onto $\left(\textrm{Ker}R_s\right)^\perp$, thus we obtain (\ref{eq:isosubrep}).

Finally, we need to show that $W = \{W_s \}$ respects the subproduct system structure: for all $s,t \in \cS$, $x \in X(s)$ and $y \in X(t)$, we must show that
\bes
W_{s+t}(U^X_{s,t}(x \otimes y)) = U^E_{s,t}(W_s(x) \otimes W_t(y)).
\ees
Since $T_{s+t}$ is injective, it is enough to show that after applying $T_{s+t}$ to both sides of the above equation we get the same thing. But $T_{s+t}$ applied to the left hand side gives
\bes
T_{s+t}W_{s+t}(U^X_{s,t}(x \otimes y)) = R_{s+t}(U^X_{s,t}(x \otimes y)) = R_s(x) R_t(y),
\ees
and $T_{s+t}$ applied to the right hand side gives
\bes
T_{s+t}(U^E_{s,t}(W_s(x) \otimes W_t(y))) = T_s(W_s(x)) T_t(W_t(y)) = R_s(x) R_t(y).
\ees
\end{proof}

\begin{corollary}\label{cor:onlyproductisorep}
Let $X$ be a subproduct system that has an isometric representation $V$ such that $V_0$ is faithful and nondegenerate. Then $X$ is a (full) product system.
\end{corollary}
\begin{proof}
Let $\Theta = \Sigma(X,V)$. Then $\Theta$ is an $e$-semigroup. Thus, if $(E,T) = \Xi(\Theta)$ is the identity representation of $\Theta$, then, by Remark \ref{rem:identityproduct}, $E$ is a (full) product system. But if $V_0$ is faithful and $V$ is isometric then $V$ is injective. By the above theorem, $X$ is isomorphic to $E$, so it is a product system.
\end{proof}

\subsection{Subproduct systems arise from $cp$-semigroups. The shift representation}\label{subsec:shift}

A question rises: \emph{does every subproduct system arise as the Arveson-Stinespring subproduct system associated with a $cp$-semigroup?} By Theorem \ref{thm:essentially_inverse}, this is equivalent to the question \emph{does every subproduct system have an injective representation?} We shall answer this question in the affirmative by constructing for every such subproduct system a canonical injective representation.

The following constructs will be of most interest when $\cS$ is a countable semigroup, such as $\mb{N}^k$.
\begin{definition}\label{def:shiftrep}
Let $X = \{X(s)\}_{s\in\cS}$ be a subproduct system. The \emph{$X$-Fock space} $\mathfrak{F}_X$ is defined as
\bes
\mathfrak{F}_X = \bigoplus_{s \in \cS} X(s).
\ees
The vector $\Omega := 1 \in \cN = X(0)$ is called \emph{the vacuum vector of $\mathfrak{F}_X$}.
The \emph{$X$-shift} representation of $X$ on $\mathfrak{F}_X$ is the representation
\bes
S^X: X \rightarrow B(\mathfrak{F}_X),
\ees
given by $S^X(x) y = U^X_{s,t}(x\otimes y)$, for all $x \in X(s), y \in X(t)$ and all $s,t \in \cS$.
\end{definition}
Strictly speaking, $S^X$ as defined above is not a representation because it represents $X$ on a C$^*$-correspondence rather
than on a Hilbert space. However, since for any C$^*$-correspondence $E$, $\cL(E)$ is a C$^*$-algebra, one can
compose a faithful representation $\pi : \cL(E) \rightarrow B(H)$ with $S^X$ to obtain a representation on a Hilbert space.

A direct computation shows that $\widetilde{S}^X_s : X(s) \otimes \mathfrak{F}_X \rightarrow \mathfrak{F}_X$ is a contraction, and also that $S^X(x)S^X(y) = S^X(U^X_{s,t}(x\otimes y))$ so $S^X$ is a completely contractive representation of $X$. $S^X$ is also injective because $S^X(x)\Omega = x$ for all $x \in X$. Thus,
\begin{corollary}
Every subproduct system is the Arveson-Stinespring subproduct system of a $cp$-semigroup.
\end{corollary}

\section{Subproduct system units and $cp$-semigroups}\label{sec:subunitsncp}

In this section, following Bhat and Skeide's constructions from \cite{BS00}, we show that subproduct systems and their units may also serve as a tool for studying $cp$-semigroups.

\begin{proposition}
Let $\cN$ be a von Neumann algebra and let $X$ be a subproduct system of $\cN$-correspondences over $\cS$, and let
$\xi = \{\xi_s\}_{s\in\cS}$ be a contractive unit of $X$, such that $\xi_0 = 1_\cN$. Then the family of maps
\be\label{eq:unitrep}
\Theta_s : a \mapsto \lel \xi_s, a \xi_s \rir,
\ee
is a semigroup of CP maps on $\cN$. Moreover, if $\xi$ is unital, then $\Theta_s$ is a unital map for all $s\in\cS$.
\end{proposition}
\begin{proof}
It is standard that $\Theta_s$ given by (\ref{eq:unitrep}) is a contractive completely positive map on $\cN$, which is unital if and only if $\xi$ is unital. The fact that $\Theta_s$ is normal goes a little bit deeper, but is also known (one may use \cite[Remark 2.4(i)]{MS02}).

We show that $\{\Theta_s\}_{s \in \cS}$ is a semigroup. It is clear that $\Theta_0(a) = a$ for all $a \in \cN$. For all $s,t \in \cS$,
\begin{align*}
\Theta_s(\Theta_t(a)) &= \lel \xi_s, \lel \xi_t, a \xi_t \rir \xi_s \rir \\
&= \lel \xi_t \otimes \xi_s, , a \xi_t \otimes \xi_s \rir \\
&= \lel U_{t,s}^* \xi_{s+t}, , a U_{t,s}^* \xi_{s+t} \rir \\
&= \lel \xi_{s+t} , a \xi_{s+t} \rir \\
&= \Theta_{s+t}(a).
\end{align*}
\end{proof}

We recall a central construction in Bhat and Skeide's approach to dilation of $cp$-semigroup \cite{BS00}, that goes back to Paschke \cite{Pas}.
Let $\cM$ be a $W^*$-algebra, and let $\Theta$ be a normal completely positive map on $\cM$ \footnote{The construction works also for completely positive maps on C$^*$-algebras, but in Theorem \ref{thm:GNS} below we will need to work with normal maps on W$^*$-algebras.}. \emph{The GNS representation
of $\Theta$} is a pair $(F_\Theta,\xi_\Theta)$ consisting of a Hilbert $W^*$-correspondence $F_\Theta$ and
a vector $\xi_\Theta \in F_\Theta$ such that
\bes
\Theta(a) = \lel \xi_\Theta, a \xi_\Theta \rir \,\, \textrm{ for all } a \in \cM.
\ees
$F_\Theta$ is defined to be the correspondence $\cM \otimes_\Theta \cM$ - which is the self-dual extension of the Hausdorff completion of the algebraic tensor product $\cM \otimes \cM$ with respect to inner product
\bes
\lel a \otimes b, c \otimes d \rir = b^* \Theta(a^* c)d .
\ees
$\xi_\Theta$ is defined to be $\xi_\Theta = 1 \otimes 1$. Note that
$\xi_\Theta$ is a unit vector, that is - $\lel \xi_\Theta,\xi_\Theta \rir = 1$, if and only if $\Theta$ is unital.

\begin{theorem}\label{thm:GNS}
Let $\Theta = \{\Theta_s \}_{s\in\cS}$ be a $cp$-semigroup on a $W^*$-algebra $\cM$. For every $s \in \cS$ let
$(F(s),\xi_{s})$ be the GNS representation of $\Theta_s$. Then
$F = \{F(s)\}_{s\in\cS}$ is a subproduct system of $\cM$-correspondences, and $\xi = \{\xi_s\}_{s \in \cS}$ is a generating contractive unit for $F$ that gives back $\Theta$ by the formula
\be\label{eq:unitRep}
\Theta_s(a) = \lel \xi_s, a \xi_s \rir \,\, \textrm{ for all } a \in \cM.
\ee
$\Theta$ is a cp$_0$-semigroup if and only if $\xi$ is a unital unit.
\end{theorem}
\begin{proof}
For all $s,t \in \cS$ define a map $V_{s,t}:F(s+t) \rightarrow F(s) \otimes F(t)$
by sending $\xi_{s+t}$ to $\xi_s \otimes \xi_t$ and extending to a bimodule map. Because
\begin{align*}
\lel a \xi_s \otimes \xi_t b, c \xi_s \otimes \xi_t d \rir
&= \lel  \xi_t b, \lel a \xi_s , c \xi_s \rir \xi_t d \rir \\
&= \lel  \xi_t b, \Theta_s(a^* c) \xi_t d \rir \\
&= b^* \lel  \xi_t, \Theta_s(a^* c) \xi_t \rir d \\
&= b^* \Theta_{t+s}(a^* c) d \\
&= \lel a \xi_{t+s} b, c \xi_{t+s} d \rir ,
\end{align*}
$V_{s,t}$ extends to a well defined isometric bimodule map from $F(s+t)$ into $F(s) \otimes F(t)$. We define the map $U_{s,t}$ to be the adjoint of $V_{s,t}$ (here it is important that we are working with $W^*$ algebras - in general, an isometry from one
Hilbert $C^*$-module into another need not be adjointable, but bounded module maps between \emph{self-dual} Hilbert modules are always adjointable, \cite[Proposition 3.4]{Pas}). The collection $\{U_{s,t}\}_{s,t \in \cS}$ makes $F$ into a subproduct system. Indeed, these maps are coisometric by definition, and they compose in an associative manner. To see the latter, we check that
$(I_{F(r)} \otimes V_{s,t})V_{r,s+t} = (V_{r,s} \otimes I_{F(t)})V_{r+s,t}$ and take adjoints.
\begin{align*}
(I_{F(r)} \otimes V_{s,t})V_{r,s+t} (a \xi_{r+s+t} b)
&= (I_{F(r)} \otimes V_{s,t}) (a \xi_r \otimes \xi_{s+t} b) \\
&= a \xi_r \otimes \xi_{s} \otimes \xi_t b .
\end{align*}
Similarly, $(V_{r,s} \otimes I_{F(t)})V_{r+s,t} (a \xi_{r+s+t} b) = a \xi_r \otimes \xi_{s} \otimes \xi_t b$. Since
$F(r+s+t)$ is spanned by linear combinations of elements of the form $a \xi_{r+s+t} b$, the $U$'s make $F$
into a subproduct system, and $\xi$ is certainly a unit for $F$.
Equation (\ref{eq:unitRep}) follows by definition of the GNS representation. Now,
\bes
\lel \xi_s, \xi_s \rir = \Theta_s(1) \,\, , \,\, s \in \cS,
\ees
so $\xi$ is a contractive unit because $\Theta_s(1) \leq 1$, and $\xi$ it is unital if and only if $\Theta_s$ is unital
for all $s$.
$\xi$ is in fact more then just a generating unit,
as $F(s)$ is spanned by elements with the form described in equation (\ref{eq:spanning}) with $(s_1, \ldots, s_n)$ a \emph{fixed} $n$-tuple such that
$s_1 + \cdots + s_n = s$.
\end{proof}

\begin{definition}
Given a $cp$-semigroup $\Theta$ on a $W^*$ algebra $\cM$, the subproduct system $F$ and the unit $\xi$
constructed in Theorem \ref{thm:GNS} are called, respectively, \emph{the GNS subproduct system} and \emph{the GNS unit} of $\Theta$. The pair $(F,\xi)$ is called \emph{the GNS representation} of $\Theta$.
\end{definition}

\begin{remark}
\emph{There is a precise relationship between the identity representation (Definition \ref{def:identityrep}) and the GNS representation of a $cp$-semigroup. The GNS
representation of a CP map is the \emph{dual} of the identity representation in a sense that is briefly
described in
\cite{MS05}. This notion of duality has been used
to move from the product-system-and-representation picture to the product-system-with-unit picture, and vice versa. See for example
\cite{Skeide06} and the references therein. It is more-or-less straightforward to use this duality to get Theorem \ref{thm:GNS}
from Theorem \ref{thm:reprep} (or the other way around)}.
\end{remark}

\section{$*$-automorphic dilation of an $e_0$-semigroup}\label{sec:aut_dil}

We now apply some of the tools developed above to dilate an e$_0$-semigroup to a semigroup of $*$-automorphisms. We shall need
the following proposition, which is a modification (suited for \emph{sub}product systems)
of the method introduced in \cite{Shalit07a} for representing a product system representation as a semigroup of
contractive operators on a Hilbert space.

\begin{proposition}\label{prop:technology}
Let $\cN$ be a von Neumann algebra
and let $X$ be a subproduct system of unital \footnote{An $\cN$-correspondence
is said to be \emph{unital} if the left action of $\cN$ is
unital. The right action of every unital $C^*$-algebra
on every Hilbert $C^*$-correspondence is unital.}
$W^*$-correspondences over $\cS$. Let $(\sigma,T)$ be a fully coisometric covariant
representation of $X$ on the Hilbert space $H$, and assume that
$\sigma$ is unital. Denote
$$H_s := \big(X(s) \otimes_{\sigma} H \big) \Big/ {\rm Ker}\widetilde{T}_s $$
and
$$\cH = \bigoplus_{s \in \cS} H_s.$$
Then there exists a semigroup of coisometries $\hat{T} = \{\hat{T}_s\}_{s\in\cS}$ on $\cH$ such that for all
$s\in\cS$, $x \in X(s)$ and $h\in H$,
\bes
\hat{T}_s \left(\delta_s \cdot x \otimes h \right) = T_s(x)h .
\ees
$\hat{T}$ also has the property that for all $s \in \cS$ and all $t \geq s$
\be\label{eq:almostIs}
\hat{T}^*_s\hat{T}_s\big|_{H_t} = I_{H_t} \quad , \quad (t \geq s).
\ee
\end{proposition}

\begin{proof}
First, we note that the assumptions on $\sigma$ and on the left action of $\cN$ imply that $H_0 \cong H$ via the identification $a \otimes h
\leftrightarrow \sigma(a)h$. This identification will be made repeatedly
below.

Define $\cH_0$ to be the space of all finitely supported functions
$f$ on $\cS$ such that for all $s \in \cS$,
$$f(s) \in H_s.$$
We equip $\cH_0$
with the inner product
$$\langle \delta_s \cdot \xi, \delta_t \cdot \eta \rangle = \delta_{s,t} \langle \xi, \eta \rangle  ,$$
for all $s,t \in \cS, \xi \in H_s, \eta \in
H_t$. Let $\cH$ be the completion of $\cH_0$ with
respect to this inner product. We have
$$\cH \cong \bigoplus_{s \in \cS} H_s .$$
It will sometimes be convenient to identify the subspace $\delta_s \cdot H_s \subseteq \cH$ with $H_s$, and for $s = 0$
this gives us an inclusion $H \subseteq \cH$.
We
define a family $\hat{T} = \{\hat{T}_s\}_{s \in \cS}$ of operators
on $\cH_0$ as follows. First, we define
$\hat{T}_0$ to be the identity. Now assume that $s>0$. If $t\in \cS$ and $t \ngeq s$, then we define $\hat{T}_s (\delta_t \cdot \xi ) = 0$ for all
$\xi \in H_t$. If $t
\geq s > 0$ we would like to define (as we did in \cite{Shalit07a})
\be\label{eq:That}
\hat{T}_s \left(\delta_t \cdot (x_{t-s}
\otimes x_s \otimes h) \right) = \delta_{t-s} \cdot
\left(x_{t-s}\otimes \widetilde{T}_s (x_s \otimes h) \right),
\ee
but since $X$ is not a true product system, we cannot identify $X(t-s) \otimes X(s)$ with $X(t)$.
For a fixed $t>0$, we define for all $s\leq t$, $\xi \in X(t)$ and $h \in H$
\bes
\check{T}_s \left(\delta_t \cdot (\xi \otimes h) \right) = \delta_{t-s} \cdot \left((I_{X(t-s)} \otimes \tT_s)(U^*_{t-s,s}\xi \otimes h) \right) .
\ees
$\check{T}_s$ can be extended to a well defined contraction from $X(t) \otimes H$ to $X(t-s) \otimes H$, for all $t\geq s$, and has an adjoint given by
\be\label{eq:Tstar}
\check{T}^*_s \delta_{t-s} \cdot \eta \otimes h = \delta_{t}\cdot \left((U_{t-s,s} \otimes I_H) (\eta \otimes \widetilde{T}^*_s h)\right) .
\ee
We are going to obtain $\hat{T}_s$ as the map $H_t \rightarrow H_{t-s}$ induced by $\check{T}_s$. Let $Y  \in H_t$ satisfy $\tT_t (Y) = 0$. We shall show
that $\check{T}_s \delta_t \cdot Y = 0$ in $\delta_{t-s} \cdot H_{t-s}$. But
\bes
\check{T}_s \delta_t \cdot Y = \delta_{t-s} \cdot \left((I_{X(t-s)} \otimes \tT_s)(U^*_{t-s,s}\otimes I_H)Y \right),
\ees
and
\begin{align*}
\tT_{t-s}\left((I_{X(t-s)} \otimes \tT_s)(U^*_{t-s,s}\otimes I_H)Y \right)
(*)&= \tT_{t}(U_{t-s,s} \otimes I_H)(U^*_{t-s,s}\otimes I_H)Y \\
(**)&= \tT_t(Y) = 0,
\end{align*}
where the equation marked by (*) follows from the fact that $T$ is a representation of subproduct systems, and the
one marked by (**) follows from the fact that $U_{t-s,s}$ is a coisometry. Thus, for all $s,t \in \cS$,
$$\check{T}_s \left(\delta_t \cdot {\rm Ker}\tT_t\right) \subseteq \delta_{t-s}\cdot {\rm Ker}\tT_{t-s} ,$$
thus $\check{T}_s$ induces a well defined contraction $\hat{T}_s$ on $\cH$ given by
\be\label{eq:defThat}
\hat{T}_s \left(\delta_t \cdot (\xi \otimes h) \right) = \delta_{t-s} \cdot \left((I_{X(t-s)} \otimes \tT_s)(U^*_{t-s,s}\xi \otimes h) \right) ,
\ee
where $\xi \otimes h$ and $(I_{X(t-s)} \otimes \tT_s)(U^*_{t-s,s}\xi \otimes h)$ stand for these elements' equivalence classes in $\big(X(t) \otimes H \big) \big/{\rm Ker}\widetilde{T}_{t}$ and
$\big(X(t-s) \otimes H \big) \big/{\rm Ker}\widetilde{T}_{t-s}$, respectively.
It follows that we have the following, more precise, variant of (\ref{eq:That}):
\bes
\hat{T}_s \left(\delta_t \cdot \left(U_{t-s,s}(x_{t-s}
\otimes x_s) \otimes h\right) \right) = \delta_{t-s} \cdot
\left(x_{t-s}\otimes \widetilde{T}_s (x_s \otimes h) \right) .
\ees
In particular,
$$\hat{T}_s \left(\delta_s \cdot x_s \otimes h \right) = T_s(x_s)h ,$$
for all $s\in \cS, x_s \in X(s), h \in H $.

It will be very helpful to have a formula for $\hat{T}_s^*$ as well. Assume that $\sum_i \xi_i \otimes h_i \in {\rm Ker}\tT_t$.
\bes
\check{T}^*_s\left(\delta_{t} \cdot \sum_i \xi_i \otimes h_i \right)
= \delta_{s+t} \cdot \left((U_{t,s} \otimes I_H) (\sum_i \xi_i \otimes \tT_s^* h_i) \right) ,
\ees
and applying $\tT_{s+t}$ to the right hand side (without the $\delta$) we get
\begin{align*}
\tT_{s+t}\left((U_{t,s} \otimes I_H) (\sum_i \xi_i \otimes \tT_s^* h_i) \right)
&= \tT_{t}(I_{X(t)} \otimes \tT_{s})(\sum_i \xi_i \otimes \tT_s^* h_i) \\
&= \tT_t (\sum_i \xi_i \otimes \tT_s \tT_s^* h_i) \\
&= \tT_t (\sum_i \xi_i \otimes h_i) = 0,
\end{align*}
because $T$ is a fully coisometric representation. So
$$\check{T}^*_s \left(\delta_t \cdot {\rm Ker}\tT_t\right) \subseteq \delta_{s+t} \cdot {\rm Ker}\tT_{s+t} ,$$
and this means that $\check{T}_s^*$ induces on $\cH$ a well defined contraction which is equal to $\hat{T}^*_s$, and is given by the formula (\ref{eq:Tstar}).

We now show that $\hat{T}$ is a semigroup. Let $s,t,u \in \cS$. If
either $s = 0$ or $t = 0$ then it is clear that the semigroup
property $\hat{T}_s \hat{T}_t = \hat{T}_{s+t}$ holds. Assume that
$s,t >0$. If $u \ngeq s+t$, then both $\hat{T}_{s} \hat{T}_{t}$
and $\hat{T}_{s + t}$ annihilate $\delta_u \cdot \xi$, for all
$\xi \in H_u$. Assuming $u \geq s+t$, we shall show that $\hat{T}_s \hat{T}_t$ and
$\hat{T}_{s+t}$ agree on elements of the form
$$Z = \delta_u \cdot \big(U_{u-t,t}(U_{u-t-s,s} \otimes I)(x_{u-s-t} \otimes x_s \otimes x_t)\big) \otimes h,$$
and since the set of all such elements is total in $H_u$, this will establish the semigroup property.
\begin{align*}
\hat{T}_{s} \hat{T}_{t} Z
&=  \hat{T}_{s} \left(\delta_{u-t}
\left( U_{u-t-s,s}(x_{u-s-t} \otimes x_s) \otimes \widetilde{T}_t(x_t \otimes h)\right) \right)
\\
&= \delta_{u-s-t} \left(x_{u-s-t} \otimes \widetilde{T}_s(x_s \otimes \widetilde{T}_t( x_t \otimes h)) \right) \\
&= \delta_{u-s-t} \left(x_{u-s-t} \otimes \widetilde{T}_s(I
\otimes \widetilde{T}_t )(x_s \otimes x_t \otimes h) \right) \\
&=  \delta_{u-s-t} \left(x_{u-s-t} \otimes \widetilde{T}_{s+t} \left(U_{s,t} (x_s
\otimes
x_t) \otimes h\right) \right) \\
&=  \hat{T}_{t+s} \delta_u \cdot \left(U_{u-t-s,t+s}\left(x_{u-s-t} \otimes U_{s,t} (x_s
\otimes
x_t)\right) \otimes h  \right) \\
&= \hat{T}_{t+s} Z.
\end{align*}
The final equality follows from the associativity condition (\ref{eq:assoc_prod}).

To see that $\hat{T}$ is a semigroup of coisometries, we take $\xi \in X(t), h \in H$, and compute
\begin{align*}
\tT_t \left(\hat{T}_s \hat{T}^*_s \delta_t \cdot (\xi \otimes h) \right)
&= \tT_t \left((I_{X(t)} \otimes \tT_s)(U_{t,s}^* \otimes I_H) (U_{t,s} \otimes I_H)(I_{X(t)} \otimes \tT_s^*) (\xi \otimes h) \right) \\
&= \tT_{s+t} (U_{t,s} \otimes I_H)(I_{X(t)} \otimes \tT_s^*) (\xi \otimes h) \\
&= \tT_t (\xi \otimes \tT_s \tT_s^* h) = \tT_t (\xi \otimes h) ,
\end{align*}
so $\hat{T}_s \hat{T}^*_s$ is the identity on $H_t$ for all $t \in \cS$, thus $\hat{T}_s \hat{T}^*_s = I_{\cH}$. Equation (\ref{eq:almostIs}) follows by a similar computation, which is a omitted.
\end{proof}

We can now obtain a $*$-automorphic dilation for any e$_0$-semigroup over any subsemigroup of $\Rpk$. 
The following result  should be compared with similar-looking results of Arveson-Kishimoto \cite{ArvKish}, Laca \cite{Laca}, Skeide \cite{Skeide07},
and Arveson-Courtney \cite{ArvCourt} (none of these cited results is strictly stronger or weaker than the result we obtain for the case of $e_0$-semigroups).

\begin{theorem}\label{thm:aut_dil}
Let $\Theta$ be a $e_0$-semigroup acting on a von Neumann algebra $\cM$. Then $\Theta$ can be dilated to a semigroup
of $*$-automorphisms in the
following sense: there is a Hilbert space $\cK$, an orthogonal projection $p$ of $\cK$ onto a subspace $\cH$ of $\cK$, a normal, faithful representation $\varphi:\cM \rightarrow B(\cK)$ such that $\varphi(1) = p$, and a semigroup $\alpha = \{\alpha_s\}_{s\in\cS}$
of $*$-automorphisms on $B(\cK)$ such that for all $a \in \cM$ and all $s \in \cS$
\be\label{eq:aut_ext}
\alpha_s(\varphi(a)) \big|_\cH = \varphi(\Theta_s(a)) ,
\ee
so, in particular,
\be\label{eq:aut_dil}
p \alpha_s(\varphi(a)) p = \varphi(\Theta_s(a)) .
\ee
The projection $p$ is increasing for $\alpha$, in the sense that for all $s \in \cS$,
\be\label{eq:increasing}
\alpha_s(p) \geq p .
\ee
\end{theorem}
\begin{remark}
\emph{
Another way of phrasing the above theorem is by using the terminology of ``weak Markov flows", as used in \cite{BS00}. Denoting $\varphi$ by $j_0$, and defining $j_s:= \alpha_s \circ
j_0$, we have that $(B(\cK),j)$ is a weak Markov flow for $\Theta$ on $\cK$, which just means that for all $t \leq s \in \cS$ and all $a \in \cM$,
\be\label{eq:weak}
j_t(1)j_s(a)j_t(1) = j_t(\Theta_{s-t}(a)) .
\ee
Equation (\ref{eq:weak}) for $t=0$ is just (\ref{eq:aut_dil}), and the case
$t \geq 0$ follows from the case $t=0$.}
\end{remark}
\begin{remark}
\emph{
The assumption that $\Theta$ is a \emph{unital} semigroup is essential, since (\ref{eq:aut_dil}) and (\ref{eq:increasing})
imply that $\Theta(1) = 1$.}
\end{remark}
\begin{remark}
\emph{
It is impossible, in the generality we are working in, to hope for a semigroup of automorphisms that extends $\Theta$ in the sense that
\be\label{eq:aut_extension}
\alpha_s(\varphi(a)) = \varphi(\Theta_s(a)) ,
\ee
because that would imply that $\Theta$ is injective.}
\end{remark}
\begin{proof}
Let $(E,T)$ be the identity representation of $\Theta$. Since $\Theta$ preserves the unit, $T$ is a fully coisometric
representation. Let $\hat{T}$ and $\cH$ be the semigroup and Hilbert
space representing $T$ as described in Proposition \ref{prop:technology}. $\{\hat{T}^*_s\}_{s\in\cS}$
is a commutative semigroup of isometries. By a theorem of Douglas \cite{Douglas}, $\{\hat{T}^*_s\}_{s\in\cS}$
can be extended to a semigroup $\{\hat{V}^*_s\}_{s\in\cS}$ of unitaries acting on a space $\cK \supseteq \cH$. We obtain a semigroup of unitaries
$V = \{\hat{V}_s\}_{s\in\cS}$ that is a dilation of $\hat{T}$, that is
\bes
P_\cH \hat{V}_s \big|_\cH = \hat{T}_s  \,\, , \,\, s \in \cS.
\ees

For any $b\in B(\cK)$, and any $s\in\cS$, we define
\bes
\alpha_s(b) = \hat{V}_s b \hat{V}_s^*.
\ees
Clearly, $\alpha = \{\alpha_s\}_{s\in\cS}$ is a semigroup of $*$-automorphisms.

Put $p = P_{\cH}$, the orthogonal projection of $\cK$ onto $\cH$.
Define $\varphi : \cM \rightarrow B(\cK)$ by $\varphi(a) = p (I \otimes a) p$, where $I \otimes a : \cH \rightarrow \cH$ is given by
\bes
(I \otimes a) \delta_t\cdot x \otimes h = \delta_t \cdot x \otimes a h  \quad , \quad x \otimes h \in E(t) \otimes H.
\ees

$\varphi$ is well defined because $T$ is an isometric representation (so ${\rm Ker}\widetilde{T}_{t}$ is always zero). We have that $\varphi$ is a faithful, normal $*$-representation (the fact that $T_0$ is the identity representation ensures that $\varphi$ is faithful). It is clear that $\varphi(1) = p$.

To see (\ref{eq:increasing}), we note that since $\hat{V}_s^*$ is an extension of $\hat{T}_s^*$, we have
$\hat{T}_s^* = \hat{V}_s^* p = p\hat{V}_s^* p$, thus
\begin{align*}
p \alpha_s(p) p &= p \hat{V}_s p \hat{V}_s^* p \\
&= p \hat{V}_s \hat{V}_s^* p \\
&= p,
\end{align*}
that is, $p \alpha_s(p) p = p$, which implies that $\alpha_s(p) \geq p$.

We now prove (\ref{eq:aut_dil}).
Let $\delta_t \cdot x \otimes h$ be a typical element of $\cH$. We compute
\begin{align*}
p \alpha_s (\varphi(a)) p \delta_t \cdot x \otimes h &= p \hat{V}_s p (I \otimes a) p \hat{V}_s^* p \delta_t \cdot x \otimes h \\
&= \hat{T}_s (I \otimes a) \hat{T}_s^* \delta_t \cdot x \otimes h \\
&= \hat{T}_s (I\otimes a) \delta_{s+t}\cdot (U_{t,s} \otimes I_H) \Big( x \otimes \widetilde{T}_s^*h \Big) \\
&= \hat{T}_s \delta_{s+t}\cdot (U_{t,s} \otimes I_H) \Big( x \otimes (I\otimes a) \widetilde{T}_s^*h \Big) \\
&= \delta_{t} \cdot x  \otimes \left(\widetilde{T}_s(I\otimes a) \widetilde{T}_s^*h\right) \\
&= \delta_{t} \cdot x \otimes (\Theta_s(a) h) \\
&= \varphi(\Theta_s(a)) \delta_t \cdot x \otimes h .
\end{align*}
Since both $p\alpha_s(\varphi(a)) p$ and $\varphi(\Theta_s(a))$ annihilate $\cK \ominus \cH$, we have
(\ref{eq:aut_dil}).

To prove (\ref{eq:aut_ext}), it just remains to show that
\bes
p \alpha_s (\varphi(a)) \big|_{\cH} =  \alpha_s (\varphi(a)) \big|_{\cH},
\ees
that is, that $\alpha_s (\varphi(a)) \cH \subseteq \cH$.
Now, $\hat{V}_s^*$ is an extension of $\hat{T}_s^*$. Moreover (\ref{eq:almostIs}) shows that
if $\xi \in H_u$ with $u \geq s$, then $\| \hat{T}_s(\xi) \| = \|\xi\|$. Thus
\bes
\|\xi\|^2 = \|\hat{V}_s \xi \|^2 = \|P_\cH \hat{V}_s \xi \|^2 + \|(I_\cK - P_\cH) \hat{V}_s \xi \|^2 = \|\hat{T}_s \xi \|^2 + \|(I_\cK - P_\cH) \hat{V}_s \xi \|^2.
\ees
So $\hat{V}_s \xi = \hat{T}_s \xi$ for $\xi \in H_u$ with $u \geq s$. Now, for a typical element $\delta_t \cdot x \otimes h$ in $H_t$, $t\in \cS$, we have
\begin{align*}
\alpha_s (\varphi(a)) \delta_t \cdot x \otimes h &=
\hat{V}_s (I \otimes a) \hat{V}_s^* \delta_t \cdot x \otimes h \\
&= \hat{V}_s (I \otimes a) \hat{T}_s^* \delta_t \cdot x \otimes h \\
&= \hat{V}_s \delta_{s+t} \cdot (U_{s,t} \otimes I_H) \Big( x \otimes (I \otimes a) \tT_s^* h \Big) \\
&= \hat{T}_s \delta_{s+t} \cdot (U_{s,t} \otimes I_H) \Big( x \otimes (I \otimes a) \tT_s^* h \Big) \in \cH,
\end{align*}
because $\delta_{s+t} \cdot x \otimes (I \otimes a) \tT_s^* h \in H_{s+t}$, and $s+t \geq s$.
\end{proof}

\section{Dilations and pieces of subproduct system representations}\label{sec:dil}
\subsection{Dilations and pieces of subproduct system representations}

\begin{definition}
Let $X$ and $Y$ be subproduct systems of $\cM$ correspondences ($\cM$ a W$^*$-algebra) over the same semigroup $\cS$. Denote by $U_{s,t}^X$ and $U_{s,t}^Y$ the coisometric maps that make $X$ and $Y$, respectively, into subproduct systems.
$X$ is said to be a \emph{subproduct subsystem of $Y$} (or simply a \emph{subsystem of $Y$} for short) if for all $s \in \cS$ the space $X(s)$ is a closed subspace of $Y(s)$, and if the orthogonal projections $p_s : Y(s) \rightarrow X(s)$ are bimodule maps that satisfy
\be\label{eq:pU}
p_{s+t} \circ U_{s,t}^Y = U_{s,t}^X \circ (p_s \otimes p_t) \,\, , \,\, s,t \in \cS.
\ee
\end{definition}

One checks that if $X$ is a subproduct subsystem of $Y$ then
\be\label{eq:p}
p_{s+t+u} \circ U^Y_{s,t+u}(I \otimes (p_{t+u}\circ U^Y_{t,u})) = p_{s+t+u} \circ U^Y_{s+t,u}((p_{s+t}\circ U^Y_{s,t})\otimes I) ,
\ee
for all $s,t,u \in \cS$. Conversely, given a subproduct system $Y$ and a family of orthogonal projections $\{p_s\}_{s \in \cS}$ that are bimodule maps satisfying (\ref{eq:p}), then by defining $X(s) = p_s Y(s)$ and
$U_{s,t}^X = p_{s+t} \circ U_{s,t}^Y$ one obtains a subproduct subsystem $X$ of $Y$ (with (\ref{eq:pU}) satisfied).

The following proposition is a consequence of the definitions.
\begin{proposition}
There exists a morphism $X \rightarrow Y$ if and only if $Y$ is isomorphic to a subproduct subsystem of $X$.
\end{proposition}
\begin{remark}\label{rem:subsystem_iso}
\emph{In the notation of Theorem \ref{thm:essentially_inverse}, we may now say that
given a subproduct system $X$ and a representation $R$ of $X$, then the Arveson-Stinespring subproduct system $E$ of $\Theta = \Sigma(X,R)$ is isomorphic to a subproduct subsystem of $X$.}
\end{remark}

The following definitions are inspired by the work of Bhat, Bhattacharyya and Dey \cite{BBD03}.

\begin{definition}\label{def:dilation}
Let $X$ and $Y$ be subproduct systems of W$^*$-correspondences (over the same W$^*$-algebra $\cM$) over $\cS$, and let $T$ be a representation of $Y$ on a Hilbert space $K$. Let $H$ be some fixed Hilbert space, and let $S=\{S_s\}_{s\in\cS}$ be a family of maps $S_s : X(s) \rightarrow B(H)$. $(Y,T,K)$ is called a \emph{dilation} of $(X,S,H)$ if
\begin{enumerate}
\item $X$ is a subsystem of $Y$,
\item $H$ is a subspace of $K$, and
\item for all $s \in \cS$,  $\widetilde{T}^*_s H \subseteq X(s) \otimes H$ and $\widetilde{T}^*_s \big|_H = \widetilde{S}^*_s$.
\end{enumerate}
In this case we say that $S$ is \emph{an $X$-piece of $T$}, or simply a \emph{piece of $T$}. $T$ is said to be an \emph{isometric} dilation of $S$ of $T$ is an isometric representation.
\end{definition}
The third item can be replaced by the three conditions
\begin{itemize}
\item[1'] $T_0(\cdot) P_H = P_H T_0(\cdot) P_H = S_0 (\cdot)$,
\item[2'] $P_H \tT_s \big|_{X(s) \otimes H} = \widetilde{S}_s$ for all $s\in \cS$, and
\item[3'] $P_H \tT_s \big|_{Y(s) \otimes K \ominus X(s) \otimes H} = 0$.
\end{itemize}
So our definition of dilation is identical to Muhly and Solel's definition of dilation of representations when $X = Y$ is a product system \cite[Theorem and Definition 3.7]{MS02}.

\begin{proposition}\label{prop:pieceisrep}
Let $T$ be a representation of $Y$, let $X$ be a subproduct subsystem of $Y$, and let $S$ an $X$-piece of $T$. Then $S$ is a representation of $X$.
\end{proposition}
\begin{proof}
$S$ is a completely contractive linear map as the compression of a completely contractive linear map. Item 1' above together with the coinvariance of $T$ imply that $S$ is coinvariant:
if $a,b \in \cM$ and $x \in X(s)$, then
\begin{align*}
S_s(axb) = P_H T_s(axb) P_H
&= P_H T_0(a) T_s(x) T_0(b) P_H \\
&= P_H T_0(a) P_H T_s(x) P_H T_0(b) P_H \\
&= S_0(a) S_s(x) S_0(b).
\end{align*}
Finally, (using Item 3' above),
\begin{align*}
S_{s+t} (U^X_{s,t}(x \otimes y)) h &= S_{s+t} (p_{s+t} U^Y_{s,t}(x \otimes y)) h \\
&= \widetilde{S}_{s+t} (p_{s+t} U^Y_{s,t}(x \otimes y) \otimes h) \\
&= P_H \widetilde{T}_{s+t} (U^Y_{s,t}(x \otimes y) \otimes h) \\
&= P_H T_s (x) T_t(y) h \\
&= P_H T_s (x) P_H T_t(y) h \\
&= S_s (x) S_t(y) h .
\end{align*}
\end{proof}

\begin{example}\emph{
Let $E$ be a Hilbert space of dimension $d$, and let $X$ be the symmetric subproduct system constructed in Example \ref{expl:symm}. Fix an orthonormal basis $\{e_1, \ldots, e_n\}$ of $E$. There is a one-to-one correspondence between c.c. representations $S$ of $X$ (on some $H$) and commuting row contractions $(S_1, \ldots, S_d)$ (of operators on $H$), given by
\bes
S \leftrightarrow \underline{S} = (S(e_1), \ldots, S(e_d)).
\ees
If $Y$ is the full product system over $E$, then any dilation $(Y,T,K)$ gives rise to a tuple $\underline{T} = (T(e_1), \ldots, T(e_d))$ that is a dilation of $\underline{S}$ in the sense of \cite{BBD03}, and \emph{vice versa}.
Moreover, $\underline{S}$ is then a commuting piece of $\underline{T}$ in the sense of \cite{BBD03}.}
\end{example}

Consider a subproduct system $Y$ and a representation $T$ of $Y$ on $K$. Let $X$ be some subproduct subsystem of $Y$. Define the following set of subspaces of $K$:
\be\label{eq:cPXT}
\cP(X,T) = \{H \subseteq K: \widetilde{T}^*_s H \subseteq X(s) \otimes H\ \textrm{ for all }s \in \cS\}.
\ee
As in \cite{BBD03}, we observe that $\cP(X,T)$ is closed under closed linear spans (and intersections), thus we may define
\bes
K^X(T) = \bigvee_{H \in \cP(X,T)} H.
\ees
$K^X(T)$ is the maximal element of $\cP(X,T)$.
\begin{definition}\label{def:Xpiece}
The representation $T^X$ of $X$ on $K^X(T)$ given by
\bes
T^X(x) h = P_{K^X(T)} T(x) h,
\ees
for $x \in X(s)$ and $h \in K^X(T)$, is called
\emph{the maximal $X$-piece of $T$}.
\end{definition}
By Proposition \ref{prop:pieceisrep}, $T^X$ is indeed a representation of $X$.

\subsection{Consequences in dilation theory of $cp$-semigroups}

\begin{proposition}\label{prop:dil_rep_dil_CP}
Let $X$ and $Y$ be subproduct systems of W$^*$-correspondences (over the same W$^*$-algebra $\cM$) over $\cS$, and let $S$ and $T$ be representations of $X$ on $H$ and
of $Y$ on $K$, respectively.
Assume that $(Y,T,K)$ is a dilation of $(X,S,H)$. Then the $cp$-semigroup $\Theta$ acting on $V_0(\cM)'$, given by
\bes
\Theta_s(a) = \widetilde{T}_s(I_{Y(s)} \otimes a)\widetilde{T}_s^* \,\, ,\,\, a \in V_0(\cM)',
\ees
is a dilation of the $cp$-semigroup $\Phi$ acting on $T_0(\cM)'$ given by
\bes
\Phi_s(a) = \widetilde{S}_s(I_{X(s)} \otimes a)\widetilde{S}_s^* \,\, ,\,\, a \in T_0(\cM)',
\ees
in the sense that for all $b \in V_0(\cM)'$ and all $s \in \cS$,
\bes
\Phi_s(P_H b P_H) = P_H \Theta_s (b) P_H .
\ees
\end{proposition}
\begin{proof}
This follows from the definitions.
\end{proof}

Although the above proposition follows immediately from the definitions,
we hope that it will prove to be important in the theory of dilations of $cp$-semigroups, because it points to
a conceptually new way of constructing dilations of $cp$-semigroups, as the following proposition and corollary illustrate.

\begin{proposition}\label{prop:exist_isodil}
Let $X=\{X(s)\}_{s\in\cS}$ be a subproduct system, and let $S$ be a fully coisometric representation of $X$ on $H$
such that $S_0$ is unital. If there exists a (full) \emph{product system} $Y = \{Y(s)\}_{s\in\cS}$ such that $X$ is a
subproduct subsystem of $Y$, then $S$ has an isometric and fully coisometric dilation.
\end{proposition}
\begin{proof}
Define a representation $T$ of $Y$ on $H$ by
\be\label{eq:extend}
T_s = S_s \circ p_s,
\ee
where, as above, $p_s$ is the orthogonal projection $Y(s) \rightarrow X(s)$.
A straightforward verification shows that $T$ is indeed a fully coisometric representation of $Y$ on $H$.
By \cite[Theorem 5.2]{Shalit08}, $(Y,T,H)$ has a minimal isometric and fully coisometric dilation $(Y,V,K)$. $(Y,V,K)$
is also clearly a dilation of $(X,S,H)$.
\end{proof}
\begin{corollary}\label{cor:e0dilwhensub}
Let $\Theta = \{\Theta_s\}_{s\in\cS}$ be a cp$_0$-semigroup and let $(E,T) = \Xi(\Theta)$ be the Arveson-Stinespring
representation of $\Theta$. If there is a (full) \emph{product system} $Y$ such that $E$ is a subproduct subsystem of $Y$, then $\Theta$ has an $e_0$-dilation.
\end{corollary}
\begin{proof}
Combine Propositions \ref{prop:semigroup}, \ref{prop:dil_rep_dil_CP} and \ref{prop:exist_isodil}.
\end{proof}

Thus, the problem of constructing $e_0$-dilations to $cp_0$-semigroups is reduced to the problem of embedding a
subproduct system into a full product system. In the next subsection we give an example of a subproduct system that cannot be embedded into full product system. When this can be done in general is a challenging open question.

\begin{corollary}\label{cor:edilwhensub}
Let $\Theta = \{\Theta_s\}_{s\in\mb{N}^k}$ be a $cp$-semigroup generated by $k$ commuting CP maps
$\theta_1, \ldots, \theta_k$, and let $(E,T) = \Xi(\Theta)$ be the Arveson
representation of $\Theta$. Assume, in addition, that
\bes
\sum_{i=1}^k \|\theta_i\|\leq 1.
\ees
If there is a (full) \emph{product system} $Y$ such that $E$ is a subproduct subsystem of $Y$, then
$\Theta$ has an $e$-dilation.
\end{corollary}
\begin{proof}
As in (\ref{eq:extend}), we may extend $T$ to a product system representation of $Y$ on $H$, which we also denote by $T$.
Denote by ${\bf e_i}$ the element of $\mb{N}^k$ with $1$ in the $i$th element and zeros elsewhere.
Then
\bes
\sum_{i=1}^k \|\widetilde{T}_{\bf e_i} \widetilde{T}_{\bf e_i}^*\| = \sum_{i=1}^k \|\theta_i\|\leq 1.
\ees
By the methods of \cite{Shalit07a}, one may show that $S$ has a minimal (regular) isometric dilation. This isometric dilation
provides the required $e$-dilation of $\Theta$.
\end{proof}

\begin{theorem}\label{thm:edil_repdil}
Let $\cM \subseteq B(H)$ be a von Neumann algebra, let $X$ be a subproduct system of $\cM'$-correspondences,
and let $R$ be an injective representation of $X$ on a Hilbert space $H$.
Let $\Theta = \Sigma(X,R)$ be the $cp$-semigroup acting on $R_0(\cM')'$ given by (\ref{eq:reprep}). Assume that
$(\alpha,K,\cR)$ is an $e$-dilation of $\Theta$, and let $(Y,V) = \Xi(\alpha)$ be the Arveson-Stinespring subproduct
system of $\alpha$ together with the identity representation. Assume, in addition, that the map
$\cR' \ni b \mapsto P_H b P_H$ is a
$*$-isomorphism of $\cR'$ onto $R_0(\cM')$. Then $(Y,V,K)$ is a dilation of $(X,R,H)$.
\end{theorem}
\begin{proof}
For every $s \in \cS$, define $W_s : Y(s) \rightarrow B(H)$ by $W_s (y) = P_H V_s(y)P_H$.
We claim that $W = \{W_s\}_{s\in\cS}$ is a representation of $Y$ on $H$. First, note that
$P_H \alpha_s (I-P_H)P_H = \Theta_s(P_H(I-P_H)P_H) = 0$, thus
$P_H \widetilde{V}_s (I \otimes (I-P_H))\widetilde{V}^*_s P_H =0$, and consequently
$P_H \widetilde{V}_s (I \otimes P_H) = P_H \widetilde{V}_s$.
It follows that $W_s (y) = P_H V_s(y) P_H = P_H V_s(y) $. From this it follows that
\begin{align*}
W_s(y_1) W_t(y_2) &= P_HV_s(y_1)P_HV_t(y_2) = P_HV_s(y_1) V_t(y_2) \\
&= P_H V_{s+t}(U^Y_{s,t}(y_1 \otimes y_2)) = W_{s+t}(U^Y_{s,t}(y_1 \otimes y_2)).
\end{align*}

By Theorem \ref{thm:essentially_inverse}, we may assume that $(X,R) = (E,T) = \Xi(\Theta)$ is the
Arveson-Stinespring representation of $\Theta$.
Because $\alpha$ is a dilation of $\Theta$, we have
\bes
\widetilde{W}_s (I \otimes a)\widetilde{W}_s^* = P_H \widetilde{V}_s (I \otimes a)\widetilde{V}^*_s P_H = \Theta_s(a),
\ees
That is, $\Theta = \Sigma(Y,W)$.
Thus, by Theorem \ref{thm:essentially_inverse} and Remark \ref{rem:subsystem_iso}, we may assume that
$E$ is a subproduct subsystem of $Y$, and that $T_s \circ p_s = W_s$, $p_s$ being the projection
of $Y(s)$ onto $E(s)$. In other words, for all $y \in Y$,
\bes
\widetilde{T}_s (p_s \otimes I_H) = P_H \widetilde{W}_s .
\ees
Therefore, $\widetilde{W}_s^* H \subseteq E(s) \otimes H$, and $\widetilde{W}_s^*\big|_H = \widetilde{T}_s^*$.
That is, $(Y,W,H)$ is a dilation of $(E,T,H)$. But $(Y,V,K)$ is a dilation of $(Y,W,H)$,
so it is also a dilation of $(E,T,H)$.
\end{proof}

The assumption that $\cR' \ni b \mapsto P_H b P_H \in \cM'$ is a $*$-isomorphism is satisfied when $\cM = B(H)$ and $\cR = B(K)$. More generally, it is satisfied whenever the central projection of $P_H$ in $\cR$ is $I_K$ (see Propositions 5.5.5 and 5.5.6 in \cite{KRI}).

Let $(\alpha,K,\cR)$ be an $e$-dilation of a semigroup $\Theta$ on $\cM \subseteq B(H)$. $(\alpha,K,\cR)$ is called a \emph{minimal dilation} if the central support of $P_H$ in $\cR$ is $I_K$ and if
\bes
\cR = W^*\left( \bigcup_{s\in\cS} \alpha_s(\cM) \right).
\ees
\begin{corollary}\label{cor:edil_repdil}
Let $\Theta$ be $cp$-semigroup on $\cM \subseteq B(H)$, and let $(\alpha,K,\cR)$ be a minimal dilation of $\Theta$. Then $\Xi(\alpha)$ is an isometric dilation of $\Xi(\Theta)$.
\end{corollary}

\subsection{$cp$-semigroups with no $e$-dilations. Obstructions of a new nature}

By Parrot's famous example \cite{Parrot}, there exist $3$ commuting contractions that do not have a commuting isometric dilation. In 1998 Bhat asked whether $3$ commuting CP maps necessarily have a commuting $*$-endomorphic dilation \cite{Bhat98}. Note that it is not obvious that the non-existence of an isometric dilation for three commuting contractions would imply the non-existence of a $*$-endomorphic dilation for $3$ commuting CP maps. However, it turns out that this is the case.

\begin{theorem}\label{thm:parrot}
There exists a $cp$-semigroup $\Theta = \{\Theta_n\}_{n \in \mb{N}^3}$ acting on a $B(H)$ for which there is no $e$-dilation $(\alpha,K,B(K))$. In fact, $\Theta$ has no \emph{minimal} $e$-dilation $(\alpha,K,\cR)$ on any von Neumann algebra $\cR$.
\end{theorem}
\begin{proof}
Let $T_1,T_2,T_3 \in B(H)$ be three commuting contractions that have no isometric dilation and such that $T_1^{n_1} T_2^{n_2} T_3^{n_3} \neq 0$ for all $n  = (n_1, n_2, n_3) \in \mb{N}^3$ (one may take commuting contractions $R_1,R_2,R_3$ with no isometric dilation as in Parrot's example \cite{Parrot}, and define $T_i = R_i \oplus 1$).
For all $n  = (n_1, n_2, n_3) \in \mb{N}^3$, define
\bes
\Theta_n (a) = T_1^{n_1} T_2^{n_2} T_3^{n_3} a  (T_3^{n_3})^*(T_2^{n_2})^* (T_1^{n_1})^* \,\, , \,\, a \in B(H).
\ees
Note that $\Theta = \Sigma(X,R)$, where $X = \{X(n)\}_{n\in\mb{N}^3}$ is the subproduct system given by $X(n) = \mb{C}$ for all $n \in \mb{N}^3$, and $R$ is the (injective) representation that sends $1 \in X(n)$ to $T_1^{n_1} T_2^{n_2} T_3^{n_3}$ (the product in $X$ is simply multiplication of scalars).

Assume, for the sake of obtaining a contradiction, that $\Theta = \{\Theta_n\}_{n\in\mb{N}^3}$ has
an $e$-dilation $(\alpha,K,B(K))$. Let $(Y,V) = \Xi(\alpha)$ be the Arveson-Stinespring subproduct system of $\alpha$ together with the identity representation. By Theorem \ref{thm:edil_repdil}, $(Y,V,K)$ is a dilation of $(X,R,H)$. It follows that $V_1,V_2,V_3$ are a commuting isometric dilation of $T_1,T_2,T_3$ where $V_1 := V(1)$ with $1 \in X(1,0,0)$, $V_2 := V(1)$ with $1 \in X(0,1,0)$, and
$V_3 := V(1)$ with $1 \in X(0,0,1)$. This is a contradiction.

Finally, a standard argument shows that if $(\alpha,K,\cR)$ is a minimal dilation of $\Theta$, then $\cR = B(K)$.
\end{proof}

Until this point, all the results that we have seen in the dilation theory of $cp$-semigroups have been anticipated by the classical theory of isometric dilations. We shall now encounter a phenomena that has no counterpart in the classical theory.

By \cite[Proposition 9.2]{SzNF70}, if $T_1, \ldots, T_k$ is a commuting $k$-tuple of contractions such that
\be\label{eq:normconditionT}
\sum_{i=1}^k \|T_i\|^2 \leq 1 ,
\ee
then $T_1,\ldots, T_k$ has a commuting regular unitary dilation (and, in particular, an isometric dilation). One is tempted to conjecture that if $\theta_1, \ldots, \theta_k$ is a commuting $k$-tuple of CP maps such that
\be\label{eq:normconditiontheta}
\sum_{i=1}^k \|\theta_i\| \leq 1,
\ee
then the tuple $\theta_1, \ldots, \theta_k$ has an $e$-dilation. Indeed, if $\theta_i(a) = T_i a T_i^*$, where $T_1, \ldots, T_k$ is a commuting $k$-tuple satisfying (\ref{eq:normconditionT}), then
it is easy to construct an $e$-dilation of $\theta_1, \ldots, \theta_k$ from the isometric dilation of $T_1, \ldots, T_k$. However, it turns out that (\ref{eq:normconditiontheta}) is far from being sufficient for an $e$-dilation to exist. We need some preparations before exhibiting an example.

\begin{proposition}\label{prop:sprdctcntrexample}
There exists a subproduct system that is not a subsystem of any product system.
\end{proposition}
\begin{proof}
We construct a counter example over $\mb{N}^3$. Let $e_1,e_2,e_3$ be the standard basis of $\mb{N}^3$. We let $X(e_1) = X(e_2) = X(e_3) = \mb{C}^2$. Let $X(e_i + e_j) = \mb{C}^2 \otimes \mb{C}^2$ for all $i,j=1,2,3$. Put $X(n) = \{0\}$ for all $n\in\mb{N}^k$ such that $|n|>2$. To complete the construction of $X$ we need to define the product maps $U^X_{m,n}$.
Let $U^X_{e_i,e_j}$ be the identity on $\mb{C}^2 \otimes \mb{C}^2$ for all $i,j$ except for $i=3,j=2$, and let $U^X_{e_3,e_2}$ be the flip. Define the rest of the products to be zero maps (except the maps $U^X_{0,n}, U^X_{m,0}$ which are identities). This product is evidently coisometric, and it is also associative, because the product of any three nontrivial elements vanishes.

Let $Y$ be a product system ``dilating" $X$. Then for all $k,m,n \in \mb{N}^k$ we have
\bes
U^Y_{k+m,n}(U^Y_{k,m}\otimes I) =  U^Y_{k,m+n}(I \otimes U^Y_{m,n}),
\ees
or
\bes
U^Y_{k+m,n} =  U^Y_{k,m+n}(I \otimes U^Y_{m,n})(U^Y_{k,m}\otimes I)^*,
\ees
and
\bes
U^Y_{k,m+n} = U^Y_{k+m,n}(U^Y_{k,m}\otimes I)(I \otimes U^Y_{m,n})^*.
\ees
Iterating these identities, we have, on the one hand,
\begin{align*}
U_{e_3,e_1+e_2} &= U^Y_{e_3+e_2,e_1}(U^Y_{e_3,e_2}\otimes I)(I \otimes U^Y_{e_2,e_1})^* \\
&= U^Y_{e_2,e_3+e_1}(I \otimes U^Y_{e_3,e_1})(U^Y_{e_2,e_3}\otimes I)^*(U^Y_{e_3,e_2}\otimes I)(I \otimes U^Y_{e_2,e_1})^* \\
&= U^Y_{e_1+e_2,e_3}(U^Y_{e_2,e_1}\otimes I)(I \otimes U^Y_{e_1,e_3})^*(I \otimes U^Y_{e_3,e_1})(U^Y_{e_2,e_3}\otimes I)^*(U^Y_{e_3,e_2}\otimes I)(I \otimes U^Y_{e_2,e_1})^*, \\
\end{align*}
and on the other hand
\begin{align*}
U_{e_3,e_1+e_2} &= U^Y_{e_3+e_1,e_2}(U^Y_{e_3,e_1}\otimes I)(I \otimes U^Y_{e_1,e_2})^* \\
&= U^Y_{e_1,e_3+e_2}(I \otimes U^Y_{e_3,e_2})(U^Y_{e_1,e_3}\otimes I)^*(U^Y_{e_3,e_1}\otimes I)(I \otimes U^Y_{e_1,e_2})^* \\
&= U^Y_{e_1+e_2,e_3}(U^Y_{e_1,e_2}\otimes I)(I \otimes U^Y_{e_2,e_3})^*(I \otimes U^Y_{e_3,e_2})(U^Y_{e_1,e_3}\otimes I)^*(U^Y_{e_3,e_1}\otimes I)(I \otimes U^Y_{e_1,e_2})^*. \\
\end{align*}
Canceling $U^Y_{e_1+e_2,e_3}$, we must have
\begin{align*}
(U^Y_{e_1,e_2}\otimes I)(I \otimes U^Y_{e_2,e_3})^* & (I \otimes U^Y_{e_3,e_2})(U^Y_{e_1,e_3}\otimes I)^*(U^Y_{e_3,e_1}\otimes I)(I \otimes U^Y_{e_1,e_2})^* \\
&= (U^Y_{e_2,e_1}\otimes I)(I \otimes U^Y_{e_1,e_3})^*(I \otimes U^Y_{e_3,e_1})(U^Y_{e_2,e_3}\otimes I)^*(U^Y_{e_3,e_2}\otimes I)(I \otimes U^Y_{e_2,e_1})^*.
\end{align*}
Now, $U^X_{e_i,e_j}$ were unitary to begin with, so the above identity implies
\begin{align*}
(U^X_{e_1,e_2}\otimes I)(I \otimes U^X_{e_2,e_3})^* & (I \otimes U^X_{e_3,e_2})(U^X_{e_1,e_3}\otimes I)^*(U^X_{e_3,e_1}\otimes I)(I \otimes U^X_{e_1,e_2})^* \\
&= (U^X_{e_2,e_1}\otimes I)(I \otimes U^X_{e_1,e_3})^*(I \otimes U^X_{e_3,e_1})(U^X_{e_2,e_3}\otimes I)^*(U^X_{e_3,e_2}\otimes I)(I \otimes U^X_{e_2,e_1})^*.
\end{align*}
Recalling the definition of the product in $X$ (the product is usually the identity), this reduces to
\begin{align*}
I \otimes U^X_{e_3,e_2} = U^X_{e_3,e_2}\otimes I.
\end{align*}
This is absurd. Thus, $X$ cannot be dilated to a product system.
\end{proof}

We can now strengthen Theorem \ref{thm:parrot}:
\begin{theorem}\label{thm:strongparrot}
There exists a $cp$-semigroup $\Theta = \{\Theta_n\}_{n \in \mb{N}^3}$ acting on a $B(H)$, such that for all $\lambda > 0$, $\lambda \Theta$ has no $e$-dilation $(\alpha,K,B(K))$, and no \emph{minimal} $e$-dilation $(\alpha,K,\cR)$ on any von Neumann algebra $\cR$.
\end{theorem}
\begin{proof}
Let $X$ be as in Proposition \ref{prop:sprdctcntrexample}. Let $\Theta$ be the $cp$-semigroup generated by the $X$-shift, as in Section \ref{subsec:shift} of the paper. Of course, $\Theta$, as a semigroup over $\mb{N}^3$, can be generated by three commuting CP maps $\theta_1,\theta_2,\theta_3$. $X$ cannot be embedded into a full product system, so by Theorem \ref{thm:edil_repdil}, $\Theta$ has no minimal $e$-dilation, nor does it have an $e$-dilation acting on a $B(K)$. Note that if $\Theta$ is scaled \emph{its product system is left unchanged} (this follows from Theorem \ref{thm:essentially_inverse}: if you take $X$ and scale the representation $S^X$ you get a scaled version of $\Theta$). So no matter how small you take $\lambda  > 0$, $\lambda \theta_1,\lambda \theta_2,\lambda \theta_3$ cannot be dilated to three commuting $*$-endomorphisms on $B(K)$, nor to a minimal three-tuple on any von Neumann algebra.
\end{proof}

Note that the obstruction here seems to be of a completely different nature from the one in the example given in Theorem \ref{thm:parrot}. The subproduct system arising there is already a product system, and, indeed, the $cp$-semigroup arising there can be dilated once it is multiplied by a small enough scalar.

\newpage
\part{Subproduct systems over $\mb{N}$}

\section{Subproduct systems of Hilbert spaces over $\mb{N}$}\label{sec:subproductN}

We now specialize to subproduct systems of Hilbert W$^*$-correspondences over the semigroup $\mb{N}$, so from now on any subproduct system is to be understood as such (soon we will specialize even further to subproduct systems of Hilbert spaces).

\subsection{Standard and maximal subproduct systems}

If $X$ is a subproduct system over $\mb{N}$, then $X(0) = \cM$ (some von Neumann algebra), $X(1)$ equals some W$^*$-correspondence $E$, and $X(n)$ can be regarded as a subspace of $E^{\otimes n}$. The following lemma allows us to consider $X(m+n)$ as a subspace of $X(m) \otimes X(n)$.

\begin{lemma}\label{lem:projection_subproduct}
Let $X = \{X(n)\}_{n \in \mb{N}}$ be a subproduct system. $X$ is isomorphic to a subproduct system $Y = \{Y(n)\}_{n\in\mb{N}}$ with coisometries $\{U_{m,n}^Y\}_{m,n \in \mb{N}}$ that satisfies
\bes
Y(1) = X(1)
\ees
and
\be\label{eq:p_subset}
Y(m+n) \subseteq Y(m) \otimes Y(n).
\ee
Moreover, if $p_{m+n}$ is the orthogonal projection of $Y(1)^{\otimes (m + n)}$ onto $Y(m+n)$, then
\be\label{eq:p_maps}
U_{m,n}^Y = p_{m+n}\Big|_{Y(m) \otimes Y(n)}
\ee
and the projections $\{p_n\}_{n \in \mb{N}}$ satisfy
\be\label{eq:p_assoc}
p_{k+m+n} = p_{k+m+n}(I_{E^{\otimes k}} \otimes p_{m+n}) = p_{k+m+n}(p_{k+m} \otimes I_{E^{\otimes n}}).
\ee
\end{lemma}
\begin{proof}
Denote by $U^X_{m,n}$ the subproduct system maps $X(s) \otimes X(t) \rightarrow X(s+t)$. Denote $E = X(1)$. We first note that for every $n$ there is a well defined coisometry $U_n : E^{\otimes n} \rightarrow X(n)$ given by composing in any way a sequence of maps $U^X_{k,m}$ (for example, one can take $U_3 = U^X_{2,1}(U^X_{1,1} \otimes I_{E})$ and so on). We define $Y(n) = {\rm Ker}(U_n)^\perp$, and we let $p_n$ be the orthogonal projection from $E^{\otimes n}$ onto $Y(n)$.
$p_n = U_n^* U_n$, so, in particular, $p_n$ is a bimodule map. For all $m,n \in \mb{N}$ we have that
\bes
E^{\otimes (m)} \otimes {\rm Ker}(U_n) \subseteq {\rm Ker}(U_{m+n}) .
\ees
Thus $E^{\otimes (m)} \otimes {\rm Ker}(U_n)^\perp \supseteq {\rm Ker}(U_{m+n})^\perp$, so
$p_{m+n} \leq I_{E^{\otimes m}} \otimes p_n$. This means that (\ref{eq:p_assoc}) holds. In addition, defining $U_{m,n}^Y$ to be $p_{m+n}$ restricted to $Y(m) \otimes Y(n) \subseteq E^{\otimes (m+n)}$ gives $Y$ the associative multiplication of a subproduct system.

It remains to show that $X$ is isomorphic to $Y$. For all $n$, $X(n)$ is spanned by elements of the form $U_n(x_1 \otimes \cdots \otimes x_n)$, with $x_1, \ldots, x_n \in E$. We define a map $V_n : X(n) \rightarrow Y(n)$ by
\bes
V_n \big( U_n(x_1 \otimes \cdots \otimes x_n) \big) = p_n (x_1 \otimes \cdots \otimes x_n).
\ees
It is immediate that $V_n$ preserves inner products (thus it is well defined) and that it maps $X(n)$ onto $Y(n)$. Finally, for all $m,n \in \mb{N}$ and $x \in E^{\otimes m}, y \in E^{\otimes n}$,
\begin{align*}
V_{m+n}\big(U_{m,n}^X (U_m(x) \otimes U_n(y)) \big)
&= V_{m+n}\big(U_{m+n}(x\otimes y) \big) \\
&= p_{m+n} (x \otimes y) \\
&= p_{m+n} (p_m x \otimes p_n y) \\
&= p_{m+n} \big((V_m U_m (x)) \otimes (V_n U_n (y))\big) \\
&= U^Y_{m+n} \big((V_m U_m (x)) \otimes (V_n U_n (y))\big) ,
\end{align*}
and (\ref{eq:iso}) holds.
\end{proof}

\begin{definition}
A subproduct system $Y$ satisfying (\ref{eq:p_subset}), (\ref{eq:p_maps}) and (\ref{eq:p_assoc}) above will be called a \emph{standard} subproduct system.
\end{definition}

Note that a standard subproduct system is a subproduct subsystem of the full product system $\{E^{\otimes n}\}_{n\in\mb{N}}$.
\begin{corollary}\label{cor:discrete_bhat}
Every $cp$-semigroup over $\mb{N}$ has an $e$-dilation.
\end{corollary}
\begin{proof}
The unital case follows from Corollary \ref{cor:e0dilwhensub} together with the above lemma. The nonunital case follows from a similar construction (where the dilation of a non-fully-coisometric representation is obtained by adapting \cite[Theorem 4.2]{Shalit07a} instead of \cite[Theorem 5.2]{Shalit08}).
\end{proof}

Let $k \in \mb{N}$, and let $E = X(1), X(2), \ldots, X(k)$ be subspaces of $E, E^{\otimes 2}, \ldots, E^{\otimes k}$, respectively, such that the orthogonal projections $p_n : E^{\otimes n} \rightarrow X(n)$ satisfy
$$p_n \leq I_{E^{\otimes i}} \otimes p_j$$
and
$$p_n \leq p_i \otimes I_{E^{\otimes j}}$$
for all $i,j,n \in \mb{N}_+$ satisfying $i+j = n \leq k$. In this case one can define a maximal standard subproduct system $X$ with the prescribed fibers $X(1), \ldots, X(k)$ by defining inductively for $n>k$
$$X(n) = \left(\bigcap_{i+j=n} E^{\otimes i} \otimes X(j)\right) \bigcap \left(\bigcap_{i+j=n} X(i) \otimes E^{\otimes j}\right).$$
It is easy to see that
$$X(n) = \bigcap_{n_1 + \ldots + n_m=n} X(n_1) \otimes \cdots \otimes X(n_m) = \bigcap_{i+j=n} X(i) \otimes X( j).$$
We then have obvious formulas for the projections $\{p_n\}_{n\in \mb{N}}$ as well, for example
$$p_n = \bigwedge_{i+j = n}p_i \otimes p_j \,\, , \,\, (n > k).$$

\subsection{Examples}\label{subsec:examples}

\begin{example}\label{expl:symmmax}
\emph{
In the case $k = 1$, the maximal standard subproduct system with prescribed fiber $X(1) = E$, with $E$ a Hilbert space, is the full product system $F_E$ of Example \ref{expl:full}. If $\dim E = d$, we think of this subproduct system as the product system representing a (row-contractive) $d$-tuple $(T_1, \ldots, T_d)$ of non commuting operators, that is, $d$ operators that are not assumed to satisfy any relations (the idea behind this last remark  must be rather vague at this point, but it shall become clearer as we proceed).
In the case $k = 2$, if $X(2)$ is the symmetric tensor product $E$ with itself then the maximal standard subproduct system with prescribed fibers $X(1), X(2)$ is the symmetric subproduct system $SSP_E$ of Example \ref{expl:symm}. We think of $SSP$ as the subproduct system representing a commuting $d$-tuple.}
\end{example}

\begin{example}\label{expl:dimexp}
\emph{
Let $E$ be a two dimensional Hilbert space with basis $\{e_1,e_2\}$. Let $X(2)$ be the space spanned by $e_1 \otimes e_1, e_1 \otimes e_2$, and $e_2 \otimes e_1$. In other words, $X(2)$ is what remains of $E^{\otimes 2}$ after we declare that $e_2 \otimes e_2=0$. We think of the maximal standard subproduct system $X$ with prescribed fibers $X(1) = E, X(2)$ as the subproduct system representing pairs $(T_1, T_2)$ of operators subject only to the condition $T_2^2 = 0$. $E^{\otimes n}$ has a basis consisting of all vectors of the form $e_\alpha = e_{\alpha_1} \otimes \cdots \otimes e_{\alpha_n}$ where $\alpha = \alpha_1 \cdots \alpha_n$ is a word of length $n$ in ``1" and ``2". $X(n)$ then has a basis consisting of all vectors $e_\alpha$ where $\alpha$ is a word of length $n$ not containing ``22" as a subword.
Let us compute $\dim X(n)$, that is, the number of such words.}

\emph{Let $A_n$ denote the number of words not containing ``22" that have leftmost letter ``1", and let $B_n$ denote the
number of words not containing ``22" that have leftmost letter ``2".
Then we have the recursive relation $A_n = A_{n-1} + B_{n-1}$ and $B_n = A_{n-1}$. The solution of this recursion gives}
$$\dim X(n) = A_n + B_n \approx \left(\frac{1+\sqrt{5}}{2}\right)^n .$$
\emph{As one might expect, the dimension of $X(n)$ grows exponentially fast.}
\end{example}

\begin{example}\label{expl:dimconst}
\emph{Suppose that we want a ``subproduct system that will represent a pair of operators $(T_1, T_2)$ such that
$T_i T_2 = 0$ for $i=1,2$". Although we have not yet made clear what we mean by this, let us proceed heuristically along the lines of the preceding examples. We let $E$ be as above, but now we declare $e_1 \otimes e_2 = e_2 \otimes e_2 = 0$. In other words, we define $X(2) = \{e_1 \otimes e_2, e_2 \otimes e_2\}^\perp$. One checks that the maximal standard subproduct system $X$ with prescribed fibers $X(1) = E, X(2)$ is given by $X(n) = \textrm{span}\{e_1 \otimes e_1 \otimes \cdots \otimes e_1, e_2 \otimes e_1 \otimes \cdots \otimes e_1\}$. This is an example of a subproduct system with two dimensional fibers.}
\end{example}

At this point two natural questions might come to mind. First, \emph{is every standard subproduct system $X$ the maximal subproduct system with prescribed fibers $X(1), \ldots, X(k)$ for some $k \in \mb{N}$?} Second, \emph{does $\dim X(n)$ grow exponentially fast (or remain a constant) for every subproduct system $X$?} The next example answers both questions negatively.

\begin{example}\label{expl:notmax}
\emph{Let $E$ be as in the preceding examples, and let $X(n)$ be a subspace of $E^{\otimes n}$ having basis the set}
$$\{e_\alpha: |\alpha|=n, \alpha \textrm{ does not contain the words } 22, 212, 2112, 21112, \ldots\} .$$
\emph{Then $X = \{X(n)\}_{n \in \mb{N}}$ is a standard subproduct system, but it is smaller than the maximal subproduct system defined by any initial $k$ fibers. Also, $X(n)$ is the span of $e_{\alpha}$ with $\alpha = 11\cdots 11, 21\cdots 11, 121\cdots 11, \ldots, 11 \cdots 12$, thus}
$$\dim X(n) = n+1 ,$$
\emph{so this is an example of a subproduct system with fibers that have a linearly growing dimension.}
\end{example}

Of course, one did not have to go far to find an example of a subproduct system with linearly growing dimension: indeed, the dimension of the fibers of the symmetric subproduct system $SSP_{\mb{C}^d}$ is known to be
\bes
\dim SSP_{\mb{C}^d}(n) = \left(
\begin{array}{c}
 n+d-1 \\
 n
\end{array} \right).
\ees
Taking $d=2$ we get the same dimension as in Example \ref{expl:notmax}. However, $SSP := SSP_{\mb{C}^2}$ and the subproduct system $X$ of Example \ref{expl:notmax} are not isomorphic: for any nonzero $x \in SSP(1)$, the ``square" $U^{SSP}_{1,1}(x \otimes x) \in SSP(2)$ is never zero, while $U^X_{1,1}(e_2 \otimes e_2) = 0$.

Here is an interesting question that we do not know the  answer to: \emph{given a solution $f:\mb{N}\rightarrow \mb{N}$ to the functional inequality}
\bes
f(m+n) \leq f(m) f(n) \,\, , \,\, m,n \in \mb{N},
\ees
\emph{does there exists a subproduct system $X$ such that $\dim X(n) = f(n)$ for all $n \in \mb{N}$?}

\begin{remark}
\emph{One can cook up simple examples of subproduct systems that are not standard. We will not write these examples down, as we already know that such a subproduct system is isomorphic to a standard one.}
\end{remark}

\subsection{Representations of subproduct systems}\label{subsec:representations}

Fix a W$^*$-correspondence $E$. Every completely contractive linear map $T_1:E\rightarrow B(H)$ gives rise to a c.c. representation $T^n$ of the full product system $F_E=\{E^{\otimes n}\}_{n \in \mb{N}}$ by defining
for all $x \in E^{\otimes n}$ and $h \in H$
\be\label{eq:alotofT}
T^n(x)h = \tT_1\big(I_{E} \otimes \tT_1\big) \cdots \big(I_{E^{\otimes (n-1)}}\otimes \tT_1\big) (x \otimes h) ,
\ee
where $\tT_1 : E \otimes H \rightarrow H$ is given by $\tT_1(e \otimes h) = T_1(e) h$. We will denote the operator acting on $x \otimes h$ in the right hand side of (\ref{eq:alotofT}) as $\tT^n$, so as not to confuse with $\tT_n$, which sometimes has a different meaning (namely: if $T$ denotes a c.c. representation of a subproduct system $X$ then
$$\tT_n: X(n) \otimes H \rightarrow H$$
is given by
$$\tT_n (x \otimes h) = T(x) h $$
for all $x \in X(n), h \in H$. Of course, when $X = F_E$, $T$ is a representation of $F_E$ and $T_1$ is the restriction of $T$ to $E$, then $\tT^n = \tT_n$ for all $n$). If $X$ is a standard subproduct system and $X(1) = E$, we obtain a completely contractive representation of $X(n)$ by restricting $T^n$ to $X(n)$. Let us denote this restriction by $T_n$, and denote the family $\{T_n\}_{n \in \mb{N}}$ by $T$.
\begin{proposition}\label{prop:representation}
Let $X$ be a standard subproduct system with projections $\{p_n\}_{n\in \mb{N}}$, and let $T_1:E \rightarrow B(H)$ be a completely contractive map. Construct the family of maps $T = \{T_n\}_{n\in \mb{N}}$, with $T_n:X(n) \rightarrow B(H)$ as in the preceding paragraph. Then the following are equivalent:
\begin{enumerate}
\item\label{it:Trep} $T$ is a representation of $X$.
\item\label{it:tT} For all $m,n \in \mb{N}$,
\be\label{eq:tTrep}
\tT_{m}(I_{X(m)} \otimes \tT_n)(p_m \otimes p_n \otimes I_H)(p^\perp_{m+n} \otimes I_H) = 0 .
\ee
\item\label{it:tTn} For all $n \in \mb{N}$,
\be\label{eq:tTn}
\tT^n (p_n^\perp \otimes I_H) = 0.
\ee
\end{enumerate}
\end{proposition}
\begin{proof}
If $T$ is a representation, then
\bes
\tT_{m}(I_{X(m)} \otimes \tT_n)(p_m \otimes p_n \otimes I_H)(p^\perp_{m+n} \otimes I_H)
= \tT_{m+n}(p_{m+n} \otimes I_H)(p^\perp_{m+n} \otimes I_H)
= 0 ,
\ees
so \ref{it:Trep} $\Rightarrow$ \ref{it:tT}. To prove \ref{it:tT} $\Rightarrow$ \ref{it:tTn} note first that (\ref{eq:tTn}) is clear for $n=1$. Assuming that (\ref{eq:tTn}) holds for $n=1,2, \ldots, k-1$, we will show that it holds
for $n=k$.
\begin{align*}
\tT^k (p_k^\perp \otimes I_H)
&= \tT^1(I \otimes \tT^{k-1}) (p_k^\perp \otimes I_H) \\
&= \tT^1(I \otimes \tT^{k-1})(I_{E} \otimes p_{k-1}^\perp \otimes I_H + I_{E} \otimes p_{k-1} \otimes I_H) (p_k^\perp  \otimes I_H)\\
(*)&= \tT^1(I \otimes \tT^{k-1}(p_{k-1} \otimes I_H)) (p_k^\perp  \otimes I_H)\\
&= \tT_1(I \otimes \tT_{k-1}(p_{k-1} \otimes I_H)) (p_k^\perp  \otimes I_H)\\
(**)&= 0 .
\end{align*}
The equality marked by (*) is true by the inductive hypothesis, and the one marked by (**) follows from (\ref{eq:tTrep}).

Finally, \ref{it:tTn} $\Rightarrow$ \ref{it:Trep}: by (\ref{eq:tTn}) we have $\tT^{n} (p_{n} \otimes I_H) = \tT^{n}$. On the other hand, $\tT^{n} (p_{n} \otimes I_H) = \tT_{n} (p_{n} \otimes I_H)$. Thus
\begin{align*}
\tT_{m+n} (p_{m+n} \otimes I_H)
&= \tT^{m+n} (p_{m+n} \otimes I_H) \\
&= \tT^{m+n} \\
&= \tT^{m}(I_{X(m)} \otimes \tT^n) \\
&= \tT_{m}(I_{X(m)} \otimes \tT_n)(p_m \otimes p_n \otimes I_H) ,
\end{align*}
which shows that $T$ is a representation.
\end{proof}

\begin{proposition}
Let $X$ be the maximal standard subproduct system with prescribed fibers $X(1), \ldots, X(k)$, and let $T_1:E \rightarrow B(H)$ be a completely contractive map. Construct $T$ as in Proposition \ref{prop:representation}. Then $T$ is a representation of $X$ if and only if
\be\label{eq:tTnk}
\tT^n (p_n^\perp \otimes I_H) = 0 \quad \textrm{for all} \quad n=1,2,\ldots, k .
\ee
\end{proposition}
\begin{proof}
The necessity of (\ref{eq:tTnk}) follows from Proposition \ref{prop:representation}. By the same proposition, to show that the condition is sufficient it is enough to show that (\ref{eq:tTnk}) holds for all $n \in \mb{N}$. Given $m \in \mb{N}$, we have $p_m = \bigwedge_q q$, where $q$ runs over all
projections of the form $q = I_{X(i)} \otimes p_j$ or $q = p_i \otimes I_{X(j)}$, with $i,j \in \mb{N}_+$ and $i+j = m$. But then $p_m^\perp = \bigvee_{q} q^\perp$, thus if (\ref{eq:tTnk}) holds for all $n<m$ then it also holds for $n=m$.
\end{proof}

\subsection{Fock spaces and standard shifts}

\begin{definition}
Let $X$ be a subproduct system of Hilbert spaces. Fix an orthonormal basis $\{e_i\}_{i \in \cI}$ of $E =X(1)$.
$X(n)$, when considered as a subspace of $\mathfrak{F}_X$, is called \emph{the $n$ particle space}.
The \emph{standard $X$-shift (related to $\{e_i\}_{i \in \cI}$) on $\mathfrak{F}_X$} is the tuple of operators $\underline{S}^X = \left(S^X_i\right)_{i \in \cI}$ in $B(\mathfrak{F}_X)$ given by
\bes
S^X_i(x) = U_{1,n}(e_i \otimes x) ,
\ees
for all $i \in \cI$, $n \in \mb{N}$ and $x \in X(n)$.
\end{definition}
It is clear that $S^X_i = S^X(e_i)$, where $S^X$ is the shift representation given by Definition \ref{def:shiftrep}.

If $F$ denotes the usual full product system (Example \ref{expl:full}) then $\mathfrak{F}_F$
is the usual Fock space and the tuple $(S^F_i)_{i \in \cI}$ is the standard shift (the $\cI$ orthogonal shift of \cite{Popescu89}). We shall denote $\mathfrak{F}_F$ as $\mathfrak{F}$ and $(S^F_i)_{i \in \cI}$ as $(S_i)_{i \in \cI}$. It is then obvious that the tuple $\left(S^X_i\right)_{i \in \cI}$ is a row contraction, as it is the compression of the row contraction $(S_i)_{i \in \cI}$. Indeed, assuming (as we may, thanks to Lemma \ref{lem:projection_subproduct}) that $U_{m,n}$ is an orthogonal projection $p_{m+n}:X(m) \otimes X(n) \rightarrow X(m+n)$, and denoting $p = \oplus_{n}p_n$, we have for all $i$ that $S^X_i = p S_i \big|{\mathfrak{F}_X}$.

\begin{example}\label{expl:qcommuting}
\emph{
The $q$-commuting Fock space of \cite{Dey07} also fits into this framework.
Indeed, let (as in \cite{Dey07}) $\Gamma(\mb{C}^d)$ be the full Fock space, let $\Gamma_q(\mb{C}^d)$ denote the $q$-commuting Fock space, and
let $Y(n)$ be the ``$n$ particle $q$-commuting space" with orthogonal projection $p_n : (\mb{C}^d)^n \rightarrow Y(n)$. Then a straightforward calculation shows that the projections $\{p_n\}_{n\in\mb{N}}$ satisfy
equation (\ref{eq:p_assoc}) of Lemma  \ref{lem:projection_subproduct}, thus $Y = \{Y(n)\}_{n\in\mb{N}}$ is a subproduct system (satisfying (\ref{eq:p_subset}) and (\ref{eq:p_maps})).
With our notation from above we have that $\mathfrak{F}_Y = \Gamma_q(\mb{C}^d)$ and that
the tuple $(S^Y_i, \ldots, S^Y_d)$ is the standard $q$-commuting shift.
}
\end{example}

$S^F$, the standard shift of the full product system on the full Fock space, will be denoted by $S$, and will be called simply \emph{the standard shift}.

By the notation introduced in Definition \ref{def:Xpiece}, the symbol $S^X$ is also used to denote the maximal $X$-piece of the standard shift $S$. The following proposition -- which is a generalization of \cite[Proposition 6]{BBD03}, \cite[Proposition 11]{Dey07} and \cite[Proposition 2.9]{Popescu06} -- shows that this is consistent.

\begin{proposition}\label{prop:Xshift}
Let $X$ subproduct subsystem of a subproduct system $Y$. Then the maximal $X$-piece of the standard $Y$-shift is the standard $X$-shift.
\end{proposition}
\begin{proof}
Let $E = Y(1)$, and let $F = F_E$ be the full product system. Viewing
$F(n)\otimes\mathfrak{F}_F$ as direct sum of $|\cI|^n$ copies of $\mathfrak{F}_F$, $(\widetilde{S})_n$ is just the row isometry
$(S_{i_1}\circ\cdots\circ S_{i_n})_{i_1, \ldots, i_n \in \cI}$ from the space of columns $\mathfrak{F}_F \oplus \mathfrak{F}_F \oplus \cdots$ into $\mathfrak{F}_F$. In other words, for $h \in \mathfrak{F}_F$ and $i_1, \ldots, i_n \in I$,
$$(\widetilde{S})_n \big((e_{i_1} \otimes \cdots \otimes e_{i_n}) \otimes h\big) = S_{i_1}\circ\cdots\circ S_{i_n}h = (e_{i_1} \otimes \cdots \otimes e_{i_n}) \otimes h.$$
This is an isometry, and the adjoint works by sending $(e_{i_1} \otimes \cdots \otimes e_{i_n}) \otimes h \in \mathfrak{F}_F$ back to $(e_{i_1} \otimes \cdots \otimes e_{i_n}) \otimes h \in F(n) \otimes \mathfrak{F}_F$, and by sending the $0,1, \ldots, n-1$ particle spaces to $0$.

Now, if $Z$ is any standard subproduct subsystem of $F$, then
\bes
\left(\widetilde{S^Z}\right)_n = P_{\mathfrak{F}_Z}\left(\widetilde{S}\right)_n \big|_{Z(n) \otimes \mathfrak{F}_Z},
\ees
thus
\be\label{eq:SZ*}
\left(\widetilde{S^Z}\right)_n^* = P_{Z(n) \otimes \mathfrak{F}_Z}\left(\widetilde{S}\right)_n^* \big|_{\mathfrak{F}_Z}.
\ee
Now if $h$ is in the $k$ particle space of $\mathfrak{F}_F$ with $k<n$, then $(\widetilde{S^Z})_n^* h = 0$. If $k\geq n$, then since $Z(k) \subseteq Z(n) \otimes Z(k-n)$ we may write $h = \sum \xi_i \otimes \eta_i$, where $\xi_i \in Z(n)$ and $\eta_i \in Z(k-n)$. Thus by (\ref{eq:SZ*}) we find that
\be\label{eq:tilde*}
(\widetilde{S^Z})_n^* \left(\sum \xi_i \otimes \eta_i \right) = \sum p^Z_n \xi_i \otimes p^Z_{k-n}\eta_i = \sum \xi_i \otimes \eta_i.
\ee
From these considerations it follows that the standard $X$-shift is in fact an $X$-piece of the standard $Y$ shift,
as $(\widetilde{S^Y})_n^* \big|_{\mathfrak{F}_X} = (\widetilde{S^X})_n^*$. It remains to show that the $X$-shift is maximal.

Assume that there is a Hilbert space $H$, $\mathfrak{F}_X \subseteq H \subseteq \mathfrak{F}_Y$, such that the compression of $S^Y$ to $H$ is an $X$-piece of $Y$, that is, $H \in \cP(X,S^Y)$ (see equation (\ref{eq:cPXT})). Let $h \in H \ominus \mathfrak{F}_X$. We shall prove that $h = 0$. Being orthogonal to all of $\mathfrak{F}_X$, $p^Y_n h$ must be orthogonal to $X(n)$ for all $n$. Thus, we may assume that $h\in Y(n) \ominus X(n)$ for some $n$. But then by (\ref{eq:tilde*})
\bes
(\widetilde{S^Y})_n^* h = h \otimes \Omega.
\ees
But since $H \in \cP(X,S^Y)$, we must have $h \otimes \Omega \in X(n) \otimes H$, and this, together with $h \in Y(n) \ominus X(n)$, forces $h=0$.
\end{proof}

\section{Zeros of homogeneous polynomials in noncommutative variables}\label{sec:projective}

In the next section we will describe a model theory for representations of subproduct systems. But before that we dedicate this section to build a precise connection between subproduct systems together with their representations and tuples of operators that are the zeros of homogeneous polynomials in non commuting variables.

\begin{remark}
\emph{The notions that we are developing give a framework for studying tuples of operators satisfying relations given by homogeneous polynomials. One can go much further by considering subspaces of Fock spaces
and ``representations", i.e., maps of the Fock space into $B(H)$, that give a framework for studying tuples of operators satisfying arbitrary (not-necessarily homogeneous) polynomial and even analytic identities. Gelu Popescu \cite{Popescu06} has already begun developing such a theory.}
\end{remark}

We begin by setting up the usual notation. Let $\cI$ be a fixed set of indices, and let $\mb{C}\langle (x_i)_{i \in \cI}\rangle$ be the algebra of complex polynomials in the non commuting variables $(x_i)_{i \in \cI}$. We denote $\underline{x} = (x_i)_{i \in \cI}$, and we consider $\underline{x}$ as a ``tuple variable". We shall sometimes write $\mb{C}\langle \underline{x}\rangle$ for $\mb{C}\langle (x_i)_{i \in \cI}\rangle$. The set of all words in $\cI$ is denoted by $\mb{F}_\cI^+$. For a word $\alpha \in \mb{F}_\cI^+$, let $|\alpha|$ denote the length of $\alpha$, i.e., the number of letters in $\alpha$.

For every word $\alpha = \alpha_1 \cdots \alpha_k$ in $\cI$ denote $\underline{x}^\alpha = x_{\alpha_1} \cdots x_{\alpha_k}$. If $\alpha = 0$ is the empty word, then this is to be understood as $1$. $k$ is also referred to in this context as the \emph{degree} of the monomial $x^\alpha$. $\mb{C}\langle \underline{x}\rangle$ is by definition the linear span over $\mb{C}$ of all such monomials, and every element in $\mb{C}\langle \underline{x}\rangle$ is called a polynomial. 
A polynomial is called \emph{homogeneous} if it is the sum of monomials of equal degree. A \emph{homogeneous ideal} is a two-sided ideal that is generated by homogeneous polynomials.

If $\underline{T} = (T_i)_{i \in \cI}$ is a tuple of operators on a Hilbert space $H$ and $\alpha = \alpha_1 \cdots \alpha_k$ is a word with letters in $\cI$, we define
$$\underline{T}^\alpha = T_{\alpha_1} T_{\alpha_2} \cdots T_{\alpha_k} .$$
We consider the empty word $0$ as a legitimate word, and define $\underline{T}^0 = I_H$. If $p(x) = \sum_\alpha c_\alpha \underline{x}^\alpha \in \mb{C}\langle \underline{x}\rangle$, we define $p(\underline{T}) = \sum_\alpha c_\alpha \underline{T}^\alpha$.

If $E$ is a Hilbert space with orthonormal basis $\{e_i\}_{i \in \cI}$, An element $e_{\alpha_1} \otimes \cdots \otimes e_{\alpha_k} \in E^{\otimes k}$ will be written in short form as $e_{\alpha}$, where $\alpha = \alpha_1 \cdots \alpha_k$. If $p(x) = \sum_\alpha c_\alpha \underline{x}^\alpha \in \mb{C}\langle \underline{x}\rangle$, we define $p(e) = \sum_\alpha c_\alpha e_\alpha$. Here $e_0$ ($0$ the empty word) is understood as the vacuum vector $\Omega$.

\begin{proposition}\label{prop:XandI}
Let $E$ be a Hilbert space with orthonormal basis $\{e_i\}_{i \in \cI}$. There is an inclusion reversing correspondence between proper homogeneous ideals $I \triangleleft \mb{C}\langle \underline{x}\rangle$ and standard subproduct systems $X = \{X(n)\}_{n \in \mb{N}}$ with $X(1) \subseteq E$. When $|\cI| < \infty$ this correspondence is bijective.
\end{proposition}
\begin{proof}
Let $X$ be such a subproduct system. We define an ideal
\be\label{eq:I^X}
I^X := \textrm{span}\{p\in \mb{C}\langle\underline{x} \rangle : \exists n > 0 .  p(e) \in E^{\otimes n} \ominus X(n)\}.
\ee
Once it is established that $I^X$ is a two-sided ideal the fact that it is homogeneous will follow from the definition. Let $p\in \mb{C}\langle\underline{x} \rangle$ be such that $p(e) \in E^{\otimes n} \ominus X(n)$ for some $n > 0$. It suffices to show that for every monomial $\underline{x}^\alpha$ we have that $\underline{x}^\alpha p(\underline{x}) \in I^X$, that is,
\bes
e_\alpha \otimes p(e) \in E^{\otimes |\alpha| + n} \ominus X(|\alpha| + n).
\ees
But since $X$ is standard, $X(|\alpha| + n) \subseteq X(|\alpha|) \otimes X(n)$, thus
\bes
E^{\otimes |\alpha|} \otimes (E^{\otimes n} \ominus X(n)) \subseteq E^{\otimes |\alpha| + n} \ominus X(|\alpha| + n).
\ees
It follows that $I^X$ is a homogeneous ideal.

Conversely, let $I$ be a homogeneous ideal. We construct a subproduct system $X_I$ as follows. Let $I^{(n)}$ be the set of all homogeneous polynomials of degree $n$ in $I$. Define
\be\label{eq:X_I}
X_I(n) = E^{\otimes n} \ominus \{p(e) : p \in I^{(n)}\}.
\ee
Denote by $p_n$ the orthogonal projection of $E^{\otimes n}$ onto $X_I(n)$. To show that $X_I$ is a subproduct system it is enough (by symmetry) to prove that for all $m,n \in \mb{N}$
\bes
p_{m+n} \leq I_{E^{\otimes m}} \otimes p_n,
\ees
or, in other words, that
\be\label{eq:X_IinEX_I}
X_I(m+n) \subseteq E^{\otimes m} \otimes X_I(n).
\ee
Let $x \in X_I(m+n)$, let $\alpha \in \cI^m$, and let $q \in I^{(n)}$. Since $I$ is an ideal, $\underline{x}^\alpha q(\underline{x})$ is in $I^{(m+n)}$, thus $\langle x, e_\alpha \otimes q(e) \rangle = 0$. This proves (\ref{eq:X_IinEX_I}).

Assume now that $|\cI| < \infty$. We will show that the maps $X \mapsto I^X$ and $I \mapsto X_I$ are inverses of each other. Let $J$ be a homogeneous ideal in $\mb{C}\langle \underline{x}\rangle$. Then
\begin{align*}
I^{X_J} &= \textrm{span}\{p\in \mb{C}\langle\underline{x} \rangle : \exists n > 0 .  p(e) \in E^{\otimes n} \ominus X_J(n)\} \\
(*)&= \textrm{span}\{p\in \mb{C}\langle\underline{x} \rangle :  \exists n > 0 . p(e) \in \{q(e): q \in J^{(n)}\}\} \\
&= \textrm{span}\{p : \exists n > 0 . p\in  J^{(n)}\} \\
(**)&= J,
\end{align*}
where (*) follows from the definition of $X_J$, and (**) from the fact that $J$ is a homogeneous ideal.

For the other direction, let $Y$ be a standard subproduct subsystem of $F_E = \{E^{\otimes n}\}_{n \in \mb{N}}$. Clearly, $(I^Y)^{(n)} = \{p\in \mb{C}\langle\underline{x} \rangle : p(e) \in E^{\otimes n} \ominus Y(n)\}$.
Thus
\begin{align*}
X_{I^Y}(n) &= E^{\otimes n} \ominus \{p(e) : p \in (I^Y)^{(n)}\} \\
&= E^{\otimes n} \ominus  \{p(e) : p \in \{q\in \mb{C}\langle\underline{x} \rangle : q(e) \in E^{\otimes n} \ominus Y(n)\} \\
&= E^{\otimes n} \ominus  (E^{\otimes n} \ominus Y(n)) \\
&= Y(n).
\end{align*}
\end{proof}

We record the definitions of $I^X$ and $X_I$ from the above theorem for later use:

\begin{definition}
Let $E$ be a Hilbert space with orthonormal basis $\{e_i\}_{i \in \cI}$ ($|\cI|$ is not assumed finite). Given a homogeneous ideal $I \triangleleft \mb{C}\langle \underline{x} \rangle$, the subproduct system $X_I$ defined by (\ref{eq:X_I}) will be called the \emph{subproduct system associated with $I$}. If $X$ is a given subproduct subsystem of $F_E$, then the ideal $I^X$ of $\mb{C}\langle \underline{x} \rangle$ defined by (\ref{eq:I^X}) will be called the \emph{ideal associated with $X$}.
\end{definition}

We note that $X_I$ depends on the choice of the space $E$ and basis $\{e_i\}_{i \in \cI}$, but different choices will give rise to isomorphic subproduct systems.

\begin{proposition}\label{prop:change_of_variables}
Let $X$ and $Y$ be standard subproduct systems with $\dim X(1) = \dim Y(1) = d < \infty$. Then $X$ is isomorphic to $Y$ if and only if there is a unitary linear change of variables in $\mb{C}\langle x_1, \ldots, x_d \rangle$ that sends $I^X$ onto $I^Y$.
\end{proposition}

Fix some infinite dimensional separable Hilbert space $H$. As in classical algebraic geometry, given a homogeneous ideal $I \triangleleft \mb{C}\langle \underline{x}\rangle$, it is natural to introduce and to study the \emph{zero set of $I$}
\bes
Z(I) := \{\underline{T} = (T_i)_{i \in \cI} \in B(H)^\cI : \forall p \in I. p(\underline{T}) = 0\}.
\ees
Also, given a set $Z \subseteq B(H)^\cI$, one may form the following two-sided ideal in $\mb{C}\langle \underline{x}\rangle$
\bes
I(Z) := \{p \in \mb{C}\langle \underline{x}\rangle : \forall \underline{T} \in Z . p(\underline{T}) = 0\}.
\ees

In the following theorem we shall use the notation of \ref{subsec:representations}. This simple result is the justification for viewing subproduct systems as a framework for studying tuples of operators satisfying certain homogeneous polynomial relations.
\begin{theorem}\label{thm:repZ(I)}
Let $E$ be a Hilbert space with orthonormal basis $\{e_i\}_{i \in \cI}$ (not necessarily with $|\cI| < \infty$), and let $I$ be a proper homogeneous ideal in $\mb{C}\langle (x_i)_{i \in \cI}\rangle$. Let $X_I$ be the associated subproduct system. Let $T_1 : E \rightarrow B(H)$ be a given representation of $E$. Define a tuple $\underline{T} = (T(e_i))_{i \in \cI}$. Construct the family of maps $T = \{T_n\}_{n\in \mb{N}}$, with $T_n:X(n) \rightarrow B(H)$ as in the paragraphs before Proposition \ref{prop:representation}. Then $T$ is a representation of $X$ if and only if $\underline{T} \in Z(I)$.
\end{theorem}
\begin{proof}
On the one hand, $E^{\otimes n} \ominus X_I(n) = \overline{\textrm{span}}\{p(e) : p \in I^{(n)}\}$. On the other hand, for every $p \in I^{(n)}$ and every $h \in H$,
\bes
\widetilde{T}^n (p(e) \otimes h) = p(\underline{T})h.
\ees
Hence, the Theorem follows from Proposition \ref{prop:representation}.
\end{proof}

\begin{lemma}\label{lem:shiftpoly}
Let $J \triangleleft \mb{C}\langle (x_i)_{i \in \cI}\rangle$, $|\cI| < \infty$, be a proper homogeneous ideal. Let $S^{X_J}$ be the $X_J$-shift representation, and define $\underline{T} = (T_i)_{i \in \cI}$ by $T_i = S^{X_J}(e_i)$, $i \in \cI$. If $p \in \mb{C} \langle \underline{x} \rangle$ is a homogeneous polynomial, then $p(\underline{T}) = 0$ if and only if $p \in J$.
\end{lemma}
\begin{proof}
The ``if" part follows from Theorem \ref{thm:repZ(I)}. For the ``only if" part, let $p \notin J$ be a homogeneous polynomial of degree $n$. Applying $p(\underline{T})$ to the vacuum vector $\Omega$, we have
\bes
p(\underline{T}) \Omega = P p(e),
\ees
where $P$ is the orthogonal projection of $E^{\otimes n}$ onto $X_J(n)$. But as $p \notin J$, $p(e)$ is not in $E^{\otimes n} \ominus X_J(n) = \ker P$, thus $P p(e) \neq 0$. In particular, $p(\underline{T}) \neq 0$.
\end{proof}

We have the following noncommutative projective Nullstellansatz.
\begin{theorem}
Let $H$ be a fixed infinite dimensional separable Hilbert space.
Let $J$ be a homogeneous ideal in $\mb{C}\langle (x_i)_{i \in \cI}\rangle$, with $|\cI| < \infty$. Then 
\bes
I(Z(J)) = J.
\ees
In particular,
$Z(J) = \{\underline{0} = (0,0,\ldots)\}$ if and only if $J$ is the ideal generated by all the $x_i, i \in \cI$.
\end{theorem}
\begin{proof}
%
$I(Z(J)) \supseteq J$ is immediate. To see the converse, first note that equality is obvious when $J = \mb{C}\langle \underline{x} \rangle$, so we may assume that $J$ is proper. Also note that since $J$ is homogeneous $Z(J)$ is scale invariant. From this it follows that $I(Z(J))$ is also a homogeneous ideal. Indeed, if $h,g \in H$, and $p(\underline{x}) = \sum_\alpha c_\alpha \underline{x}^\alpha \in I(Z(J))$, then for all $\lambda \in \mb{C}$ one has for every tuple $\underline{T} = (T_i)_{i \in \cI} \in Z(I)$,
\bes
0 = \langle p(\lambda\underline{T})h,g \rangle = \sum_k \left(\sum_{|\alpha| = k} c_\alpha \langle \underline{T}^\alpha h,g\rangle\right) \lambda^{k},
\ees
and since a nonzero univariate polynomial has only finitely many zeros, it follows the homogeneous components of $p$ are all in $I(Z(J))$.

Assume now that $p$ is a homogeneous polynomial not in $J$. Let $S^{X_J}$ be the $X_J$-shift representation, and define $\underline{T} = (T_i)_{i \in \cI}$ by $T_i = S^{X_J}(e_i)$, $i \in \cI$. It is clear that $B(H)^\cI$ contains some unitarily equivalent copy of $\underline{T}$, which we also denote by $\underline{T}$. By Theorem \ref{thm:repZ(I)}, $\underline{T} \in Z(J)$. But by Lemma \ref{lem:shiftpoly}, $p(\underline{T}) \neq 0$, so $p \notin I(Z(J))$. This completes the proof.
\end{proof}

\section{Universality of the shift: universal algebras and models}\label{sec:universal}
In \cite{Arv98}, Arveson established a model for commuting, row-contractive tuples.
Using an idea from that paper that appeared also in \cite{BBD03} and \cite{Dey07} -- an idea that rests upon Popescu's ``Poisson Transform" introduced in \cite{Popescu99} (and pushed forward in \cite{MS08} and \cite{Popescu06}) -- we construct below a model for representations of subproduct systems. Roughly speaking, we will show that every representation of a subproduct system $X$ is a piece of a scaled inflation of the shift. Our model should be compared with a similar model obtained by Popescu in \cite{Popescu06}.
We will also see below that the operator algebra generated by the shift $S^X$ is the universal operator algebra generated by a representation of $X$.

\subsection{Notation for this section}\label{subsec:notation}

We continue to use the notation set in the previous section.
Let $X$ be a standard subproduct system of Hilbert spaces over $\mb{N}$, to be fixed throughout this section. Let $p_n:E^{\otimes n} \rightarrow X(n)$ be the projections. Denote $E = X(1)$. Let $\{e_i\}_{i \in \cI}$ be an orthonormal basis for $E$, fixed once and for all.

We denote the standard $X$-shift tuple by $\underline{S^X} = (S^X_i)_{i \in \cI}$ , and we denote the standard $X$-shift representation of $X$ on $\mathfrak{F}_X$ by $S^X$. We consider $\mathfrak{F}_X$ to be a subspace of the full Fock space $\mathfrak{F}$, we denote the full shift by
$\underline{S} = (S_i)_{i\in \cI}$, and we denote the full shift representation of $F$ on $\mathfrak{F} := \mathfrak{F}_F$ by $S$.

Given a representation $T: X \rightarrow B(H)$, we will write $\underline{T} = (T_i)_{i \in \cI}$ for the tuple $(T(e_i))_{i\in \cI}$.

We denote by $\cA_X$ the unital algebra
$$\cA_X := \overline{\textrm{span}} \{\underline{S^X}^\alpha : \alpha \in  \mb{F}_\cI^+\} .$$

We denote by $\cE_X$ the operator system
$$\cE_X := \overline{\textrm{span}}\cA_X \cA_X^*,$$
and by $\cT_X = C^*(\underline{S^X})$ the C$^*$-algebra generated by $S^X_i$, $i\in \cI$ and $I_{\mathfrak{F}_X}$.
We denote by $\cK(\mathfrak{F}_X)$ the algebra of compact operators on $\mathfrak{F}_X$

If $T$ and $U$ are two representations of $X$ on Hilbert spaces $H$ and $K$, respectively, then we define
\bes
T \oplus U
\ees
to be the representation of $X$ on $H \oplus K$ given by $(T \oplus U)(x) = T(x) \oplus U(x)$.
We also define
\bes
T \otimes I_K
\ees
to be the representation of $X$ on $H \otimes K$ given by $(T \otimes I_K)(x) = T(x) \otimes I_K$.

\subsection{Popescu's ``Poisson Transform"}

After obtaining the results of this section, we discovered that they were obtained earlier by Popescu \cite{Popescu06}. We are presenting them here since they are important for the rest of this paper.

\begin{proposition}\label{prop:KinE}
$\cK(\mathfrak{F}_X) \subseteq \cE_X$.
\end{proposition}
\begin{proof}
By the definition of representation, we have that $S(e_\alpha) = \underline{S}^\alpha$. By Definition \ref{def:dilation} and the remarks following it, we have that $\underline{S}^{\alpha *}\big|_{\mathfrak{F}_X} = \left(\underline{S^X}\right)^{\alpha *}$ for all $\alpha$. Let $x \in X(n)$. Let $|\alpha|=k$. If $n<k$ then $\left(\underline{S^X}\right)^{\alpha *} x = 0$. If $n \geq k$, then since $X(n) \subseteq X(k) \otimes X(n-k)$, we may write $x = \sum_i x_k^i \otimes x_m^i$, where $x_k^i \in  X(k)$, $x_m^i \in X(m)$, and $m = n-k$. We have then
\be\label{eq:SXstar}
\underline{S^X}^{\alpha *} x = \underline{S}^{\alpha *} \sum_i x_k^i \otimes x_m^i = \sum_i \langle e_\alpha , x_k^i \rangle x_m^i \in X(m).
\ee
We then have for $x \in X(n)$:
\bes
\left(I - \sum_{|\alpha|=k}\underline{S^X}^{\alpha}\underline{S^X}^{\alpha *} \right)x =
\begin{cases}
x , & n < k \cr
x - \sum_{|\alpha|=k}\underline{S^X}^{\alpha}\sum_i \langle e_\alpha , x_k^i \rangle x_m^i  , & n \geq k
\end{cases} .
\ees
But
\begin{align*}
\sum_{|\alpha|=k}\underline{S^X}^{\alpha}\sum_i \langle e_\alpha , x_k^i \rangle x_m^i
&= \sum_{|\alpha|=k} p_n\left(\sum_i \langle e_\alpha , x_k^i \rangle e_\alpha \otimes x_m^i \right) \\
&= p_n\left(\sum_i \sum_{|\alpha|=k} \langle e_\alpha , x_k^i \rangle e_\alpha \otimes x_m^i \right) \\
&= p_n\left(\sum_i x_k^i \otimes x_m^i \right) \\
&= x.
\end{align*}
We thus conclude that $I - \sum_{|\alpha|=k}\underline{S^X}^{\alpha}\underline{S^X}^{\alpha *} = P_{W}$, where $W = \mb{C} \oplus X(1) \oplus \cdots \oplus X(k-1)$. In particular,
\be\label{eq:PC}
I - \sum_{i\in \cI}S^X_{i} \left(S^X\right)^{*}_i = P_{\mb{C}}.
\ee
Equations (\ref{eq:SXstar}) and (\ref{eq:PC}) give
\bes
(\underline{S^X})^\beta \left(I - \sum_{i\in \cI}S^X_{i} \left(S^X\right)^{*}_i \right)\underline{S^X}^{\alpha *} x
= p_{|\beta|} \langle e_\alpha, x  \rangle e_\beta.
\ees
As the elements $p_{|\beta|}e_\beta$ span $\mathfrak{F}_X$, it follows that $\cK(\mathfrak{F}_X) \subseteq \cE_X$.
\end{proof}

Given a representation $T$ of $X$ on a Hilbert space $H$ and given an integer $m \in \mb{N}$, we denote by $m \cdot T$ the representation
\bes
m \cdot T : X \rightarrow B(\underbrace{H \oplus H \oplus \cdots \oplus H}_{m \textrm{ times}})
\ees
given by $m \cdot T(x) = \underbrace{T(x) \oplus T(x) \oplus \cdots \oplus T(x)}_{m \textrm{ times}}$.
$\underline{T}$ is a row contraction (i.e., $\sum_{i \in \cI} T_i T_i^* \leq I_H$) if and only if $T$ is completely contractive.
When $\underline{T}$ is a row contraction the \emph{defect operator} $\Delta (\underline{T})$ is defined as
\bes
\Delta(\underline{T}) = I - \sum_{i \in \cI} T_i T_i^*,
\ees
and the \emph{Poisson Kernel} \cite{Popescu99} associated with $\underline{T}$ is the family of isometries $\{K_r\left(\underline{T}\right)\}_{0\leq r < 1}$
\bes
K_r\left(\underline{T}\right):H \rightarrow \mathfrak{F} \otimes H,
\ees
given by
\bes
K_r\left(\underline{T}\right) h = \sum_{\alpha \in \mb{F}_\cI^+} e_\alpha \otimes \left(r^{|\alpha|}\Delta(r\underline{T})^{1/2} \underline{T}^{\alpha *} h \right).
\ees
(See the beginning of \cite[Section 8]{Popescu99} for the remark that $\underline{T}$ has ``property (P)", and \cite[Lemma 3.2]{Popescu99} for the fact that these are isometries). When it makes sense, we also define $K_1\left(\underline{T}\right)$ by the same formula with $r=1$. The \emph{Poisson transform} is then defined as a map
\bes
\Phi = \Phi_{\underline{T}}: C^*(\underline{S}) \rightarrow B(H)
\ees
\bes
\Phi(a) = \Phi_{\underline{T}}(a) = \lim_{r \nearrow 1}K_r\left(\underline{T}\right)^* (a \otimes I) K_r\left(\underline{T}\right).
\ees
By \cite[Theorem 3.8]{Popescu99}, $\Phi$ is a unital, completely positive, completely contractive, satisfies
\bes
\Phi(\underline{S}^\alpha \underline{S}^{\beta *}) = \underline{T}^\alpha \underline{T}^{\beta *},
\ees
and is multiplicative on $Alg(\underline{S},I_{\mathfrak{F}})$, the algebra generated by $\underline{S}$ and $I_{\mathfrak{F}}$ ($\Phi$ is in fact an $Alg(\underline{S},I_{\mathfrak{F}})$-morphism).

\begin{theorem}\label{thm:CP1}
Let $T$ be a c.c. representation of $X$ on $H$. There exists a unital, completely positive, completely contractive map
\bes
\Psi: \cE_X \rightarrow B(H)
\ees
that satisfies
\bes
\Psi\left((\underline{S^X})^\alpha (\underline{S^X})^{\beta *}\right) = \underline{T}^\alpha \underline{T}^{\beta *} \,\, , \,\, \alpha, \beta \in \mb{F}_\cI^+
\ees
and
\be\label{eq:Amorphism}
\Psi(ab) = \Psi(a) \Psi(b) \,\, , \,\, a \in \cA_X, b \in \cE_X.
\ee
\end{theorem}
\begin{proof}
By the lemma below, the range of $K_r\left(\underline{T}\right)$ is contained in $\mathfrak{F}_X \otimes H$ for all $0 \leq r < 1$, thus
\bes
(P_{\mathfrak{F}_X} \otimes I_H) K_r\left(\underline{T}\right) = K_r\left(\underline{T}\right).
\ees
We may then define
\begin{align*}
\Psi(\underline{T})(\left((\underline{S^X})^\alpha (\underline{S^X})^{\beta *}\right))
&= \lim_{r \nearrow 1}K_r\left(\underline{T}\right)^* \left(\left((\underline{S^X})^\alpha (\underline{S^X})^{\beta *}\right) \otimes I\right) K_r\left(\underline{T}\right) \\
(*)&= \lim_{r \nearrow 1}K_r\left(\underline{T}\right)^* \left(\left(\underline{S}^\alpha \underline{S}^{\beta *}\right) \otimes I\right) K_r\left(\underline{T}\right) \\
&= \underline{T}^\alpha \underline{T}^{\beta *} ,
\end{align*}
where in (*) we have made use of the coinvariance of $\mathfrak{F}_X$ under $\underline{S}$.
This obviously extends to the desired map on $\cE_X$.
\end{proof}

\begin{lemma}
$K_r\left(\underline{T}\right)H \subseteq \mathfrak{F}_X \otimes H$.
\end{lemma}

\begin{proof}
Let $h \in H$. It suffices to show that for all $n \in \mb{N}$, the element
\bes
\sum_{|\alpha| = n} e_\alpha \otimes \left(r^n\Delta(r\underline{T})^{1/2} \underline{T}^{\alpha *} h \right) = (I \otimes r^n\Delta(r\underline{T})^{1/2}) \sum_{|\alpha| = n} e_\alpha \otimes (\underline{T}^{\alpha *} h)
\ees
is in $X(n) \otimes H$. However, $X(n) \otimes H$ (considered as a subspace of $E^{\otimes n} \otimes H$) is reduced by $(I \otimes r^n\Delta(r\underline{T})^{1/2})$, so it is enough to show that
\bes
\xi := \sum_{|\alpha| = n} e_\alpha \otimes (\underline{T}^{\alpha *} h) \in X(n) \otimes H.
\ees
Let $x \in E^{\otimes n} \ominus X(n)$ and $g \in H$. The proof will be completed by showing that
$\langle \xi, x \otimes g \rangle = 0$.
\begin{align*}
\langle \xi, x \otimes g \rangle
&= \sum_{|\alpha|=n} \langle e_\alpha \otimes T(e_\alpha)^* h, x \otimes g \rangle \\
&= \sum_{|\alpha|=n} \langle e_\alpha, x \rangle \langle h, T(e_\alpha)g \rangle \\
&= \left\langle h, T\left(\sum_{|\alpha|=n} \langle e_\alpha, x \rangle e_\alpha \right) g \right\rangle \\
&= \langle h, \widetilde{T}^n(x \otimes g) \rangle ,
\end{align*}
and by Proposition \ref{prop:representation}, the last expression in this chain of equalities is $0$.
\end{proof}

\subsection{The universal algebra generated by a tuple subject to homogeneous polynomial identities}

\begin{theorem}\label{thm:universal}
$J \triangleleft \mb{C}\langle (x_i)_{i \in \cI}\rangle$, be a homogeneous ideal.
Then $\cA_{X_J}$ is the universal unital operator algebra generated by a row contraction in $Z(J)$, that is: $\cA_{X_J}$ is a norm closed unital operator algebra generated by a tuple in $Z(J)$, (namely, $(S^{X_J}_i)_{i\in \cI}$), and
if $\mathcal{B} \subseteq B(H)$ is another norm closed unital operator algebra generated by a row contraction $(T_i)_{i\in\cI} \in Z(J)$, then there is a unique unital and completely contractive homomorphism $\varphi$ of $\cA_{X_J}$ onto $\mathcal{B}$, such that $\varphi(S^{X_J}_i) = T_i$ for all $i \in \cI$.
\end{theorem}
\begin{proof}
This follows from Theorems \ref{thm:repZ(I)} and \ref{thm:CP1}.
\end{proof}

\subsection{A model for representations: every completely bounded representation of $X$ is a piece of an inflation of $S^X$}

We will now construct a model for representations of subproduct systems. 
In \cite[Section 2]{Popescu06}, a similar but different model -- that includes also a fully coisometric part and not only the shift -- has been obtained.

\begin{theorem}\label{thm:model2}
Let $\underline{T}$ be a completely bounded representation of the subproduct system $X$ on a separable Hilbert space $H$, and let $K$ be an infinite dimensional, separable Hilbert space. Then for all $r > \|T\|_{cb}$, $T$ is unitarily equivalent to a piece of
\be\label{eq:shift_r}
S^X \otimes rI_K.
\ee
Moreover, $\|T\|_{cb}$ is equal the infimum of $r$ such that $T$ is a piece of an operator as in  (\ref{eq:shift_r}).
\end{theorem}
\begin{proof}
It is known that $\|T\|_{cb} = \|(T_i)_{i \in \cI}\|_{row}$, where $T_i = T(e_i)$. Thus if $r > r_0 = \|T\|_{cb}$, then
$\sum_{i \in \cI} T_i T_i^* \leq r_0^2 I < r^2 I$. Put $W_i = r^{-1}T_i$, so $\sum_{i\in\cI}W_i W_i^* \leq r_0^2 /r^2 I$. Then $K_1\left(\underline{W}\right)$ is an isometry (it is equal to $K_{r_0/r}(r/r_0 \underline{W} )$, and $r/r_0 \underline{W} $ is a row contraction). Thus we may define a map (as in the proof of Theorem \ref{thm:CP1})
\bes
\Psi: B(\mathfrak{F}_X) \rightarrow B(H)
\ees
by
\bes
\Psi(a) = K_1\left(\underline{W}\right)^* \left(a \otimes I\right) K_1\left(\underline{W}\right).
\ees
$\Psi$ is a normal completely positive unital map that satisfies
\bes
\Psi\left((\underline{S^X})^\alpha (\underline{S^X})^{\beta *}\right) = \underline{W}^\alpha \underline{W}^{\beta *} \,\, , \,\, \alpha, \beta \in \mb{F}_\cI^+ .
\ees
Since $\Psi$ is normal it has a \emph{normal} minimal Stinespring dilation $\Psi(a) = V^* \pi(a) V$, with $\pi : B(\mathfrak{F}_X) \rightarrow B(L)$ a normal $*$-homomorphism and $V : H \rightarrow L$ an isometry. It is well known that $\pi$ is equivalent to a multiple of the identity representation.
Thus we obtain, up to unitary equivalence and after identifying $H$ with $VH$, that $r^{-1} T_i = P_H \pi(S^X_i) P_H = P_H (S^X_i \otimes I_G) P_H$, for some Hilbert space $G$. To see that $T$ is a piece of $S^X \otimes I_G$ we need to show that $(S^X_i \otimes I_G)^* \big|_H = T^*_i$ for all $i \in \cI$.
In other words, we need to show that $P_H \pi(S^X_i) = P_H \pi(S^X_i) P_H$.
But, for all $b \in \cE_X$,
\begin{align*}
P_H \pi(S^X_i) \pi(b) P_H &= P_H \pi(S^X_i b) P_H \\
&= \Psi(S^X_i b) \\
(*)&= \Psi(S^X_i) \Psi (b) \\
&= P_H \pi(S^X_i) P_H \pi(b) P_H,
\end{align*}
where (*) follows from (\ref{eq:Amorphism}). By Proposition \ref{prop:KinE}, the strong operator closure of $\cE_X$ is $B(\mathfrak{F}_X)$. $P_H \pi(S^X_i) = P_H \pi(S^X_i) P_H$ now follows from the minimality and normality of the dilation.

It is clear that $r^{-1}T$ is a also piece of $S^X \otimes I_K$ for every $K$ with $\dim K \geq \dim G$, so we may choose $K$ to be infinite dimensional.

We want to show that necessarily $\dim K \geq \dim H$. Since $S^X \otimes I_K$ is a dilation of $r^{-1}T$, $I_L - \sum_{i \in \cI} S^X_i (S^X_i)^* \otimes I_K$ is a dilation of $I_H - \sum_{i \in \cI} r^{-2}T_i T_i^*$. But the latter operator is invertible so it has rank $\dim H$. Thus the rank of $P_\mb{C} \otimes I_K = I_L - \sum_{i \in \cI} S^X_i(S^X_i)^* \otimes I_K$, which is $\dim K$, must be greater.

Now the final assertion is clear.
\end{proof}

We can now obtain a general von Neumann inequality.
\begin{theorem}\label{thm:vNinequality}
Let $X$ be a subproduct system, and let $T$ be a c.c. representation of $X$ on a Hilbert space $H$. Let $\{e_1, \ldots, e_d\}$ be an orthonormal set in $X(1)$, and define $T_i = T(e_i)$ and $S^X_i = S^X(e_i)$ for $i=1,\ldots,d$. Then for every polynomials $p$ and $q$ in $d$ non commuting variables,
\bes
\|p(T_1, \ldots, T_d) q(T_1, \ldots, T_d)^*\| \leq \|p(S^X_1, \ldots, S^X_d) q(S^X_1, \ldots, S^X_d)^*\|.
\ees
\end{theorem}
\begin{proof}
Since $T$ is a piece of $S^X \otimes r I_K$ for all $r>1$, we have
\bes
p(T_1, \ldots, T_d) q(T_1, \ldots, T_d)^*  = P \Big( p(rS_1, \ldots, rS_d) q(rS_1, \ldots, rS_d)^* \otimes I_K \Big) P
\ees
for some projection $P$, and the result follows by taking $r\searrow 1$.
\end{proof}

\section{The operator algebra associated to a subproduct system}\label{sec:operator_algebra}

\subsection{}
Let $X$ be a subproduct system. Recall the definitions of $\cA_X$ and $\cE_X$  from \ref{subsec:notation}. If $\{e_i\}_{i \in \cI}$ is an orthonormal basis for $X(1)$, then $\cA_X$ is the unital operator algebra generated by $(S^X_i)_{i \in \cI}$ with $S^X_i = S^X(e_i)$. If $\{f_i\}_{i \in \cI}$ is another orthonormal basis then the tuple $(S^X(f_i))_{i \in \cI}$ is not necessarily unitarily equivalent to $(S^X_i)_{i \in \cI}$. For instance (with the above notation), if $X$ and $\{e_1, e_2\}$ are as in Example \ref{expl:notmax}, and
\bes
f_1 = \frac{1}{\sqrt{2}}(e_1 + e_2) \quad, \quad f_2 = \frac{1}{\sqrt{2}}(e_1 - e_2),
\ees
then $S^X_1, S^X_2$ are partial isometries, whereas $T_1 = S^X(f_1)$ and $T_2 = S^X(f_2)$ are not.
Thus, the unitary equivalence of the row $(\underline{S}^X_i)$ does not determine the isomorphism class of the subproduct system $X$.
\begin{proposition}\label{prop:UE_ISO}
Let $X$ and $Y$ be two subproduct systems with $X(1) = E$ and $Y(1) = F$. Assume that $\{e_i\}_{i \in \cI}$ is an orthonormal basis for $E$ and that $\{f_i\}_{i\in\cI}$ is an orthonormal basis for $F$. Then the shifts $(S^X_i)_{i \in \cI}$ and $(S^Y_i)_{i \in \cI}$ are unitarily equivalent as rows (i.e., there exists a unitary $V: \mathfrak{F}_X \rightarrow \mathfrak{F}_Y$ such that $V S^X_i = S^Y_i V$ for all $i \in \cI$), if and only if there is an isomorphism of subproduct systems $W: X \rightarrow Y$ such that $We_i = f_i$ for all $i \in \cI$.
\end{proposition}
\begin{proof}
If $X$ and $Y$ are isomorphic with the isomorphism $W$ sending $e_i$ to $f_i$, then define a unitary $V: \mathfrak{F}_X \rightarrow \mathfrak{F}_Y$ by
\bes
V = \bigoplus_{n \in \mb{N}} W\big|_{X(n)}.
\ees
$V S^X_i = S^Y_i V$ follows from the properties of $W$. Conversely, a unitary $V$ intertwining $\underline{S}^X$ and $\underline{S}^Y$ must send $\Omega_X$ to $\Omega_Y$. Indeed, such a unitary must send $\{\Omega_X\}^\perp$ (which is equal to $\vee_i \textrm{Im}S^X_i$) onto a subspace of $\{\Omega_Y\}^\perp$ that has codimension $1$ in $\mathfrak{F}_Y$, thus it must send $\{\Omega_X\}^\perp$ onto $\{\Omega_Y\}^\perp$. It follows that $V \Omega_X = \Omega_Y$. Thus, given a unitary $V$ intertwining $\underline{S}^X$ and $\underline{S}^Y$, we may define $W\big|{X(n)} : X(n) \rightarrow Y(n)$ by
\bes
W S^X_\alpha \Omega = V S^X_\alpha \Omega = S^Y_\alpha \Omega ,
\ees
for all $|\alpha| = n$, and it is easy to see that the maps $W\big|_{X(n)}$ assemble to form an isomorphism of subproduct systems.\end{proof}

In the example preceding the proposition, we saw how the shift ``tuple" $(S^X_1,S^X_2)$ depends essentially on the choice of basis in $E$. However, the closed unital algebra generated by $(S^X_1, S^X_2)$ is isomorphic to the one generated by $(T_1, T_2)$. Similar remarks hold for $\cE_X$ and $\cT_X$.


\begin{example}
\emph{Let $X$ be the subproduct system given by $X(0) = \mb{C}, X(1) = \mb{C}^2$ and $X(n) = 0$ for all $n \geq 2$. Let $Y$ be the subproduct system given by $Y(0) = Y(1) = Y(2) = \mb{C}$ and $Y(n) = 0$ for all $n \geq 3$. Then since $\cE_X$ and $\cE_Y$ contain the compact operators on $\cF_X$ and $\cF_Y$ (the Fock spaces), we have $\cE_X = \cT_X \cong M_3(\mb{C}) \cong \cT_Y = \cE_Y$.}

\emph{On the other hand, let $\{e_1, e_2\}$ be an orthonormal basis for $X(1)$. Then if $\Omega$ is the vacuum vector, then $\cA_X$ is generated by $S^X(\Omega) = I,S^X(e_1),S^X(e_2)$. In the base $\{\Omega,e_1,e_2\}$ for $\cF_X$, these operators have the form}
\bes
\begin{pmatrix}
  1 & 0 & 0  \\
  0 & 1 & 0  \\
  0 & 0 & 1  \\
\end{pmatrix},
\begin{pmatrix}
  0 & 0 & 0  \\
  1 & 0 & 0  \\
  0 & 0 & 0  \\
\end{pmatrix},
\begin{pmatrix}
  0 & 0 & 0  \\
  0 & 0 & 0  \\
  1 & 0 & 0  \\
\end{pmatrix}.
\ees
\emph{Thus,}
\bes \cA_X \cong \left\{
\begin{pmatrix}
  a & 0 & 0  \\
  b & a & 0  \\
  c & 0 & a  \\
\end{pmatrix} \Big| a,b,c \in \mb{C} \right\}.
\ees
\emph{On the other hand, $\cA_Y$ is generated by}
\bes
I =
\begin{pmatrix}
  1 & 0 & 0  \\
  0 & 1 & 0  \\
  0 & 0 & 1  \\
\end{pmatrix},
S^Y(f_1) = \begin{pmatrix}
  0 & 0 & 0  \\
  1 & 0 & 0  \\
  0 & 1 & 0  \\
\end{pmatrix},
\big(S^Y(f_1)\big)^2 = \begin{pmatrix}
  0 & 0 & 0  \\
  0 & 0 & 0  \\
  1 & 0 & 0  \\
\end{pmatrix},
\ees
\emph{where $\{f_1\}$ is an orthonormal basis for $Y(1)$. Thus}

\bes \cA_Y \cong \left\{
\begin{pmatrix}
  a & 0 & 0  \\
  b & a & 0  \\
  c & b & a  \\
\end{pmatrix} \Big| a,b,c \in \mb{C} \right\}.
\ees
\emph{So $\cA_X \ncong \cA_Y$ (in $\cA_X$ the solutions of $T^2 = 0$ form a two dimensional subspace, and in $\cA_Y$ they form
a one dimensional subspace). }
\end{example}

\subsection{$\cA_X$ as a graded algebra.}
For every subproduct system $X$ there exists a unique completely contractive multiplicative linear functional
$\rho_0 : \cE_X \rightarrow \mb{C}$ that sends $\lambda I$ to $\lambda$ and $S^X_\alpha$ to $0$ when $|\alpha| > 0$.
The existence of $\rho_0$ follows from Theorem \ref{thm:CP1} (using the Poisson Transform),
but it is also clear that $\rho_0$ is just the vector state associated with the vacuum vector $\Omega_X$:
\bes
\rho_0(T) = \langle T \Omega_X, \Omega_X \rangle \,\, , \,\, T \in \cA_X.
\ees
$\rho_0$ can be considered also as a conditional expectation $\rho_0 : \cA_X \rightarrow \mb{C} \cdot \Omega_X$, inducing a direct sum
\be\label{eq:direct_sum}
\cA_X = \rho_0 \cA_X \oplus \ker \rho_0 = \mb{C}\cdot I \oplus \sum_i S_i^X \cA_X .
\ee
$\cA_X$ contains a dense graded subalgebra, with the homogeneous elements of degree $n$ being $S^X(\xi)$, where $\xi \in X(n)$. To be precise, we have the following proposition.

\begin{proposition}
Every $T \in \cA_X$ can be written in a unique way as
\bes
T = \sum_{n = 0}^\infty T_n,
\ees
where $T_n \in \overline{\textrm{span}}\{S^X(\xi) : \xi \in X(n)\}$ and the sum is Cesaro convergent in the norm topology.
\end{proposition}
\begin{proof}
The proof uses a familiar gadget in operator algebra theory, the \emph{gauge action of the torus}.
For every $t \in [-\pi,\pi]$, let $W_t : X \rightarrow X$ be the subproduct system automorphism given by
\bes
X(n) \ni \xi \mapsto e^{i nt} \xi \in X(n).
\ees
The \emph{gauge action} on $\cA_X$ is given by
\bes
\gamma_t(T) = W_t T W_t^* \,\, , \,\, T \in \cA_X.
\ees
Note that if $\alpha \in \cI^n$, then
\bes
\gamma_t(S^X_\alpha) = e^{i nt}S^X_\alpha ,
\ees
and it follows also that for all $T \in \overline{\textrm{span}}\{S^X(\xi) : \xi \in X(n)\}$,
\bes
\gamma_t(T) = e^{i nt}T .
\ees
Moreover, for all $T \in \cA_X$, the path $t \mapsto \gamma_t(T)$ is strong operator continuous.
Given $T \in \cA_X$, we define
\bes
\Phi_n(T) = \frac{1}{2 \pi}\int_{-\pi}^{\pi} \gamma_t(T) e^{-i nt} dt,
\ees
where this is interpreted in the strong operator sense. In more detail,
$\frac{1}{2 \pi}\int_{-\pi}^{\pi} \gamma_t(T) e^{- i nt} dt$ is the operator that for all $h \in H$ acts as
\bes
\left(\frac{1}{2 \pi}\int_{-\pi}^{\pi} \gamma_t(T) e^{- i nt} dt\right) h = \frac{1}{2 \pi}\int_{-\pi}^{\pi}  \gamma_t(T)
e^{- i nt} h dt,
\ees
where the integral on the right is usual vector valued integration.
It follows that if $T \in \overline{\textrm{span}}\{S^X(\xi) : \xi \in X(m)\}$ then
\bes
\Phi_n(T) =
\begin{cases}
T , & \textrm{if } m=n \cr
0  , & \textrm{else.}
\end{cases}
\ees
$\Phi_n$ is a completely contractive linear map. For a finite sum $\sum_{m=0}^N T_m$ with
$T_m \in \overline{\textrm{span}}\{S^X(\xi) : \xi \in X(m)\}$ and $N \geq n$,
\bes
\Phi_n\left(\sum_{m=0}^N T_m \right) = T_n.
\ees
As such finite sums are dense in $\cA_X$, it follows that the range of $\Phi_n$ is equal to
$\overline{\textrm{span}}\{S^X(\xi) : \xi \in X(n)\}$.

Define linear maps on $\cA_X$ by
\bes
\Psi_N(T) = \sum_{n=0}^{N}\left(1 - \frac{n}{N}\right)\Phi_n(T).
\ees
Our goal is to prove that $\sum_n \Phi_n(T)$ is Cesaro convergent to $T$ in the norm topology, that is, to show that for
all $T \in \cA_X$,
\bes
\left\|\Psi_N(T) - T \right\| \stackrel{N \rightarrow \infty}{\longrightarrow} 0.
\ees
But
\begin{align*}
\sum_{n=0}^N \left(1 - \frac{n}{N} \right)  \Phi_n(T) h
&= \frac{1}{2 \pi}\int_{-\pi}^{\pi} \sum_{n=0}^N \left(1 - \frac{n}{N} \right) e^{-int} \gamma_t (T) h dt  \\
&= \frac{1}{2 \pi}\int_{-\pi}^{\pi} \sum_{n=-N}^N \left(1 - \frac{n}{N} \right) e^{-int} \gamma_t (T) h  dt \\
&= \frac{1}{2 \pi}\int_{-\pi}^{\pi} F_N(t) \gamma_t (T) h dt ,
\end{align*}
where $F_N(t)$ is the Fej\'er  Kernel (see page 12 in \cite{Katznelson})
\bes
F_N(t) = \frac{1}{N+1}\left(\frac{\sin\frac{(N+1)t}{2} }{\sin\frac{t}{2}} \right)^2.
\ees
From the above integral representation and the fact that $\int |F_N(t)|dt = 2\pi$, it follows that
$\Psi_N$ is a contraction for all $N$.

Let $T \in \cA_X$ and fix $\epsilon > 0$. Then there is a $T_\epsilon = \sum_{m=0}^N T_m$ such that $\|T - T_\epsilon\| < \epsilon$. $\sum_n \Phi_n(T_\epsilon)$ converges to $T_\epsilon$, so
$\Psi_n(T_\epsilon)$ converges to $T_\epsilon$. There is some $M$ such that for $n > M$, $\|T_\epsilon - \Psi_n(T_\epsilon)\|<\epsilon$, so for $n > M$
\begin{align*}
\|T - \Psi_n(T)\| &\leq \|T - T_\epsilon\| + \|T_\epsilon - \Psi_n(T_\epsilon)\| + \|\Psi_n(T_\epsilon - T)\| \\
&< \epsilon + \epsilon + \epsilon .
\end{align*}
That shows that $\sum_n \Phi_n(T)$ is Cesaro convergent to $T$, and it remains to prove the uniqueness assertion. Assume that $T = \sum_n T_n$, where the sum is Cesaro convergent to $T$, and $T_n \in \overline{\textrm{span}}\{S^X(\xi) : \xi \in X(n)\}$. Then for all $N > n$,
\bes
\Phi_n\left(\sum_{m=0}^N  \left(1 - \frac{m}{N} \right) T_m\right) = \left(1 - \frac{n}{N} \right)T_n \stackrel{N \rightarrow \infty}{\longrightarrow} T_n.
\ees
On the other hand,
\bes
\Phi_n\left(\sum_{m=0}^N  \left(1 - \frac{m}{N} \right) T_m\right) \stackrel{N \rightarrow \infty}{\longrightarrow} \Phi_n(T),
\ees
whence $T_n = \Phi_n(T)$.
\end{proof}

\subsection{Vacuum state preserving isometric isomorphisms of $\cA_X$.}

\begin{lemma}\label{lem:iso_vacuum}
Let $\varphi : \cA_X \rightarrow \cA_Y$ be an isometric isomorphism. Then $\varphi$ is unital.
\end{lemma}
\begin{proof}
A theorem of Arazy and Solel \cite{ArazySolel} implies that an isometric map between $\cA_X$ and $\cA_Y$ must send $I \in \cA_X$ to an isometry in $\cA_X \cap \cA_X^*$. It follows that $\varphi(I) = cI$, $|c| = 1$. But since $\varphi$ is a homomorphism, then $c = 1$.
\end{proof}

\begin{lemma}\label{lem:vacuum_norm}
For all $n \in \mb{N}$, $\xi \in X(n)$
\bes
\|S^X(\xi)\| = \|S^X(\xi) \Omega_X\| = \|\xi\|.
\ees
\end{lemma}
\begin{proof}
Because $S^X(\xi)$ maps the orthogonal summands $X(k)$ of $\mathfrak{F}_X$ into the orthogonal summands $X(k+n)$, it suffices to show that for all $\eta \in X(k)$, $\|S^X(\xi)\eta\| \leq \|\xi\|\|\eta\|$
(because $S^X(\xi) \Omega_X = \xi$). Now, $S^X(\xi) \eta = p^X_{n+k}(\xi \otimes \eta)$, thus
\bes
\|S^X(\xi) \eta\|^2 \leq \|\xi \otimes \eta\|^2 = \|\xi\|^2 \| \eta\|^2.
\ees
\end{proof}

\begin{lemma}\label{lem:iso_grading}
Let $\varphi : \cA_X \rightarrow \cA_Y$ be an isometric isomorphism that preserves the direct sum decomposition (\ref{eq:direct_sum}). Then $\varphi$ preserves the grading: if $ \xi \in X(n)$ then $\varphi(S^X(\xi))$ is in the norm closure of $\textrm{span}\{S^Y(\eta) : \eta \in Y(n)\}$.
\end{lemma}
\begin{proof}
Since $\varphi$ is a homomorphism, it suffices to show, say, that $\varphi(S^X_1)$ has ``degree one", that is, it is in
the norm closure of $\textrm{span}\{S^Y(\eta) : \eta \in Y(1)\}$. By assumption, we may write $\varphi(S^X_1) = \sum_i a_i S^Y_i + T$, with $T$ in the closure of $\textrm{span}\{S^Y(\eta) : \eta \in Y(n), n \geq 2\}$. But $\varphi^{-1}(\sum_i a_i S^Y_i + T) = S^X_1$, and $\varphi^{-1}(T)$ is in the norm closure of $\textrm{span}\{S^X(\xi) : \eta \in X(n), n \geq 2\}$, so $\varphi^{-1}(\sum_i a_i S^Y_i) = S^X_1 + B$, with $B = -\varphi^{-1}(T)$ (note that $\varphi^{-1}$ also preserves the direct sum decomposition (\ref{eq:direct_sum})).

If $T = 0$ then we are done, so assume $T \neq 0$. Then $B \neq 0$, also. But
\bes
1 = \|S^X_1\| = \|S^X_1 \Omega_X\| < \|(S^X_1 + B)\Omega_X\| \leq \|S^X_1 + B\| = \|\sum_i a_i S^Y_i\|,
\ees
and at the same time
\bes
\|\sum_i a_i S^Y_i\| = \|\sum_i a_i S^Y_i \Omega_Y\| < \|(\sum_i a_i S^Y_i + T) \Omega_Y\| \leq \|\sum_i a_i S^Y_i + T\| = \|S^X_1\| = 1.
\ees
From $T \neq 0$ we arrived at $1<1$, thus $T=0$.
\end{proof}

\begin{theorem}\label{thm:algebra_iso}
$X \cong Y$ if and only if $\cA_X$ and $\cA_Y$ are isometrically isomorphic with an isomorphism that preserves the direct sum decomposition (\ref{eq:direct_sum}), and this happens if and only if $\cA_X$ and $\cA_Y$ are isometrically isomorphic with a grading preserving isomorphism. In fact, if $\varphi : \cA_X \rightarrow \cA_Y$ is a grading preserving isometric isomorphism then there is an isomorphism $V: X \rightarrow Y$ such that for all $T\in \cA_X$, $\varphi(T) = V T V^*$.
\end{theorem}

\begin{proof}
$X \cong Y$ implies $\cA_X \cong \cA_Y$ because these algebras are then generated by unitarily equivalent tuples.

For the converse, we will assume that X and Y are standard subproduct systems. The isomorphism $V : X \rightarrow Y$ is defined on the fiber $X(n)$ by
\bes
V (\xi) = V(S^X(\xi)\Omega_X) = \varphi(S^X(\xi))\Omega_Y \,\, , \,\, \xi \in X(n).
\ees
If it is well defined, then it is onto. Lemma \ref{lem:vacuum_norm} shows that $V$ is an isometry on the fibers:
\bes
\|S^X(\xi)\Omega_X\| = \|S^X(\xi)\| = \|\varphi(S^X(\xi))\| = \|\varphi(S^X(\xi))\Omega_Y\|.
\ees
Lemma \ref{lem:iso_grading} implies that $V(\xi)$ sits in $Y(n)$. $V$ respects the subproduct structure: if $m,n \in \mb{N}$, $\xi \in X(n), \eta \in X(m)$, then
\begin{align*}
V p^X_{m,n}(\xi \otimes \eta) &= V S^X(p^X_{m,n}(\xi \otimes \eta)) \Omega_X \\
&= \varphi(S^X(p^X_{m,n}(\xi \otimes \eta))) \Omega_Y \\
&= \varphi(S^X(\xi) S^X(\eta))\Omega_Y \\
&= \varphi(S^X(\xi)) \varphi(S^X(\eta))\Omega_Y \\
(*)&= p_{m,n}^Y \left( \varphi(S^X(\xi)) \Omega_Y \otimes \varphi(S^X(\eta))\Omega_Y \right) \\
&= p_{m,n}^Y(V(\xi) \otimes V(\eta)).
\end{align*}
(*) follows from the facts $S^Y(y) \Omega_Y = y$ and $S^Y(y_1) S^Y(y_2) \Omega_Y = S^Y(p_{m,n}^Y (y_1 \otimes y_2)) \Omega_Y = p_{m,n}^Y(y_1 \otimes y_2) = p_{m,n}^Y(S^Y(y_1) \Omega_Y \otimes S^Y(y_2) \Omega_Y)$.

Finally, let us show that for all $T\in \cA_X$, $\varphi(T) = V T V^*$.
What we mean by this is that for all $\xi \in X$, $\varphi(S^X(\xi)) = V S^X(\xi) V^*$. Let $\varphi(S^X(\eta))\Omega_Y = V(\eta)$ be a typical element in $\mathfrak{F}_Y$.
\begin{align*}
V S^X(\xi) V^* \varphi(S^X(\eta))\Omega_Y &= V S^X(\xi) \eta \\
&= V p^X(\xi \otimes \eta) \\
&= \varphi(S^X(p^X(\xi \otimes \eta))) \Omega_Y \\
&= \varphi(S^X(\xi) S^X(\eta)) \Omega_Y \\
&= \varphi(S^X(\xi)) \varphi(S^X(\eta)) \Omega_Y ,
\end{align*}
This completes the proof.
\end{proof}

\section{Classification of the universal algebras of $q$-commuting tuples}\label{sec:qcommuting}

\begin{definition}
A matrix $q$ is called \emph{admissible} if $q_{ii}=0$ and $0 \neq q_{ij}=q_{ji}^{-1}$ for all $i \neq j$.
\end{definition}

\subsection{The $q$-commuting algebras $\cA_q$ and their universality}
Let $\{e_1, \ldots,  e_d\}$ be an orthonormal basis for $E := \mb{C}^d$, to be fixed (together with $d$) throughout this section. Let $q \in M_d(\mb{C})$ be an admissible matrix, and let $X_q$ be the maximal standard subproduct system with fibers
\bes
X_q(1) = E \,\, , \,\, X_q(2) = E \otimes E \ominus \textrm{span}\{e_i \otimes e_j - q_{ij} e_j \otimes e_i : 1\leq i, j \leq d, i \neq j \}.
\ees
When $q_{ij} = 1$ for all $i<j$, then $X_{q}$ is the symmetric subproduct system $SSP$. The Fock spaces $\mathfrak{F}_{X_q}$ have been studied in \cite{Dey07}.

For brevity, we shall write $S^q_i$ instead of $S^{X_q}_i$.
We denote by $\cA_q$ the algebra $\cA_{X_q}$. By Theorem \ref{thm:universal}, the algebra $\cA_q$ is the universal
norm closed unital operator algebra generated by a row contraction $(T_1, \ldots, T_d)$ satisfying the relations
\bes
T_i T_j = q_{ij}T_j T_i \,\, , \,\, 1 \leq i < j \leq d.
\ees


\subsection{The character space of $\cA_q$}\label{subsec:charA_q}

Let $\cM_q$ be the space of all (contractive) multiplicative and unital linear functionals on $\cA_q$, endowed with the weak-$*$ topology. We shall call $\cM_q$ \emph{the character space of $\cA_q$}. Every $\rho \in \cM_q$ is uniquely determined by the $d$-tuple of complex numbers $(x_1, \ldots, x_d)$, where $x_i = \rho(S^q_i)$ for $i = 1, \ldots, d$. Since a contractive linear functional is completely contractive, $(x_1, \ldots, x_d)$ must be a row contraction, that is, $|x_1|^2 + \ldots + |x_d|^2 \leq 1$. In other words, $(x_1, \ldots, x_d)$ is in the unit ball $B_d$ of $\mb{C}^d$. The multiplicativity of $\rho$ implies that $(x_1, \ldots, x_d)$ must lie inside the set
\bes
Z_q := \{(z_1, \ldots, z_d)\in B_d: (1-q_{ij})z_iz_j = 0, 1 \leq i < j \leq d\} .
\ees

Conversely, Theorem \ref{thm:universal} implies that every $(x_1, \ldots, x_d) \in Z_q$ gives rise to a character $\rho \in \cM_q$ that sends $S^q_i$ to $x_i$. Thus the map
\bes
\cM_q \ni \rho \mapsto (\rho(S^q_1), \ldots, \rho(S^q_d)) \in Z_q
\ees
is injective and surjective. It is also obviously continuous (with respect to the weak-$*$ and standard topologies). Since $\cM_q$ is compact, we have the homeomorphism
\be
\cM_q \cong Z_q.
\ee

Note that the vacuum state $\rho_0$ corresponds to the point $0 \in Z_q \subset \mb{C}^d$.

When $q_{ij} = 1$, the condition $(1-q_{ij})z_iz_j = 0$ is trivially satisfied, so when $q_{i,j}=1$ for all $i,j$, then $Z_q$ is the unit ball $B_d$. When $q_{ij} \neq 1$, the condition is that either $z_i = 0$ or $z_j = 0$. Thus, if for all $i,j$, $q_{ij}\neq 1$, then $Z_q$ is the union of $d$ discs glued together at their origins.

\subsection{Classification of the $\cA_q$, $q_{ij}\neq 1$}

Given a permutation $\sigma$ (on a set with $d$ elements), let $U_\sigma$ be the matrix that induces the same permutation on the standard basis of $\mb{C}^d$.

\begin{proposition}\label{prop:similarity}
Let $q$ and $r$ be two admissible $d \times d$ matrices. Assume that there is a permutation $\sigma \in S_d$ such that $r = U_\sigma q U_\sigma^{-1}$, and let $\lambda_1, \ldots, \lambda_d$ be any complex numbers on the unit circle. Then the map
\be\label{eq:sigma}
e_i \mapsto \lambda_i e_{\sigma(i)}
\ee
extends to an isomorphism of $X_q$ onto $X_r$, and thus the map
\bes
S_i^q \mapsto \lambda_i S_{\sigma(i)}^r
\ees
extends to a completely isometric isomorphism between $\cA_q$ and $\cA_r$.
\end{proposition}
\begin{proof}
For all $n$, the map (\ref{eq:sigma}) extends to a unitary $V_n$ of $E^{\otimes n}$. For $n=2$, this unitary sends $e_i \otimes e_j - q_{ij}e_j \otimes e_i$ to $\lambda_i \lambda_j e_{\sigma(i)} \otimes e_{\sigma(j)} - \lambda_i \lambda_j q_{ij}e_{\sigma(j)} \otimes e_{\sigma(i)}$. But $r = U_\sigma q U_\sigma^{-1}$ implies $r_{\sigma(i)\sigma(j)} = q_{ij}$, thus
\bes
V_2: e_i \otimes e_j - q_{ij}e_j \otimes e_i \mapsto \lambda_i \lambda_j e_{\sigma(i)} \otimes e_{\sigma(j)} - \lambda_i \lambda_j r_{\sigma(i)\sigma(j)}e_{\sigma(j)} \otimes e_{\sigma(i)},
\ees
so $V_2$ is a unitary between $X_q(2)$ and $X_r(2)$ that respects the product. By induction, it follows that $V = \{V_n\big|_{X_q(n)}\}_n$ is an isomorphism of subproduct systems. The final assertion follows from Proposition \ref{prop:UE_ISO}.
\end{proof}

\begin{theorem}\label{thm:subproduct_iso_q}
Let $q$ and $r$ be two admissible $d \times d$ matrices such that $q_{ij},r_{ij}\neq 1$ for all $i,j$. Then $X_q$ is isomorphic to $X_r$ if and only if there is a permutation $\sigma \in S_d$ such that $r = U_\sigma q U_\sigma^{-1}$. In this case the isomorphisms are precisely those of the form
\bes
e_i \mapsto \lambda_i e_{\sigma(i)},
\ees
where $\lambda_1, \ldots, \lambda_d$ are any complex numbers on the unit circle, and $\sigma$ is such that $r = U_\sigma q U_\sigma^{-1}$.
\end{theorem}
\begin{proof}
One direction is Proposition \ref{prop:similarity}, so assume that there is an isomorphism of subproduct systems $V: X_q \rightarrow X_r$. Let $f_i := V^{-1}e_i$. There is a $d \times d$ unitary matrix $U = (u_{ij})$ such that $f_i = \sum_j u_{ij}e_j$. As $V$ is an isomorphism of subproduct systems, we have for all $i \neq j$
\bes
Vp_2^{X_q}(f_i \otimes f_j - r_{ij}f_j \otimes f_i) = p_2^{X_r}(e_i \otimes e_j - r_{ij}e_j \otimes e_i) = 0,
\ees
thus
\bes
(\sum_k u_{ik} e_k)\otimes(\sum_l u_{jl} e_l) - r_{ij} (\sum_k u_{jk} e_k)\otimes(\sum_l u_{il} e_l) \in \textrm{span} \{e_m \otimes e_n - q_{mn} e_n \otimes e_m : m \neq n\},
\ees
or
\be\label{eq:sum_in}
\sum_{k,l}(u_{ik}u_{jl} - r_{ij}u_{jk}u_{il})e_k \otimes e_l \in \textrm{span} \{e_m \otimes e_n - q_{mn} e_n \otimes e_m : m \neq n\}.
\ee
The coefficients of the vectors $e_k \otimes e_k$ in the sum above must vanish, thus $u_{ik}u_{jk} - r_{ij}u_{jk}u_{ik} = 0$ for all $i \neq j$. Since $r_{ij}\neq 1$, we must have $u_{jk}u_{ik} = 0$ for all $k$ and all $i \neq j$. Thus the unitary matrix $U$ has precisely one nonzero element in each column, and it therefore must be of the form $U_\sigma^{-1} D$, where $D$ is a diagonal unitary matrix.

Equation (\ref{eq:sum_in}) becomes
\bes
u_{i\sigma(i)}u_{j\sigma(j)}e_{\sigma(i)} \otimes e_{\sigma(j)} - r_{ij}u_{j\sigma(j)}u_{i\sigma(i)}e_{\sigma(j)} \otimes e_{\sigma(i)} \in \textrm{span} \{e_m \otimes e_n - q_{mn} e_n \otimes e_m : m \neq n\},
\ees
but this can only happen if

\bes
u_{i\sigma(i)}u_{j\sigma(j)}e_{\sigma(i)} \otimes e_{\sigma(j)} - r_{ij}u_{j\sigma(j)}u_{i\sigma(i)}e_{\sigma(j)} \otimes e_{\sigma(i)}
\ees
is proportional to
\bes
e_{\sigma(i)} \otimes e_{\sigma(j)} - q_{\sigma(i)\sigma(j)}e_{\sigma(j)} \otimes e_{\sigma(i)} ,
\ees
that is $u_{i\sigma(i)}u_{j\sigma(j)}q_{\sigma(i)\sigma(j)} = u_{j\sigma(j)}u_{i\sigma(i)}r_{ij}$, or $r_{ij} = q_{\sigma(i)\sigma(j)}$. Replacing $\sigma$ with $\sigma^{-1}$, the proof is complete.
\end{proof}

\begin{corollary}\label{cor:auto}
Let $q$ be an admissible $d \times d$ matrix such that there is no permutation $\sigma \in S_d$ such that $q = U_\sigma q U_\sigma^{-1}$. Assume that $q_{ij} \neq 1$ for all $i,j$. Then the only automorphisms of $X_q$ are unitary scalings of the basis $\{e_1, \ldots, e_d\}$.
\end{corollary}

\begin{theorem}\label{thm:algebra_iso_q}
Let $q$ and $r$ be two admissible $d \times d$ matrices such that $q_{ij},r_{ij}\neq 1$ for all $i,j$. Then $\cA_q$ is isometrically isomorphic to $\cA_r$ if and only if there is a permutation $\sigma \in S_d$ such that $r = U_\sigma q U_\sigma^{-1}$. In this case the isometric isomorphisms between $\cA_q$ and $\cA_r$ are precisely those of the form
\bes
S^q_i \mapsto \lambda_i S^r_{\sigma(i)},
\ees
where $\lambda_1, \ldots, \lambda_d$ are any complex numbers on the unit circle.
\end{theorem}
\begin{proof}
If $r = U_\sigma q U_\sigma^{-1}$, then by Proposition \ref{prop:similarity} and Theorem \ref{thm:algebra_iso} $\cA_q$ and $\cA_r$ are isomorphic (with an isomorphism that preserves the direct sum decomposition (\ref{eq:direct_sum})).

Conversely, assume that $\varphi: \cA_q \rightarrow \cA_r$ is a completely isometric isomorphism. Then $\varphi$ induces a homeomorphism between $\cM_r$ and $\cM_q$ by $\rho \mapsto \rho \circ \varphi$.
Recall that $\cM_q$ and $M_r$ are both homeomorphic to $d$ discs glued together at the origin. Thus the homeomorphism $\rho \mapsto \rho \circ \varphi$ must take $\rho_0$ of ${X_r}$ to $\rho_0$ of ${X_q}$, because these are the unique points in $\cM_r$ and $\cM_q$, respectively, that when removed from $M_r$ and $\cM_q$ leave $d$ disconnected punctured discs. Thus $\varphi$ sends the vacuum state of $\cA_r$ to the vacuum state of $\cA_q$, and must therefore preserve the direct sum decomposition (\ref{eq:direct_sum}). By Theorem \ref{thm:algebra_iso}, there is an isomorphism of subproduct systems $V: X_q \rightarrow X_r$ such that $\varphi(\bullet) = V \bullet V^*$. By Theorem \ref{thm:subproduct_iso_q} we conclude that there is a permutation $\sigma \in S_d$ such that $r = U_\sigma q U_\sigma^{-1}$. It also follows that
$\varphi(S_i^q) = \lambda_i S_{\sigma(i)}^r$.
\end{proof}

\begin{corollary}\label{cor:iso_auto}
Let $q$ be an admissible $d \times d$ matrix such that there is no permutation $\sigma \in S_d$ such that $q = U_\sigma q U_\sigma^{-1}$. Then the only isometric automorphisms of $\cA_q$ are unitary scalings of the shift $\{S^q_1, \ldots, S^q_d\}$.
\end{corollary}

As a corollary of the above discussion we have:
\begin{corollary}
Let $q$ and $r$ be two admissible $d \times d$ matrices such that $q_{ij},r_{ij}\neq 1$ for all $i,j$. Then $\cA_q$ is isometrically isomorphic to $\cA_r$ if and only if $X_q \cong X_r$.
\end{corollary}

\subsection{$X_q$ and $\cA_q$, $d = 2$}
In the particular case $d=2$, we let a complex number $q$ parameterize the spaces $X_q$ (we may allow also $q = 0$) defined to be the maximal standard subproduct system with fibers
\bes
X_q(1) = \mb{C}^2 \,\, , \,\, X_q(2) = \mb{C}^2 \otimes \mb{C}^2 \ominus \textrm{span}\{e_1 \otimes e_2 - q e_2 \otimes e_1 \}.
\ees
Since $\cM_1 \cong B_2$, $\cA_1$ is not isomorphic to any $\cA_q$ with $q \neq 1$ (recall that when $q \neq 1$, $\cM_q$ is homeomorphic to two discs glued together at the origin). Thus Theorem \ref{thm:algebra_iso_q} gives:
\begin{corollary}
Assume that $d=2$. Then $X_q \cong X_r$ if and only if $\cA_q$ is isometrically isomorphic to $\cA_r$, and either one of these happens if and only if either $r = q$ or $r = q^{-1}$.
\end{corollary}

Elias Katsoulis has pointed out to us that the above corollary also follows from the techniques of \cite{DavidsonKatsoulis}.

The above result is reminiscent to the fact that two rotation algebras $A_\theta$ and $A_{\theta'}$ are isomorphic if and only if either $e^{2\pi i \theta} = e^{2\pi i \theta'}$ or $(e^{2\pi i \theta})^{-1} = e^{2\pi i \theta'}$. One cannot help but wonder whether one can draw a deeper connection between these results then the superficial one, in particular, can the classification of  rotation algebras be deduced from the classification of the algebras $\cA_q$?

By Corollaries \ref{cor:auto} and \ref{cor:iso_auto} we have the following.
\begin{corollary}
Let $d=2$ and let $q \neq 1$. Then subproduct system $X_q$ has no automorphisms aside form the unitary scalings of the basis. The algebra $\cA_q$ has no isometric automorphisms other than unitary scalings of the generators.
\end{corollary}

On the other hand, a direct calculation shows that every unitary on $\mb{C}^2$ extends to an automorphism of $X_1$, and thus induces a non-obvious automorphism of $\cA_1$.

\section{Standard maximal subproduct systems with $\dim X(1) = 2$ and $\dim X(2) = 3$}\label{sec:A}

Again, let $\{e_1, \ldots,  e_d\}$ be an orthonormal basis for $E := \mb{C}^d$. We will soon turn attention to the case $d=2$. For a matrix $A \in M_d(\mb{C})$, we define the \emph{symmetric part} of $A$ to be $A^s := (A + A^t)/2$ and the \emph{antisymmetric part} of $A$ to be $A^a := (A - A^t)/2$. Denote by $X_A$ the maximal standard subproduct system with fibers
\bes
X_A(1) = E \,\, , \,\, X_A(2) = E \otimes E \ominus \textrm{span}\left\{\sum_{i,j=1}^d a_{ij}e_i \otimes e_j \right\}.
\ees
We will write $S^A$ for the shift $S^{X_A}$. We will also write $\cA_A$ for $\cA_{X_A}$.

\begin{proposition}\label{prop:classifyX_A}
Let $A, B \in M_d(\mb{C})$. Then there is an isomorphism $V: X_A \rightarrow X_B$ if and only if there exists $\lambda \in \mb{C}$ and a unitary $d \times d$ matrix $U$ such that $B = \lambda U^t A U$. In this case, $U$ extends to the isomorphism $V$ between $X_A$ and $X_B$ by $V_1 = U$.
\end{proposition}
\begin{proof}
Let $V: X_A \rightarrow X_B$ be an isomorphism of subproduct systems. There is a $d\times d$ unitary matrix $U = (u_{ij})$ such that
\bes
f_i := V_1 (e_i) = \sum_{j=1}^d u_{ij}e_j   .
\ees
Then
\begin{align*}
0 &= V_1 (p_2^X(\sum_{i,j}a_{ij}e_i \otimes e_j)) \\
&= p_2^Y (\sum_{i,j}a_{ij} f_i \otimes f_j),
\end{align*}
so $\sum_{i,j}a_{ij} f_i \otimes f_j$ must be a spanning vector of $\textrm{span}\left\{\sum_{i,j} b_{ij}e_i \otimes e_j \right\}$. Writing out fully what this means,
\bes
\lambda \sum_{i,j}a_{ij} \sum_{k,l} u_{ik}u_{jl}  e_k \otimes e_l =  \sum_{k,l} b_{kl}e_k \otimes e_l
\ees
for some $\lambda \in \mb{C}$, so
\bes
b_{kl} = \lambda \sum_{i,j}a_{ij}u_{ik}u_{jl} .
\ees
But the right hand side is precisely the $kl$-th element of $\lambda U^t A U$.

Conversely, assuming $B = \lambda U^t A U$, one can read the above argument from finish to start to obtain an isomorphism $V: X_A \rightarrow X_B$.
\end{proof}

We see that for $X_A$ and $X_B$ to be isomorphic the ranks of $A$ and $B$ must be the same, as well as the ranks of their symmetric and anti-symmetric parts. For example, if $A$ is symmetric and $B$ is not then $X_A \ncong X_B$, a result which may not seem obvious at first glance.

\begin{theorem}\label{thm:classifyA_A}
Assume that $d=2$. Let $A,B \in M_2(\mb{C})$ be any two matrices. Then $\cA_{A}$ is isometrically isomorphic to $\cA_B$ if and only if $X_A \cong X_B$, and this happens if and only if there exists $\lambda \in \mb{C}$ and a unitary $2 \times 2$ matrix $U$ such that $B = \lambda U^t A U$.
\end{theorem}

The proof of Theorem \ref{thm:classifyA_A} will occupy the rest of this section. Denote by $\cM_A$ the character space of $\cA_A$, that is, the topological space of contractive multiplicative and unital linear functionals on $\cA_A$, endowed with the weak-$*$ topology.

\begin{lemma}
The topology of $\cM_A$ depends on the rank $r(A^s)$ of the symmetric part $A^s$ of $A$:
\begin{enumerate}
	\item If $r(A^s) = 0$ then $\cM_A \cong B_2$, the unit ball in $\mb{C}^2$.
	\item If $r(A^s) = 1$ then $\cM_A \cong D$, the unit disc in $\mb{C}$.
	\item If $r(A^s) = 2$ then $\cM_A$ is homeomorphic to two discs pasted together at the origin.
\end{enumerate}
\end{lemma}

\begin{proof}
We proceed similarly to the lines of \ref{subsec:charA_q}.
Every character $\rho \in \cM_A$ is uniquely determined by $\lambda_1 = \rho(S^A_1)$ and $\lambda_2 = \rho(S^A_2)$, which lie in $B_2$. Conversely, every $(\lambda_1, \lambda_2) \in B_2$ that satisfies
\bes
\sum_{i,j}a_{ij}\lambda_i \lambda_j = 0
\ees
gives rise to a character $\rho$ by defining $\lambda_1 = \rho(S^A_1)$ and $\lambda_2 = \rho(S^A_2)$. Thus,
\bes
\cM_A \cong V_A := \left\{(\lambda_i, \lambda_j) \in B_2 : \sum_{i,j}a_{ij} \lambda_i \lambda_j = 0 \right\}.
\ees
Clearly, $V_A = V_{A^s}$. However, every symmetric $2 \times 2$ matrix is complex congruent to one of the following:
\bes
D_0 = \begin{pmatrix}
  0 & 0 \\
  0 & 0 \\
\end{pmatrix},
D_1 = \begin{pmatrix}
  1 & 0 \\
  0 & 0 \\
\end{pmatrix}
\,\, \textrm{or} \,\,
D_2 = \begin{pmatrix}
  1 & 0 \\
  0 & 1 \\
\end{pmatrix},
\ees
i.e., there exists a nonsingular matrix $T$ such that $A^s = T^t D_i T$, for $i = r(A^s)$. But then $V_{A^s} = T^{-1}V_{D_i} \cong V_{D_i}$, so it remains to verify that $V_{D_i}$ is homeomorphic to the spaces listed in the statement of the lemma.
\end{proof}

\begin{corollary}
If $r(A^s) \neq r(B^s)$ then $\cA_A \ncong \cA_B$.
\end{corollary}

We can use this corollary to break down the classification of the algebras $\cA_A$ to the classification of the algebras $\cA_A$ with fixed $r(A^s)$. The easiest case is $r(A^s) = 0$, because then $A$ is either the zero matrix or a multiple of $\begin{pmatrix}
  0 & 1 \\
  -1 & 0 \\
\end{pmatrix}$, and these two matrices give rise to non isomorphic algebras (these are the algebras generated by the full and symmetric shifts, respectively).

The next easiest case is $r(A^s) = 2$.
\begin{lemma}
If $A,B \in M_2(\mb{C})$ and $r(A^s) = r(B^s) = 2$, then $\cA_{A}$ is isometrically isomorphic to $\cA_B$ if and only if $X_A \cong X_B$, and this happens if and only if there exists $\lambda \in \mb{C}$ and a unitary $2 \times 2$ matrix $U$ such that $B = \lambda U^t A U$. Any isometric isomorphism between $\cA_A$ and $\cA_B$ arises as conjugation by the subproduct system isomorphism arising from $U$.
\end{lemma}
\begin{proof}
In light of Theorem \ref{thm:algebra_iso} and Proposition \ref{prop:classifyX_A}, it suffices to show that any isometric isomorphism $\varphi: \cA_A \rightarrow \cA_B$ sends the vacuum state to the vacuum state. But the vacuum state in $\cM_A$ and in $\cM_B$ corresponds to the point where the two discs are glued together. Since $\varphi$ induces a homeomorphism between $\cM_B$ and $\cM_A$, it must send the vacuum state to the vacuum state.
\end{proof}

\begin{remark}\emph{
In the previous section we have seen already that there is a continuum of non-(completely isometrically)-isomorphic algebras $\cA_{A}$ and subproduct systems $X_{A}$ with $r(A^s) = 2$, namely the algebras $\cA_q$. One can see that these algebras $\cA_A$ are not exhausted by the algebras $\cA_q$ of the previous section. For example, all the algebras $\cA_A$ with $A = \begin{pmatrix}
  1 & 0 \\
  0 & q \\
\end{pmatrix}$, with $q>0$, are non-isomorphic, and only for $q = 1$ is this algebra isomorphic to an $\cA_q$ (in this case $q = -1$).}
\end{remark}

We now come to the trickiest case, $r(A^s) = 1$.


\begin{lemma}\label{lem:symrank1}
If $A,B \in M_2(\mb{C})$ are two symmetric matrices of rank $1$, then there exists $\lambda \in \mb{C}$ and a unitary $2 \times 2$ matrix $U$ such that $B = \lambda U^t A U$, and consequently $X_A \cong X_B$ and $\cA_{A}$ is isometrically isomorphic to $\cA_B$.
\end{lemma}
\begin{proof}
We only have to prove the first assertion, and we may assume that $B = \begin{pmatrix}
  1 & 0 \\
  0 & 0 \\
\end{pmatrix}$. We may also assume that there is a unit vector $v = (v_1, v_2)^t$ such that $A = v v^t$. Now let
\bes
U = \begin{pmatrix}
  \overline{v_1} & \overline{v_2} \\
  \overline{v_2} & -\overline{v_1} \\
\end{pmatrix}.
\ees
Then
\bes
U^t A U = \begin{pmatrix}
  \overline{v_1} & \overline{v_2} \\
  \overline{v_2} & -\overline{v_1} \\
\end{pmatrix}^t v v^t \begin{pmatrix}
  \overline{v_1} & \overline{v_2} \\
  \overline{v_2} & -\overline{v_1} \\
\end{pmatrix}
= \begin{pmatrix}
  \overline{v_1} & \overline{v_2} \\
  \overline{v_2} & -\overline{v_1} \\
\end{pmatrix}^t \begin{pmatrix}
  v_1 & 0 \\
  v_2 & 0 \\
\end{pmatrix} = \begin{pmatrix}
  1 & 0 \\
  0 & 0 \\
\end{pmatrix}.
\ees
\end{proof}


Below we will also need the following lemma.
\begin{lemma}\label{lem:notsymrank1}
Let $A$ be a $2\times2$ matrix for which $r(A^s) = 1$. Then there exists one and only one $q\geq 0$ for which there is a $\lambda \in \mb{C}$ and a unitary $U$ such that
\bes
\begin{pmatrix}
1 & q \\ -q & 0
\end{pmatrix} = \lambda U^t A U.
\ees
Furthermore, if $A$ is non-symmetric then $A$ is congruent to the matrix
\bes
\begin{pmatrix}
1 & 1 \\ -1 & 0
\end{pmatrix}.
\ees
\end{lemma}
\begin{proof}
Direct verification, using Lemma \ref{lem:symrank1} and the fact that congruations preserves, up to a scalar, the anti-symmetric part.
\end{proof}

Let us write $A_q$ for the matrix
\bes
A_q = \begin{pmatrix}
1 & q \\ -q & 0
\end{pmatrix}.
\ees
By the above lemma, we may restrict attention only to the algebras $\cA_{A_q}$ with $q \geq 0$.

Recall that the character space $\cM_{A_q}$ of $\cA_{A_q}$ is identified with the closed unit disc $\overline{\mb{D}}$
by
\bes
\cM_{A_q} \ni \rho \longleftrightarrow \rho(S^{A_q}_2) \in \overline{\mb{D}} .
\ees
We write $\rho_z$ for the character that sends $S^{A_q}_2$ to $z \in \overline{\mb{D}}$.
This identifies the vacuum vector $\rho_0$ with the point $0$.
Recall also that if $\varphi : \cA_{A_q} \rightarrow \cA_{A_r}$ is an isometric isomorphism, then it induces a
homeomorphism $\varphi_* : \cM_{A_r} \rightarrow \cM_{A_q}$ given by $\varphi_* \rho = \rho \circ \varphi$.
We write $F_\varphi$ for the homeomorphism $\overline{\mb{D}} \rightarrow \overline{\mb{D}}$ induced by $\varphi$,
that is, $F_\varphi$ is the unique self map of $\overline{\mb{D}}$ that satisfies
\bes
\varphi_* \rho_z = \rho_{F_\varphi(z)} \,\, , \,\, z \in \overline{\mb{D}}.
\ees

Let us introduce the notation
\bes
\cO(0;q,r) = \{F_\varphi(0) \big| \varphi : \cA_{A_q} \rightarrow \cA_{A_r} \textrm{ is an isometric isomorphism}\},
\ees
and
\bes
\cO(0;q) = \cO(0;q,q).
\ees
\begin{lemma}\label{lem:nozero}
Let $q,r \geq 0$. If $q \neq r$ then $0$ does not lie in $\cO(0;q,r)$.
\end{lemma}
\begin{proof}
Assume that $0 \in \cO(0;q,r)$. Then there is some isometric isomorphism $\varphi:\cA_{A_q} \rightarrow \cA_{A_q}$ that preserves the character $\rho_0$. It follows from Theorem \ref{thm:algebra_iso} and Proposition \ref{prop:classifyX_A} that, for some unitary $2 \times 2$ matrix $U$ and some $\lambda \in \mb{C}$, $A_q = \lambda U^t A_r U$. But, as noted in Lemma \ref{lem:notsymrank1}, this is impossible if $r \neq q$.
\end{proof}

\begin{lemma}\label{lem:rotinv}
The sets $\cO(0;q,r)$ are invariant under rotations around $0$.
\end{lemma}
\begin{proof}
For $\lambda$ with $|\lambda| = 1$, write $\varphi_\lambda$ for the isometric isomorphism mapping $S_i^{A_q}$ to
$\lambda S_i^{A_q}$ ($i = 1,2$). For $b = F_\varphi(0) \in \cO(0;q,r)$, consider $\varphi \circ \varphi_\lambda$.
We have $\rho_0((\varphi \circ \varphi_\lambda)(S_2^{A_q})) = \rho_0(\varphi(\lambda S_2^{A_q})) = \lambda \rho_0(\varphi(S_2^{A_q})) = \lambda b$.
Thus $\lambda b \in \cO(0;q,r)$.
\end{proof}

\begin{lemma}
Let $q,r \geq 0$. If $q \neq r$ then $\cA_{A_q}$ is not isometrically isomorphic to $\cA_{A_r}$.
\end{lemma}

\begin{proof}
Assume that $\varphi :\cA_{A_q} \rightarrow \cA_{A_r}$ is an isometric isomorphism. We have
$\rho_0 \circ \varphi = \rho_b$, with $b = F_\varphi(0)$, and $F_\varphi$ is a homeomorphism of $\overline{\mb{D}}$
onto itself.

By definition, $b \in \cO(0;q,r)$. By Lemma \ref{lem:nozero}, $b \neq 0$. Denote $C := \{z:|z|=|b|\}$. By Lemma
\ref{lem:rotinv}, $C \subseteq \cO(0;q,r)$. Consider $C' := F_\varphi^{-1}(C)$. We have that $C' \subseteq \cO(0;r)$.
$C'$ is a simply connected closed path in $\mb{D}$ that goes through the origin. By Lemma \ref{lem:rotinv}, the
interior of $C'$, $\textrm{int}(C')$, is in $\cO(0;r)$. But then $F_\varphi(\textrm{int}(C'))$
is the interior of $C$, and it is in $\cO(0;q,r)$. But then $0 \in \cO(0;q,r)$, contradicting Lemma \ref{lem:nozero}.
\end{proof}

That concludes the proof of Theorem \ref{thm:classifyA_A}.

\section{The representation theory of Matsumoto's subshift C$^*$-algebras}\label{sec:subshift}

In \cite{Ma} Kengo Matsumoto introduced a class of C$^*$-algebras that arise from symbolic dynamical systems called ``subshifts" (we note that in the later paper \cite{CM04} Carlsen and Matsumoto suggest another way of associating a C$^*$-algebra with a subshift. Here we are discussing only the algebras originally introduced in \cite{Ma}). These \emph{subshift algebras}, as we shall call them, are strict generalizations of Cuntz-Krieger algebras and have been extensively studied by Matsumoto, T. M. Carlsen and others. For example, the following have been studied: criteria for simplicity and pure-infiniteness; conditions on the underlying dynamical systems for subshift algebras to be isomorphic; the automorphisms of the subshift algebras; K-theory of the subshift algebras; and much more. In this section we will use the framework constructed in the previous sections to give a complete description of all representations of a subshift algebra when the subshift is of finite type.


\subsection{Subshifts and the corresponding subproduct systems and C$^*$-algebras}

Our references for subshifts are \cite{Ma} and \cite[Chapter 3]{BrinStuck02}.

Let $\cI = \{1, 2, \ldots, d\}$ be a fixed finite set. $\cI^{\mb{Z}}$ is the space of all two-sided infinite sequences, endowed with the product topology. The \emph{left shift} (or simply \emph{the shift}) on $\cI^{\mb{Z}}$ is the homeomorphism $\sigma : \cI^{\mb{Z}} \rightarrow \cI^{\mb{Z}}$ given by $(\sigma(x))_k = x_{k+1}$. Let $\Lambda$ be a shift invariant closed subset of $\cI^{\mb{Z}}$. By this we mean $\sigma(\Lambda) = \Lambda$. The topological dynamical system $(\Lambda, \sigma\big|_\Lambda)$ is called a \emph{subshift}. Sometimes $\Lambda$ is also referred to as the subshift.

If $W$ is a set of words in $1,2,\ldots,d$, one can define a subshift by forbidding the words in $W$ as follows:
\bes
\Lambda_W = \{x \in \cI^{\mb{Z}} : \textrm{no word in $W$ occurs as a block in $x$}\}.
\ees
Conversely, every subshift arises this way: i.e., for every subshift $\Lambda$ there exists a collection of words $W$, called \emph{the set of forbidden words}, such that $\Lambda = \Lambda_W$. In this context, if $W$ can be chosen finite then $\Lambda = \Lambda_W$ is called \emph{a subshift of finite type}, or \emph{SFT} for short. By replacing $\cI$ if needed, we may always assume that $W$ has no words of length one. If $W$ can be chosen such that the longest word in $W$ has length $k+1$ then $\Lambda$ is called a \emph{$k$-step SFT}. A $1$-step SFT is also called a \emph{topological Markov chain}. A basic result is that every SFT is isomorphic to a topological Markov chain (\cite[Proposition 3.2.1]{BrinStuck02}).

For a fixed subshift $(\Lambda, \sigma\big|_\Lambda)$, we set
\bes
\Lambda^k = \{\alpha : \textrm{$\alpha$ is a word with length $k$ occurring in some $x \in \Lambda$}\},
\ees
and $\Lambda_l = \cup_{k=0}^l \Lambda^k$, $\Lambda^* = \cup_{k=0}^\infty \Lambda^k$. With the subshift $(\Lambda, \sigma\big|_\Lambda)$ we associate a subproduct system $X_\Lambda$ as follows. Let $\{e_i\}_{i \in \cI}$ be an orthonormal basis of a Hilbert space $E$. We define
\bes
X_\Lambda (0) = \mb{C},
\ees
and for $n \geq 1$ we define
\bes
X_\Lambda (n) = \textrm{span}\{e_\alpha : \alpha \in \Lambda^n\}.
\ees
We define a product $U_{m,n}: X_\Lambda (m) \otimes X_\Lambda(n) \rightarrow X_\Lambda(m+n)$
by
\bes
U_{m,n}(e_\alpha \otimes e_\beta) =
\begin{cases}
e_{\alpha \beta} , & \textrm{if } \alpha\beta \in \Lambda^{m+n} \cr
0  , & \textrm{else.}
\end{cases}
\ees
Since $\Lambda^{m+n} \subseteq \Lambda^m \cdot \Lambda^n$, $X_\Lambda$ is a standard subproduct system.

\begin{definition}
The C$^*$-algebra associated with a subshift $(\Lambda, \sigma\big|_\Lambda)$ is defined as the quotient algebra
\bes
\cO_\Lambda := \cO_{X_\Lambda} = \cT_{X_\Lambda} / \cK(\mathfrak{F}_{X_\Lambda}).
\ees
\end{definition}
\begin{remark}\emph{
Just to prevent confusion: In \cite{Ma}, $\cO_\Lambda$ was defined as the quotient by the compacts of the C$^*$-algebra generated by the ``creation operators" (that is, the $X$-shift) on $\mathfrak{F}_X$, without using the language of subproduct systems.}
\end{remark}

\subsection{Subproduct systems that come from subshifts}

\begin{proposition}
Let $X$ be a standard subproduct system such that there is an orthonormal basis $\{e_i\}_{i \in \cI}$ of $X(1)$, with $\cI$ finite, such that
 \begin{enumerate}
 \item Every $X(n)$, $n\geq1$, is spanned by vectors of the form $e_\alpha$ with $|\alpha| = n$.
 \item For all $m,n \in \mb{N}$, $|\alpha| = n$ and $e_\alpha \in X(n)$, implies that there is some $\beta,\gamma \in \cI^m$ such that $e_\beta \otimes e_\alpha$ and $e_\alpha \otimes e_{\gamma}$ are in $X(m+n)$.
 \end{enumerate}
 Then there is a shift invariant closed subset $\Lambda$ of $\cI^{\mb{Z}}$ such that $X = X_\Lambda$. $X$ is the maximal standard subproduct system with prescribed fibers $X(1), X(2), \ldots, X(k+1)$ if and only if $\Lambda$ is $k$-step SFT.
\end{proposition}
\begin{proof}
For all $k \in \mb{N}$, define
\bes
\Lambda^{(k)} = \{\alpha \in \cI^k : e_\alpha \in X(k)\}.
\ees
For all $m \in \mb{Z}, k \in \mb{N}$, define the closed sets
\bes
A_{m,k} = \{x \in \cI^{\mb{Z}} : (x_m, x_{m+1}, \ldots, x_{m+k-1}) \in \Lambda^{(k)}\}.
\ees
Condition (2) implies that $X(k)$ always contains a nonzero vector of the form $e_\alpha$, $|\alpha| = k$. That implies that the family $\{A_{m,k} \}_{m,k}$ has the finite intersection property. Indeed,
\bes
A_{m_1,k_1} \cap A_{m_2,k_2} \supseteq A_{M,K} \neq \emptyset,
\ees
where $M = \min\{m_1, m_2 \}$, $K = \max\{m_2 + k_2, m_1 + k_1\} - M$. By compactness of $\cI^{\mb{Z}}$ we conclude that the closed set
\bes
\Lambda := \bigcap_{m,k}A_{m,k}
\ees
is non-empty. $\Lambda$ is invariant under the left and the right shifts, so $\sigma(\Lambda) = \Lambda$, so $(\Lambda, \sigma\big|_\Lambda)$ is a subshift. By condition (2), $\Lambda^k = \Lambda^{(k)}$. Condition (1) together with the definition of $X_\Lambda$  now imply that $X = X_\Lambda$.

The final assertion follows from the following facts, together with $X = X_\Lambda$. Fact number one:
\bes
E^{\otimes n} \ominus X_\Lambda(n) = \textrm{span}\{e_\alpha : \textrm{$\alpha$ is a forbidden word of length $n$} \}.
\ees
Fact number two: $X$ is the maximal standard subproduct system with prescribed fibers $X(1), \ldots, X(k+1)$ if and only if for every $n > k+1$,
\bes
X(n) = \bigcap_{i+j=n} X(i) \otimes X(j),
\ees
or in other words, if and only if
\begin{align*}
E^{\otimes n} \ominus X(n) &= \bigvee_{i+j=n} \left(E^{\otimes n} \ominus (X(i) \otimes X(j))\right) \\
&= \bigvee_{i+j=n} \left(E^{\otimes i}\otimes (E^{\otimes j} \ominus X(j)) + (E^{\otimes i} \ominus X(i)) \otimes E^{\otimes j}\right).
\end{align*}
Fact number three: $\Lambda$ is a $k$-step SFT if and only if for every $n > k+1$,
\begin{align*}
& \{\textrm{forbidden words of length $n$}\} = \\
& \bigcup_{i+j=n}\left(\cI^i \cdot \{\textrm{forbidden words of length $j$}\} \cup \{\textrm{forbidden words of length $i$}\}\cdot \cI^j \right).
\end{align*}
These facts assemble together to complete the proof.
\end{proof}

Not every subproduct system is isomorphic to one that comes from a subshift. Indeed, in the symmetric subproduct system $SSP$ (see Example \ref{expl:symm}) for any basis $\{e_i\}_{i \in \cI}$ of $X(1)$, the product $e_i \otimes e_j$ for $i \neq j$ is never in $X(2)$, and thus the images $f_i$ and $f_j$ of $e_i$ and $e_j$ in any isomorphic subproduct system $X$ can never be such that $f_i \otimes f_j$ is mapped isometrically to $U^X_{1,1}(f_i \otimes f_j)$. Thus if $SSP$ is isomorphic to $X_\Lambda$ for some subshift $\Lambda$, then $\Lambda$ must be the subshift containing only constant sequences. But such $X_\Lambda$ is clearly not isomorphic to $SSP$.

As another example, the subproduct system $X(0) = \mb{C}$, $X(1) = \mb{C}^2$, and $X(n) = 0$ for $n > 1$, cannot be of the form $X_\Lambda$ for any $\Lambda \subseteq \cI^{\mb{Z}}$.

\subsection{The representation theory of the C$^*$-algebra associated with a subshift of finite type}

Let $\Lambda$ be a fixed subshift in $\cI^{\mb{Z}}$ (with $\cI = \{1,2, \ldots, d\}$), and let $X = X_\Lambda$ be the associated subproduct system. We will denote the $X$-shift by $S$ (instead of $S^X$) to make some formulas more readable. Let $Z_i$ be the image of $S_i$ in the quotient $\cO_\Lambda$. We define for $i \in \cI$, $k \in \mb{N}$ the sets
\bes
E_i^k = \{\alpha \in \Lambda^k : i \alpha \in \Lambda^* \}.
\ees

\begin{lemma}\label{lem:E_i^k}
If $\Lambda$ is a $k$-step SFT, then for all $i\in\cI$,
\bes
\{\gamma \in \Lambda^* : |\gamma| \geq k,  i \gamma \in \Lambda^* \} = \{\alpha \beta \in \Lambda^*   : \alpha \in E_i^k, \beta \in \Lambda^* \}.
\ees
\end{lemma}
\begin{proof}
Assume that $\gamma \in \Lambda^*$ is such that $|\gamma| \geq k$ and $i \gamma \in \Lambda^*$. Defining $\alpha = \gamma_1 \cdots \gamma_k$ and $\beta = \gamma_{k+1}\cdots \gamma_{k+l}$, we have that $\gamma = \alpha \beta$ where $\alpha \in E_i^k$ and $\beta \in \Lambda^*$.

Conversely, if $\gamma = \alpha \beta \in \Lambda^*$ where $\alpha \in E_i^k$ and $\beta \in \Lambda^*$, then $i \gamma$ must be in $\Lambda^*$. Indeed, if not, then $i \gamma$ must contain a forbidden word. But $\gamma \in \Lambda^*$, thus the forbidden word must be in $i \alpha$ (since $\Lambda$ is a $k$-step SFT). But that is impossible because $\alpha \in E_i^k$.
\end{proof}

\begin{lemma}\label{lem:E=T}
If $\Lambda$ is a $k$-step SFT then for all $i,j \in \cI$, $i \neq j$,
\bes
S_i^*S_j = 0 ,
\ees
and
\be\label{eq:SXimodK}
S_i^*S_i = \sum_{\alpha \in E_i^k} \underline{S}^\alpha \underline{S}^{\alpha*}  \,\, \mod \cK_X.
\ee
Consequently, $\cE_X = \cT_X$.
\end{lemma}
%

\begin{proof}
Since the $S_i$ are partial isometries with orthogonal ranges, we have $S_i^* S_j = 0$ for all $i \neq j$. Since $\cK_X \subseteq \cE_X \subseteq \cT_X$ (Proposition \ref{prop:KinE}), $\cE_X = \cT_X$ will be established once we prove (\ref{eq:SXimodK}).

$S_i^*S_i$ is the projection onto the initial space of $S_i$. Call this space $G$. We have
\bes
G = \textrm{span}\{e_\alpha : \alpha \in \Lambda^* \textrm{ such that } i \alpha \in \Lambda^* \}.
\ees
The space
\bes
G' = \textrm{span}\{e_\alpha : \alpha \in \Lambda^* \textrm{ such that } i \alpha \in \Lambda^* \textrm{ and } |\alpha| \geq k\}
\ees
has finite codimension in $G$. But by Lemma \ref{lem:E_i^k},
\bes
G' = \{e_{\alpha \beta} : \alpha \beta \in \Lambda^*, \alpha \in E_i^k\},
\ees
that is, $G'$ is spanned by $e_\gamma$ where $\gamma$ runs through all legal words beginning with some $\alpha \in E_i^k$. Thus, $G'$ is the range of the projection $\sum_{\alpha \in E_i^k} \underline{S}^\alpha \underline{S}^{\alpha^*}$. Since $G'$ has finite codimension in $G$, we have (\ref{eq:SXimodK}).
\end{proof}

\begin{proposition}
For every subshift $\Lambda$, the $d$-tuple $\underline{Z} = (Z_1, \ldots, Z_d)$ satisfies the following relations:
\be\label{eq:Z1}
p(\underline{Z}) = 0 \,\, , \textrm{ for all } p \in I^X ,
\ee
\be\label{eq:Z2}
Z_i^* Z_j = 0 \,\,  , \textrm{ for all } i,j \in \cI \, , i \neq j ,
\ee
and
\be\label{eq:Z3}
\sum_{i=1}^d Z_i Z_i^* = 1.
\ee
In particular, $Z_i$ is a partial isometry for all $i \in \cI$.
If $\Lambda$ is a $k$-step SFT, the $\underline{Z}$ also satisfies
\be\label{eq:Z4}
Z_i^*Z_i = \sum_{\alpha \in E_i^k} \underline{Z}^\alpha \underline{Z}^{\alpha *} \,\, , \textrm{ for all } i \in \cI.
\ee
\end{proposition}

\begin{proof}
The quotient map $\cT_X \rightarrow \cO_\Lambda$ is a $*$-homomorphism, so (\ref{eq:Z1}) follows from Theorem \ref{thm:repZ(I)}. (\ref{eq:Z2}) and (\ref{eq:Z4}) follow from the previous lemma, and (\ref{eq:Z3}) follows from equation (\ref{eq:PC}).
\end{proof}

\begin{theorem}\label{thm:rep_subshift}
Let $\Lambda$ be a $k$-step SFT. Every unital representation $\pi : \cO_\Lambda \rightarrow B(H)$ is determined by a row-contraction $\underline{T} = (T_1, \ldots, T_d)$ satisfying relations (\ref{eq:Z1})-(\ref{eq:Z4}) such that $\pi(Z_i) = T_i$ for all $i \in \cI$. Conversely, every row contraction in $B(H)^d$ satisfying the relations (\ref{eq:Z1})-(\ref{eq:Z4}) gives rise to a unital representation $\pi : \cO_\Lambda \rightarrow B(H)$ such $\pi(Z_i) = T_i$ for all $i \in \cI$.
\end{theorem}
\begin{proof}
It is the second assertion that is non-trivial, and we will try to convince that it is true. By Theorem \ref{thm:CP1}, there is unital completely positive map
\bes
\Psi : \cE_X \rightarrow B(H)
\ees
sending $\underline{S}^\alpha \underline{S}^{\beta*}$ to $\underline{T}^\alpha \underline{T}^{\beta*}$. Since enough of the rank one operators on $\mathfrak{F}_X$ arise as $\underline{S}^\alpha (I - \sum_{i = 1}^d S_i S_i^*) \underline{S}^{\beta*}$ (see equation (\ref{eq:PC})), and because $\underline{T}$ satisfies (\ref{eq:Z3}), we must have that $\Psi(K) = 0$ for every $K \in \cK(\mathfrak{F}_X)$. By Lemma \ref{lem:E=T}, $\cE_X = \cT_X$, and it follows that $\Psi$ induces a positive and unital (hence contractive) mapping
\bes
\pi : \cO_\Lambda \rightarrow B(H)
\ees
that sends $\underline{Z}^\alpha \underline{Z}^{\beta*}$ to $\underline{T}^\alpha \underline{T}^{\beta*}$. Roughly speaking: $\pi$ must be multiplicative because $\underline{Z}$ and $\underline{T}$ satisfy the same relations. In more detail: every product $(\underline{Z}^\alpha \underline{Z}^{\beta*})(\underline{Z}^{\alpha'} \underline{Z}^{\beta'*})$ may be written, using the relations (\ref{eq:Z1})-(\ref{eq:Z4}) as some sum $\sum_{\gamma,\delta} \underline{Z}^\gamma \underline{Z}^{\delta*}$. The mapping $\pi$ then takes this sum to $\sum_{\gamma,\delta} \underline{T}^\gamma \underline{T}^{\delta*}$, and this can be rewritten (using the same relations) as \bes
(\underline{T}^\alpha \underline{T}^{\beta*})(\underline{T}^{\alpha'} \underline{T}^{\beta'*}) =
\pi(\underline{Z}^\alpha \underline{Z}^{\beta*})\pi(\underline{Z}^{\alpha'} \underline{Z}^{\beta'*}).
\ees
This shows that
\bes
\pi\left((\underline{Z}^\alpha \underline{Z}^{\beta*})(\underline{Z}^{\alpha'} \underline{Z}^{\beta'*})\right) = \pi(\underline{Z}^\alpha \underline{Z}^{\beta*})\pi(\underline{Z}^{\alpha'} \underline{Z}^{\beta'*}),
\ees
and since the elements of the form $\underline{Z}^\alpha \underline{Z}^{\beta*}$ span $\cO_\Lambda$, and since $\pi$ is a positive linear map, it follows that $\pi$ is in fact a $*$-representation.
\end{proof}

%
%
%
%
%


\bibliographystyle{amsplain}

\end{document}